\documentclass[english,12pt]{elsarticle}
\usepackage[english]{babel}
\usepackage{geometry}
\usepackage{hyperref}
\usepackage{amsthm}
\usepackage{amsmath}
\usepackage{amssymb}
\usepackage{graphicx}
\usepackage{thmtools, thm-restate}
\usepackage{cleveref}
\makeatletter
\theoremstyle{plain}
\newtheorem{theorem}{Theorem}[section]
\newtheorem{definition}{Definition}[section]
\newtheorem{lemma}{Lemma}[section]
\newtheorem{remark}{Remark}[section]
\newtheorem{proposition}{Proposition}[section]
\newtheorem{corollary}{Corollary}[section]

\usepackage{babel}
\usepackage{tikz}
\usepackage{tikz-cd}
\numberwithin{equation}{section}
\usepackage{caption}
\usepackage{comment}

\usepackage{accents}
\newcommand{\dbtilde}[1]{\accentset{\approx}{#1}}

\newcommand{\id}{\mathrm{id}}

\usepackage{etoolbox}
\apptocmd{\sloppy}{\hbadness 10000\relax}{}{}

\let\today\relax
\makeatletter
\def\ps@pprintTitle{%
	\let\@oddhead\@empty
	\let\@evenhead\@empty
	\def\@oddfoot{\footnotesize\itshape
		\hfill\today}%
	\let\@evenfoot\@oddfoot
}
\makeatother

\usepackage{dynkin-diagrams}

\begin{document}

\begin{frontmatter}

	\title{Classification of Hopf superalgebra structures on Drinfeld super Yangians}

	\author[1]{Alexander Mazurenko}
	\ead{mazurencoal@gmail.com}

	\author[2]{Vladimir A. Stukopin}
	\ead{stukopin@mail.ru}

	\address[1]{MCCME (Moscow Center for Continuous Mathematical Education)}
	\address[2]{MIPT (Moscow Institute of Physics and Technology) \\ SMI of VSC RAS (South Mathematical Institute of Vladikavkaz Scientific Center of Russian Academy of Sciences) \\ MCCME (Moscow Center for Continuous Mathematical Education)}

	\begin{abstract}
		We construct a minimalistic presentation of Drinfeld super Yangians in the case of special linear superalgebra associated with an arbitrary Dynkin diagram. This gives us a possibility to introduce Hopf superalgebra structure on Drinfeld super Yangians. Using complete Weyl group we classify Drinfeld super Yangians endowed with mentioned Hopf superalgebra structures. Also it is constructed an isomorphism between completions of Drinfeld super Yangians and quantum loop superalgebras.
	\end{abstract}


\end{frontmatter}

\section*{Acknowledgements}

This work is funded by Russian Science Foundation, scientific grant 21-11-00283. 

\vspace{1cm}

\section{Introduction}

Let $\mathfrak{g}(E,\Pi,p)$ be the special linear superalgebra defined by its simple root system $\Pi$, which lies in finite dimensional superspace $E$ with scalar product $(\cdot,\cdot)$ and parity function $p$. We will denote by $D$ a set which elements parametrize Dynkin diagrams. We recall that Drinfeld super Yangian $Y_{\hbar}(\mathfrak{g}(E,\Pi,p))$ and quantum loop superalgebra $U_q(L\mathfrak{g}(E,\Pi,p))$ are deformations of the universal enveloping superalgebras of the current superalgebra $\mathfrak{g}[t]$ and loop superalgebra $\mathfrak{g}[z,z^{-1}]$ respectively. We suppose that parameters of deformations $q$ and $\hbar$ are related by $q = e^{\frac{\hbar}{2}}$ and $q$ is not a root of unity. We consider Drinfeld super Yangian $Y_{\hbar}(\mathfrak{g}(E,\Pi,p))$ and quantum loop superalgeba $U_q(L\mathfrak{g}(E,\Pi,p))$ as associative superalgebras. It is well-known that Drinfeld super Yangian and quantum loop superalgebra are both Hopf superalgebras and our main task is to study relations between these Hopf superalgebras. We begin with studing in detail a minimalistic presentation of Drinfeld super Yangian (see \cite{GNW18}) associated with an arbitrary Dynkin diagram. We use this presentation in order to introduce a Hopf superalgebra structure on Drinfeld super Yangian by analogy with \cite{GNW18}. Our actions are motivated by desire to reduce calculations required to verificate that the introduced by us comultiplication is indeed a homomorphism of superalgebras. Using the introduced comultiplication we easily define the antipode. We also provide a relation between the comultiplication with Lie bisuperalgebra structure on special Lie superalgebra. Thus the coproduct on Drinfeld super Yangian is the deformation of the standard co-supercommutator of the initial Lie superbialgebra. The main difference with \cite{GNW18} is that we generalize results to super case and work with an arbitrary Dynkin diagram which makes the proofs more complicated. It is important to mention that Drinfeld super Yangian and quantum loop superalgebra also have so-called Drinfeld coproduct which was introduced by Drinfeld and was study by S. Khoroshkin and V. Tolstoy (\cite{KT96}), and also D. Hernandez (see \cite{H05}, \cite{H07} and also \cite{GNW18}). This coproduct is given in terms of infinite series though they can be regularized through deformation (see \cite{H07}). In order to study this comultiplication it is used the language of meromorphic categories introduced by D. Kazhdan and Y. Soibelman (see \cite{KS95} and \cite{GNW18} for more general case of affine Yangians). We are going to consider a relation between Drinfeld and standard comultiplications in the next paper. We also plan to describe structures of Hopf superalgebras on the quantum loop superalgebra and their connection with the structures of Hopf superalgebras on the Drinfeld super Yangian.

Next we construct an isomorphism of associative superalgebras $$\widehat{\Phi} : \widehat{U_{q} (L\mathfrak{g}(E,\Pi,p))} \to \widehat{Y_{\hbar}(\mathfrak{g}(E,\Pi,p))}$$ between appropriate completions of quantum loop superalgebra $U_{q} (L\mathfrak{g}(E,\Pi,p))$ and Drinfeld super Yangian $Y_{\hbar}(\mathfrak{g}(E,\Pi,p))$. We act by analogy with the paper \cite{S18} though we provide in some cases absolutely new proofs of important results. The proof of all these result is given in great details. 

Finally we prove when Drinfeld super Yangians $Y_{\hbar}(\mathfrak{g}(E,\Pi,p))$ and $Y_{\hbar}(\mathfrak{g}(E,\Pi_1,p))$ are isomorphic for any simple root systems $\Pi$ and $\Pi_1$. In order to do this we construct special equations based on the comultiplication properties which tell us how Dynkin diagrams of isomorphic Hopf superalgerbas look like. We also use a complete Weyl group in order to give a geometric intuition of the obtained result. Using the isomorphism $\widehat{\Phi}$ we endow the completion of quantum loop superalgebra with the structure of Hopf superalgerba and translate the obtained result about classification of Hopf superalgebra stuctures on it. We are going to discuss this issue in detail in the next paper.

\vspace{1cm}

\section{Preliminaries}

\label{sec:prelim}

In this paper we use the following notation. Let  $\mathbb{N}$, $\mathbb{N}_{0}$, $2 \mathbb{N}_{0}$, $\mathbb{Z}$, $\mathbb{Z}^{*}$, $\mathbb{Q}$ and $\mathbb{C}$ denote the sets of natural numbers, natural numbers with zero, even numbers in $\mathbb{N}_{0}$, integers, integers without zero, rational numbers and complex numbers, respectively. Let $\Bbbk$ be an algebraically closed field of characteristic zero. Denote by $M_{n,m}$ $(n,m \in \mathbb{N} )$ a matrix ring of $n \times m$ matrices. We also use the Iverson bracket defined by $ [P] = \begin{cases} 1 \text{ if } P \text{ is true;} \\ 0 \text{ otherwise}, \end{cases}$ where $P$ is a statement that can be true or false.

\subsection{Lie Superalgebra $A(m,n)$}

\label{subs:LS}

For an associative superalgebra $A$ and its two homogeneous elements $x, y$ and an element $v\in \Bbbk$ we define
\begin{eqnarray}
	&[x,y] := xy - (-1)^{|x||y|} y x, \quad \{x,y\} := xy + (-1)^{|x||y|} yx, \nonumber \quad \\
	&[x,y]_{v} := xy - (-1)^{|x||y|} v y x, \quad
\end{eqnarray}
where $|x|$ denotes the $\mathbb{Z}_2$-grading of $x$ (that is, $x \in A_{|x|} $).

A Lie superalgebra is a superalgerba $\mathfrak{g} = \mathfrak{g}_{\bar{0}} \oplus \mathfrak{g}_{\bar{1}}$ with the bilinear bracket (the super Lie bracket) $[\cdot, \cdot]: \mathfrak{g} \times \mathfrak{g} \to \mathfrak{g}$ which satisfies the following axioms: for any homogeneous $x, y, z \in \mathfrak{g}$
\begin{equation}
	[x, y] = - (-1)^{|x| |y|} [y, x],
\end{equation}
\begin{equation}
	\label{eq:sje}
	(-1)^{|x||z|} [x,[y,z]] + (-1)^{|x||y|} [y,[z,x]] + (-1)^{|z||y|}[z,[x,y]] = 0.
\end{equation}

Given two superspaces $A = A_{\bar{0}} \oplus A_{\bar{1}}$ and $B = B_{\bar{0}} \oplus B_{\bar{1}}$, their tensor product $A \otimes B$ is also a superspace with $(A \otimes B)_{\bar{0}} = A_{\bar{0}} \otimes B_{\bar{0}} \oplus A_{\bar{1}} \otimes B_{\bar{1}}$ and $(A \otimes B)_{\bar{1}} = A_{\bar{0}} \otimes B_{\bar{1}} \oplus A_{\bar{1}} \otimes B_{\bar{0}}$. Furthermore, if $A$ and $B$ are superalgebras, then $A \otimes B$ is made into a superalgebra, the graded tensor product of the superalgebras $A$ and $B$, via the following multiplication:
\begin{equation}
	(x \otimes y) (x^{'} \otimes y^{'}) = (-1)^{|y| |x^{'}|} (x x^{'}) \otimes (y y^{'}) \text{ for any } x \in A_{|x|}, x^{'} \in A_{|x^{'}|}, y \in B_{|y|}, y^{'} \in B_{|y^{'}|}.
\end{equation}
We will use only graded tensor products of superalgebras throughout this paper.

We use a set $D$ to label Dynkin diagrams that will appear further. Fix $m, n \in \mathbb{Z}_{\ge 1}$ and $d \in D$. From now on we set $I = \{1,2,...,m+n-1\}$. Consider a superspace $V^d = V_{\bar{0}}^d \oplus V_{\bar{1}}^d$ with a $\Bbbk$-basis $v_1,...,v_{m+n}$ such that each $v_i$ is either even ($v_{i} \in V_{\bar{0}}^d$) or odd ($v_{i} \in V_{\bar{1}}^d $). We set $n_{+} := \text{dim}(V_{\bar{0}}^d)$, $n_{-} := \text{dim}(V_{\bar{1}}^d)$, $m+n=n_{+} + n_{-} = \text{dim}(V^d)$. For $i \in I$, define $\bar{i} = \begin{cases} \bar{0}, & \mbox{if } v_{i} \in V_{\bar{0}}^d \\ \bar{1}, & \mbox{if } v_{i} \in V_{\bar{1}}^d \end{cases}$. Consider a free $\mathbb{Z}$-module $P_d := \bigoplus_{i=1}^{m+n} \mathbb{Z} \epsilon_{i,d}$ with the bilinear form determined by $(\epsilon_{i,d}, \epsilon_{j,d}) = \delta_{i j} (-1)^{\bar{i}}$ (we set $(-1)^{\bar{0}}:=1, (-1)^{\bar{1}}:=-1$). Let $\Delta_{d} = \{ \epsilon_{j,d} - \epsilon_{i,d} \; | \; i, j \in I, i \ne j \} \subset P_d$ be the root system. Then $\Delta_{d} = \Delta_{d}^{+} \cup \Delta_{d}^{-}$. Here $\Delta_{d}^{+} := \{ \epsilon_{j,d} - \epsilon_{i,d} \}_{1 \le j < i \le m+n}$ is the set of positive roots and $\Delta_{d}^{-} := - \Delta_{d}^{+}$ is the set of negative roots. Let $ \Pi_d = \{ \alpha_{i,d} := \epsilon_{i,d} - \epsilon_{i+1,d} \}_{i \in I} \subset \Delta_d$ be the set of simple roots. Set $|\alpha_{i,d}|:=\overline{i}+\overline{i+1} \in \mathbb{Z}_{2} $ for $i \in I$. Then $\Delta_{\bar{0}, d} := \{ \alpha \; | \; |\alpha| = \bar{0} \}$ is the set of even roots and $\Delta_{\bar{1},d} := \{ \alpha \; | \; |\alpha| = \bar{1} \}$ is the set of odd roots. Finally, let $C_{d} = (c_{ij}^d)_{i,j \in I}$ be the associated Cartan matrix, that is, $c_{ij}^d := (\alpha_{i,d}, \alpha_{j,d})$. Denote by $D_d$ the Dynkin diagram associated with $\Pi_d$ \cite{K77}. Yamane gave defining relations of a Lie superalgebra of type $A(m-1,n-1)$ ($m,n \ge 1$) associated with Dynkin diagram $D_d$ (for more details see \cite{Y91}, \cite{Y94}). Denote such Lie superalgebras by $\mathfrak{g}_{d} := \mathfrak{sl}(V^{d})$ for $d \in D$. In order to give more explicit description we give the well-known

\begin{proposition}
	\label{pr:supliealgdef}
	Consider a superspace $V^{d}$ ($d \in D$) and the Lie superalgebra $\mathfrak{g}_d$ with associated Cartan matrix $C_d = (c_{ij}^{d})_{i,j \in I}$. Then $\mathfrak{g}_d$ is generated by $h_{i,d}$, $e_{i,d}$ and $f_{i,d}$ for $i \in I$ ($|h_{i,d}| = \bar{0}$ and $|e_{i,d}|=|f_{i,d}|=|\alpha_{i,d}|$), where the generators satisfy the relations
	\begin{equation}
	[h_{i,d}, h_{j,d}] = 0, \; [h_{i,d}, e_{j,d}] = c_{ij}^d e_{j,d}, \; [h_{i,d}, f_{j,d}] = - c_{ij}^{d} f_{j,d}, \; [e_{i,d}, f_{j,d}] = \delta_{ij} h_{i,d}
	\end{equation}
	and the "super classical Serre-type" relations
	\begin{equation}
	\label{eq:scstr1}
	[e_{i,d}, e_{i,d}] = [f_{i,d}, f_{i,d}] = 0, \; \text{if} \; |\alpha_{i,d}|=\bar{1},
	\end{equation}
	\begin{equation}
	\label{eq:scstr2}
	(ad_{e_{i,d}})^{1+|c_{ij}^{d}|} e_{j,d} = (ad_{f_{i,d}})^{1+|c_{ij}^{d}|} f_{j,d} = 0, \; \text{if} \; i \ne j, \; \text{and} \; |\alpha_{i,d}|=\bar{0}
	\end{equation}
	plus nonstandart super Serre relations for type A:
	\begin{equation}
	\label{eq:scstr3}
	[ [ [e_{j-1,d}, e_{j,d}], e_{j+1,d} ] , e_{j,d} ] = [ [ [f_{j-1,d}, f_{j,d}], f_{j+1,d} ] , f_{j,d} ] = 0, \; \text{if} \; j-1, j, j + 1 \in I \; \text{and} \; |\alpha_{j,d}|=\bar{1},
	\end{equation}
	where for $x \in \mathfrak{g}_d$ the linear mapping $ad_{x}: \mathfrak{g}_d \to \mathfrak{g}_d$ called adjoint action is defined by $ad_{x}(y) = [x,y]$ for all $y \in \mathfrak{g}_d$.
\end{proposition}

\label{rm:lsr}
\begin{remark}
	\begin{enumerate}
		\item Relations \eqref{eq:scstr1} are also valid for $|\alpha_{i,d}|=\bar{0}$ by definition.
		\item Relations \eqref{eq:scstr2} are also valid for $|\alpha_{i,d}|=\bar{1}$, but in that case, they already follow from \cref{eq:sje,eq:scstr1}.
		\item Relations \eqref{eq:scstr3} are also valid for $|\alpha_{j,d}|=\bar{0}$, but in that case, they already follow from \cref{eq:sje,eq:scstr1,eq:scstr2}.
	\end{enumerate}
\end{remark}

Recall that the classical Lie superalgebra $A(m-1,n-1)$ ($m,n \ge 1$) coincides with $\mathfrak{sl}(m|n)$ for $m \ne n$, and with the quotient $\mathfrak{sl}(m|n) / (E)$ for $m = n$, where $E = \sum_{i=1}^{2n} E_{ii}$ is the central element, where $E_{ij} \in M_{2n, 2n}( \Bbbk)$ denotes matrix with $1$ at $(i,j)$-position and zeros elsewhere ($i,j \in \{1,2,...,2n\}$). In our notations $\mathfrak{sl}(m|n)$ is actually $\mathfrak{g}_d = \mathfrak{sl}(V^{d})$ for some $d \in D$ with the associated Dynkin diagram $D_{d}$ (the distinguished Dynkin diagram) which can be described in the following way
\begin{equation}
	D_{d} = ( \epsilon_{1,d} - \epsilon_{2,d}, ... , \epsilon_{i,d} - \epsilon_{i+1,d}, ... , \epsilon_{m+n-1,d} - \epsilon_{m+n,d} ),
	\label{eq:dddslmn}
\end{equation}
where $|\epsilon_{i,d}| = \bar{0}$ for all $i \in \{1,...,m\}$ and $|\epsilon_{i,d}| = \bar{1}$ for all $i \in \{m+1,...,m+n\}$.

Denote by $\mathfrak{n}^{+}_d$ (resp. $\mathfrak{n}^{-}_d$) and $\mathfrak{h}_d$ the subalgebra of $\mathfrak{g}_d$ generated by $e_{1,d}$, $...$, $e_{m+n-1,d}$ (resp. $f_{1,d}$, $...$, $f_{m+n-1,d}$) and $h_{1,d}$, $...$, $h_{m+n-1,d}$. Then define by $\mathfrak{b}^{+}_d = \mathfrak{h}_d \oplus \mathfrak{n}^{+}_d$ (resp. $\mathfrak{b}^{-}_d = \mathfrak{h}_d \oplus \mathfrak{n}^{-}_d$) the positive Borel subalgebra (resp. the negative Borel subalgebra) of $\mathfrak{g}_d$.

It is well-known that $P_{d} \otimes_{\mathbb{R}} \Bbbk \cong \mathfrak{h}_{d}^{*}$ in the case of $\mathfrak{g}_d$ ($d \in D$). Also $c_{ij}^{d} = \alpha_{j,d}(h_{i,d})$ for all $i, j \in I$. For any root $\alpha \in \Delta_{d}^{\pm}$, we have
\[ \mathfrak{g}_{ \alpha } = \{ x \in \mathfrak{g}_d \; | \; [h,x] = \pm \alpha(h) x \}. \]
Recall the definition of the root space decomposition of $\mathfrak{g}_d$:
\begin{equation}
	\mathfrak{g}_d = ( \bigoplus_{ \alpha \in \Delta_{d}^{-} } \mathfrak{g}_{ \alpha } ) \oplus \mathfrak{h}_{d} \oplus ( \bigoplus_{ \alpha \in \Delta_{d}^{+} } \mathfrak{g}_{ \alpha } ).
	\label{eq:rsd}
\end{equation}

Consider the following total order $"\le"$ on $\Delta^{+}_d$:
\begin{equation}
	\label{eq:orderRootquant}
	\alpha_{j,d} + \alpha_{j+1,d} + ... + \alpha_{i,d} \le \alpha_{j^{'},d} + \alpha_{j^{'}+1,d} + ... + \alpha_{i^{'},d} \mbox{ iff } j < j^{'} \mbox{ or } j = j^{'}, i \le i^{'}.
\end{equation}
For every $\beta \in \Delta^{+}_d$, we choose a decomposition $\beta = \alpha_{i_1,d}+...+\alpha_{i_p,d}$ such that $[...[e_{\alpha_{i_1,d}}, e_{\alpha_{i_2,d}}],...,e_{\alpha_{i_p,d}}]$ is a non-zero root vector $e_{\beta}$ of $\mathfrak{g}_d$ (here, $e_{\alpha_{i,d}} = e_{i,d}$ denotes the standart Chevalley generator of $\mathfrak{g}_d$). Then, we define
\[ e_{\beta} := [...[ [ e_{i_1,d}, e_{i_2,d} ], e_{i_3,d} ], ... , e_{i_p,d} ], \]
\[ f_{\beta} := [...[ [ f_{i_p,d}, f_{i_{p-1},d} ], f_{i_{p-2},d} ], ... , f_{i_1,d} ]. \]
One more time in particular, $e_{\alpha_{i,d}} = e_{i,d}$ and $f_{\alpha_{i,d}} = f_{i,d}$. It is well known that $ \mathfrak{g}_{ \alpha } = \langle x_{\alpha} \rangle $ as a vector super space for any $\alpha \in \Delta_{d}^{\pm}$.

It can be shown \cite{K77} that there exists a unique (up to a constant factor) non-degenerate supersymmetric invariant bilinear form $\langle . , . \rangle : \mathfrak{g}_d \times \mathfrak{g}_d \to \Bbbk $, i. e. for $x,y,z \in \mathfrak{g}_d$
\begin{enumerate}
	\item the form induces an isomorphism $\mathfrak{g}_d \cong (\mathfrak{g}_d)^{*}$;
	\item $ \langle x , y \rangle = (-1)^{|x||y|} \langle y , x \rangle $;
	\item $ \langle [x,y] , z \rangle + (-1)^{|x||y|} \langle y , [x,z] \rangle = 0$.
\end{enumerate}
We choose a bilinear form on $\mathfrak{g}_d$ in such a way that for all $\alpha, \beta \in \Delta_d^{+}$ we have $\langle e_{\alpha}, f_{\beta} \rangle = \delta_{\alpha,\beta}$. Fix a basis $\{ h_{(k)} \}_{k \in I}$ of $\mathfrak{h}_d$. Then we can construct an orthonormal basis $\{ h^{(k)} \}_{k \in I}$ with respect to the $\langle . , . \rangle$, i. e. $\langle h^{(i)}, h^{(j)} \rangle = \delta_{ij}$ ($i,j \in I$). Recall that
\[ \sum_{k \in I} \langle h^{(k)} , h_{i,d} \rangle h^{(k)} = h_{i,d}, \]
for all $i \in I$. Then the Casimir operator $\Omega_d \in \mathfrak{g}_d \otimes \mathfrak{g}_d$ is given by the formula
\begin{equation}
	\label{eq:CE}
	\Omega_d = \sum_{k \in I} h^{(k)} \otimes h^{(k)} + \sum_{ \alpha \in \Delta^{+}_d} (-1)^{|\alpha|} e_{\alpha} \otimes f_{\alpha} +  \sum_{ \alpha \in \Delta^{+}_d} f_{\alpha} \otimes e_{\alpha}.
\end{equation}
The element $\Omega_d$ is the even, invariant, supersymmetric element, i. e.
\begin{enumerate}
	\item $|\Omega_d| = 0$;
	\item $[ x \otimes 1 + 1 \otimes x, \Omega_d ] = 0$ for all $x \in \mathfrak{g}_d$;
	\item $\Omega_d = \tau_{\mathfrak{g}_d,\mathfrak{g}_d}(\Omega_d)$, where $\tau_{\mathfrak{g}_d,\mathfrak{g}_d}$ is defined by \eqref{eq:taudef}.
\end{enumerate}

Define $ht( \sum_{i \in I} n_{i} \alpha_{i,d} ) = \sum_{i \in I} n_{i}$ for all $( n_{i} )_{i \in I} \in \mathbb{Z}^{|I|}$.

\vspace{1cm}

\subsection{The Lie superbialgebra structures on polynomial current Lie superalgebra and loop superalgebra}

A Lie superbialgebra $(\mathfrak{g}, [\cdot, \cdot], \delta)$ (see \cite{GZB91}) is a Lie superalgebra $(\mathfrak{g}, [\cdot, \cdot])$ with a linear map $\delta: \mathfrak{g} \to \mathfrak{g} \otimes \mathfrak{g}$ called co-supercommutator that preserves the $\mathbb{Z}_{2}$-grading and satisfies the following conditions:
\begin{equation}
	\delta + \tau_{\mathfrak{g}, \mathfrak{g}} \circ \delta = 0,
\end{equation}
\begin{equation}
	\label{eq:bisuli1}
	( \delta \otimes id_{\mathfrak{g}} ) \circ \delta - (id_{\mathfrak{g}} \otimes \delta) \circ \delta = (id_{\mathfrak{g}} \otimes \tau_{\mathfrak{g}, \mathfrak{g}}) \circ (\delta \otimes id_{\mathfrak{g}}) \circ \delta,
\end{equation}
\begin{equation}
	\label{eq:bisuli2}
	\delta([x,y]) = ( ad_{x} \otimes id_{\mathfrak{g}} + id_{\mathfrak{g}} \otimes ad_{x} ) \delta(y) - ( ad_{y} \otimes id_{\mathfrak{g}} + id_{\mathfrak{g}} \otimes ad_{y} ) \delta(x),
\end{equation}
where $x, y \in \mathfrak{g}$, $id_{\mathfrak{g}}$ is the identity map on $\mathfrak{g}$, $ad_{x} y = [x,y]$ is the adjoint and $\tau_{V, W}: V \otimes W \to W \otimes V$ is the linear function given by
\begin{equation}
	\label{eq:taudef}
	\tau_{V, W}(v \otimes w) = (-1)^{|v| |w|} w \otimes v
\end{equation}
for homogeneous $v \in V$ and $w \in W$, where $V$ and $W$ are super vector spaces.

We also need the following well-known result \cite{A93}
\begin{proposition}
	Let $(\mathfrak{g}, [\cdot, \cdot])$ be a Lie superalgebra and let $r \in \mathfrak{g} \otimes \mathfrak{g}$ be even and $r + \tau_{\mathfrak{g}, \mathfrak{g}} \circ r = 0$. Then $(\mathfrak{g}, [\cdot, \cdot], \delta r)$ is a Lie superbialgebra if and only if the element
	\[ CYB(r) := [r_{12}, r_{13}] + [r_{12},r_{23}] + [r_{13},r_{23}] \in \mathfrak{g}^{\otimes 3} \]
	is $\mathfrak{g}$-invariant, where for every $a \in \mathfrak{g}$, $\delta r(a) = [a \otimes 1 + 1 \otimes a, r]$.
\end{proposition}

Consider the polynomial current Lie superalgebra $\mathfrak{g}_d[u] := \mathfrak{g}_d \otimes_{\Bbbk} \Bbbk[u]$, where $|u| = 0$. Thus
\[ (\mathfrak{g}_d[u])_{\bar{0}} = (\mathfrak{g}_d)_{\bar{0}} \otimes_{\Bbbk} \Bbbk[u] \text{ and } (\mathfrak{g}_d[u])_{\bar{1}} = (\mathfrak{g}_d)_{\bar{1}} \otimes_{\Bbbk} \Bbbk[u]. \]
The Lie superalgebra structure on $\mathfrak{g}_d[u]$ is defined by
\begin{equation}
[ a \otimes u^m, b \otimes u^n ] := [a,b] \otimes u^{m+n},
\end{equation}
where $a, b \in \mathfrak{g}_d$ and $m,n \in \mathbb{N}_{0}$. $\mathfrak{g}_d[u]$ is a Lie superbialgebra (see \cite{AB21} for the non-super case). The co-supercommutator $\delta_d: \mathfrak{g}_d[z] \to \mathfrak{g}_d[u] \otimes \mathfrak{g}_d[v] \cong \mathfrak{g}_d^{\otimes 2}[u, v]$ is given by the following formula
\begin{equation}
	\delta_d(f(z)) = [ f(u) \otimes 1 + 1 \otimes f(v), \frac{\Omega_d}{u-v} ].
\end{equation}
Thus we have
\begin{equation}
	\delta_d( a z^{m} ) = [ a u^{m} \otimes 1 + 1 \otimes a v^{m}, \frac{\Omega_d}{u-v} ] =
\end{equation}
\[ [ a \otimes 1, \Omega_d ] \frac{u^m-v^m}{u-v} = [ m > 0 ] [ a \otimes 1, \Omega_d ] \sum_{k=0}^{m-1} u^{k} v^{m-1-k}, \]
where $a \in \mathfrak{g}_d$, $\Omega_d$ is defined by \eqref{eq:CE} and $m \in \mathbb{N}_{0}$. Note that
\begin{equation}
	\label{eq:cosuperbr1}
	\delta_d(a) = 0;
\end{equation}
\begin{equation}
	\label{eq:cosuperbr2}
	\delta_d(a u) = [ a \otimes 1, \Omega_d ] =
\end{equation}
\[ \sum_{i \in I }  [ a, h_{i,d}^{'} ] \otimes h_{i,d}^{'} + \sum_{\alpha \in \Delta^{+}_d} (-1)^{|\alpha|} [a, e_{\alpha,d}] \otimes f_{\alpha,d} + \sum_{\alpha \in \Delta^{+}_d} [a, f_{\alpha, d}] \otimes e_{\alpha, d}, \]
where $a \in \mathfrak{g}_d$.

The loop superalgebra is defined as $L(\mathfrak{g}_d) := \mathfrak{g}_d \otimes \Bbbk[u,u^{-1}]$, where $|u|=|u^{-1}|=0$. Thus
\[ L(\mathfrak{g}_d)_{\bar{0}} = (\mathfrak{g}_d)_{\bar{0}} \otimes \Bbbk[u,u^{-1}] \text{ and } L(\mathfrak{g}_d)_{\bar{1}} = (\mathfrak{g}_d)_{\bar{1}} \otimes \Bbbk[u,u^{-1}]. \]
The Lie superalgebra structure is defined by
\begin{equation}
[ a \otimes u^{m}, b \otimes u^{n} ] := [a,b] \otimes u^{m+n},
\end{equation}
where $a, b \in \mathfrak{g}_d$ and $m,n \in \mathbb{Z}$. $L(\mathfrak{g}_d)$ is a Lie superbialgebra (see \cite{AB21} for the non-super case). The co-supercommutator $\phi_d: L(\mathfrak{g}_d) \to L(\mathfrak{g}_d)^{\otimes 2} \cong \mathfrak{g}_d^{\otimes 2}[u, u^{-1}, v, v^{-1}]$ is given by the following formula \cite{KPSST}
\begin{equation}
	\phi_d(f(z)) = [ f(u) \otimes 1 + 1 \otimes f(v), r_{0}(u, v) ],
\end{equation}
where
\[r_{0}(u, v) = \frac{1}{2} ( \frac{u+v}{u-v} \Omega_d + w_d ), \]
\[ w_d = \sum_{\alpha \in \Delta_{d}^{+}} f_{\alpha,d} \otimes e_{\alpha,d} - \sum_{\alpha \in \Delta_{d}^{+}} (-1)^{|\alpha|} e_{\alpha,d} \otimes f_{\alpha,d}. \]
Thus we have
\begin{equation}
\phi_d(a z^m) = [ a u^m \otimes 1 + 1 \otimes a v^m, r_{0}(u, v) ] =
\end{equation}
\[ \frac{1}{2} ( [a \otimes 1, \Omega_d] \frac{(u+v)(u^m-v^m)}{u-v} +  [ a u^m \otimes 1 + 1 \otimes a v^m, w_d] ). \]

\vspace{1cm}

\subsection{Weyl Group and Complete Weyl Group}

The material is in some way standart. We reformulate definitions and some important results about Weyl groups in a constructive way needed for our purposes. For classical definitions see for example \cite{K90} and \cite{CH09}. Recall notations introduced in Subsection \ref{subs:LS}.

Fix an arbitrary $(n_{+}, n_{-}) \in \mathbb{N}^{2}$. Denote by $S_{X}$ the symmetric group on the set $X$ and by $S_m$ the symmetric group of degree $m$ ($m \in \mathbb{N}$). Let $P(n_{+}, n_{-})$ be the family of root systems $P_{d}$ ($d \in D$) such that for associated superspaces $V^{d}$ we have $n_{+} := \text{dim}(V_{\bar{0}}^d)$, $n_{-} := \text{dim}(V_{\bar{1}}^d)$. Note that $|P(n_{+}, n_{-})| < \infty$. Also for any $P \in P(n_{+}, n_{-})$ we can construct the associated superspace $V$ and the Lie superalgebra $\mathfrak{sl}(V)$ in the unique way. Denote by $P_{st}(n_{+}, n_{-}) \in P(n_{+}, n_{-})$ the element which has the basis $ \mathcal{B}_{st} = \{ \epsilon_{i} \}_{\{1,2,..., n_{+}+n_{-}\}}$ such that $|\epsilon_{i}| = \bar{0}$ for all $i \in I_{\bar{0}} = \{1,2,..., n_{+}\}$ and $|\epsilon_{i}| = \bar{1}$ for all $i \in I_{\bar{1}} = \{n_{+}+1,n_{+}+2,..., n_{+}+n_{-}\}$. All elements in $P(n_{+}, n_{-})$ are isomorphic as free $\mathbb{Z}$-modules endowed with $\mathbb{Z}_{2}$-grading.

Also we need the following classical result (see for proof \cite{KT08}). Recall that $S_{m}$ is generated by the adjacent transpositions $ \{ \sigma_{i} \}_{ i \in \{1,2,...,m-1\} } $. Define the sets
\[ \Sigma_{k} = \{ \id_{S_{m}}, \sigma_{k}, \sigma_{k} \circ \sigma_{k-1} , ... , \sigma_{k} \circ \sigma_{k-1} \circ ... \circ \sigma_{2} \circ \sigma_{1} \} \]
for all $k \in \{1,2,...,m-1\}$. Then
\begin{lemma}
	\emph{\cite{KT08}}
	For any permutation $w \in S_{m}$ there is a unique element
	\[ (w_{1}, w_{2}, ... , w_{m-1}) \in \Sigma_{1} \times \Sigma_{2} \times ... \times \Sigma_{m-1}, \]
	such that $w = w_{1} \circ w_{2} \circ ... \circ w_{m-1}$.
	\label{lm:udecperm}
\end{lemma}

\subsubsection{Weyl Group}
\label{sub:WGdef}

The Weyl group $W(n_{+}, n_{-})$ of type $(n_{+}, n_{-})$ is the group $S_{I_{\bar{0}}} \times S_{I_{\bar{1}}}$ with the group faithful representation $\rho : S_{I_{\bar{0}}} \times S_{I_{\bar{1}}} \to Aut(P_{st}(n_{+}, n_{-}))$ (described below). Here $Aut(P_{st}(n_{+}, n_{-}))$ is notation for an automorphism group. For an arbitrary element $w = (u, v) \in S_{I_{\bar{0}}} \times S_{I_{\bar{1}}}$ we define
\[ w(\epsilon_{i}) := \rho( w ) (\epsilon_{i}) = \epsilon_{ u(i)}, \]
if $|\epsilon_{i}| = \bar{0}$ and
\[ w(\epsilon_{i}) := \rho( w ) (\epsilon_{i}) = \epsilon_{ v(i)}, \]
if $|\epsilon_{i}| = \bar{1}$, where $ \mathcal{B}_{st} = \{ \epsilon_{i} \; | \; i \in \{1,2,...,n_{+}+n_{-}\} \}$ is the early constructed basis of $P_{st}(n_{+}, n_{-})$. Note that $|\rho( w ) (\epsilon_{i})| = |\epsilon_{i}|$ by definition.

We show that this definition is correct. It is obvious that $ \rho( \id_{S_{I_{\bar{0}}}} \times \id_{S_{I_{\bar{1}}}} ) = \id_{P_{st}(n_{+}, n_{-})}$. Next it is easy to see that for any $ w_{1} = (u, v)$ and $w_{2} = (y, p) \in W(n_{+}, n_{-})$ we have
\[ ( \rho(w_{2}) \circ \rho( w_{1} )  ) (\epsilon_{i}) = \begin{cases}
	\rho(w_{2}) ( \epsilon_{ u(i) } ), \text{ if } |\epsilon_{i}| = \bar{0} \\
	\rho(w_{2}) (\epsilon_{ v(i) } ), \text{ if } |\epsilon_{i}| = \bar{1}
\end{cases} = \begin{cases}
	\epsilon_{  y ( u(i) ) }, \text{ if } |\epsilon_{i}| = \bar{0} \\
	\epsilon_{ q ( v(i) ) }, \text{ if } |\epsilon_{i}| = \bar{1}
\end{cases} = \]
\[ \begin{cases}
	\epsilon_{ (y \circ u )(i) }, \text{ if } |\epsilon_{i}| = \bar{0} \\
	\epsilon_{ (q \circ v )(i) }, \text{ if } |\epsilon_{i}| = \bar{1}
\end{cases} = \rho (w_{2} \circ w_{1})(\epsilon_{i}), \]
where $\mathcal{B}_{st} = \{ \epsilon_{i} \; | \; i \in \{1,2,...,n_{+}+n_{-}\} \}$ is the basis of $P_{st}(n_{+}, n_{-})$. Also we can endow $W(n_{+}, n_{-})$ with the $\mathbb{Z}_{2}$-grading assuming that $|w| = \bar{0}$ for all $w \in W(n_{+}, n_{-})$. Note that $\rho$ is the $\mathbb{Z}_{2}$-graded function and, moreover, it is the even function. The result follows.

Let $\{ \sigma_{i}^{n_{+}} \}_{ i \in I_{\bar{0}} \setminus \{n_{+}\} }$ and $\{ \sigma_{i}^{n_{-}} \}_{ i \in  I_{\bar{1}} \setminus \{n_{+}+n_{-}\} }$ be the adjacent transpositions which generate $S_{I_{\bar{0}}}$ and $S_{I_{\bar{1}}}$ respectively. Thus $W(n_{+}, n_{-})$ is generated by elements
\[ \{ \sigma_{i}^{n_{+}} \times \id_{S_{I_{\bar{1}}}}, \id_{S_{I_{\bar{0}}}} \times \sigma_{j}^{n_{-}} \}_{ i \in I_{\bar{0}} \setminus \{n_{+}\}}^{ j  \in  I_{\bar{1}} \setminus \{n_{+}+n_{-}\} }, \]
 which we call simple reflections. We introduce notations $\sigma_{i}^{n_{+}} := \sigma_{i}^{n_{+}} \times \id_{S_{I_{\bar{1}}}}$ and $\sigma_{j}^{n_{-}} := \id_{S_{I_{\bar{0}}}} \times \sigma_{j}^{n_{-}}$ for convenience. Thus generators satisfy the following relations
\[ ( \sigma_{i}^{n_{\pm}} )^2 = \id_{S_{I_{\bar{0}}}} \times \id_{S_{I_{\bar{1}}}} \]
for all $i \in I_{\bar{0}} \setminus \{n_{+}\}$ in the case of $+$ and $i \in  I_{\bar{1}} \setminus \{n_{+}+n_{-}\}$ in the case of $-$ respectively;
\[ \sigma_{i}^{n_{\pm}} \sigma_{j}^{n_{\pm}} = \sigma_{j}^{n_{\pm}} \sigma_{i}^{n_{\pm}} \text{ for } |i-j| > 1 \]
for $i, j \in I_{\bar{0}} \setminus \{n_{+}\}$ in the case of $+$ and $i, j \in  I_{\bar{1}} \setminus \{n_{+}+n_{-}\}$ in the case of $-$ respectively;
\[ ( \sigma_{i}^{n_{\pm}} \sigma_{i+1}^{n_{\pm}} )^{3} = \id_{S_{I_{\bar{0}}}} \times \id_{S_{I_{\bar{1}}}} \]
for $i \in I_{\bar{0}} \setminus \{n_{+}-1, n_{+}\}$ in the case of $+$ and $i \in  I_{\bar{1}} \setminus \{n_{+}+n_{-}-1,n_{+}+n_{-}\}$ in the case of $-$ respectively;
\[ \sigma_{i}^{n_{+}} \sigma_{j}^{n_{-}} = \sigma_{j}^{n_{-}} \sigma_{i}^{n_{+}} \]
for all $i \in I_{\bar{0}} \setminus \{ n_{+}\}$ and $j \in  I_{\bar{1}} \setminus \{n_{+}+n_{-}\}$.

\begin{remark}
	\item We associate with each $P_{d} \in P(n_{+}, n_{-})$ ($d\in D$) its own $W(d)$ Weyl group $(S_{n_{+}+n_{-}}, \rho_{d} : S_{n_{+}+n_{-}} \to Aut_{\Bbbk}(P_{d}))$ in such a way that representations $\rho_{d}$ and $\rho$ are isomorphic. It is possible as there exists a $\mathbb{Z}_{2}$-graded even isomorphism $\zeta_{d} : P_{d} \to P_{st}(n_{+}, n_{-})$ that sends basis elements $\{ \epsilon_{i,d} \}_{i \in \{1,...,n_{+}+n_{-}\}}$ of $P_{d}$ to elements of basis $\mathcal{B}_{st}$. Then set $\rho_{d} (?) := \zeta_{d}^{-1}  \circ \rho (?) \circ \zeta_{d}$. Thus if start from an arbitrary element of $P(n_{+}, n_{-})$ and define its Weyl group we get an equivalent construction to that described in the case of $P_{st}(n_{+}, n_{-})$. We select to describe $W_{c}(n_{+}, n_{-})$ in the case of $P_{st}(n_{+}, n_{-})$ for convenience.
	\item The first three relations are called the Coxeter relations.
\end{remark}

\vspace{1cm}

\subsubsection{Complete Weyl Group}
\label{sb:WGD}

The complete Weyl group $W_{c}(n_{+}, n_{-})$ of type $(n_{+}, n_{-})$ is the group $S_{n_{+}+n_{-}}$ with the group faithful representation $\rho : S_{n_{+}+n_{-}} \to Aut_{\Bbbk}(P_{st}(n_{+}, n_{-}))$. Here $Aut_{\Bbbk}(P_{st}(n_{+}, n_{-}))$ is notation for an automorphism group. For an arbitrary element $w \in S_{n_{+}+n_{-}}$ we define
\[ w(\epsilon_{i}) := \rho( w ) (\epsilon_{i}) = \epsilon_{ w(i)}, \]
where $\mathcal{B}_{st} = \{ \epsilon_{i} \; | \; i \in \{1,2,...,n_{+}+n_{-}\} \}$ is the early constructed basis of $P_{st}(n_{+}, n_{-})$.

We show that this definition is correct. It is obvious that $ \rho( \id_{S_{n_{+}+n_{-}}} ) = \id_{P_{st}(n_{+}, n_{-})}$. Next it is easy to see that for any $ w_{1}, w_{2} \in W_{c}(n_{+}, n_{-})$ we have
\[ ( \rho(w_{2}) \circ \rho( w_{1} )  ) (\epsilon_{i}) = \rho(w_{2}) ( \epsilon_{ w_{1}(i) } ) = \epsilon_{ w_{2} (w_{1}(i)) } = \]
\[ \epsilon_{ (w_{2} \circ w_{1})(i) } =  \rho (w_{2} \circ w_{1})(\epsilon_{i}), \]
where $\mathcal{B}_{st} = \{ \epsilon_{i} \; | \; i \in \{1,2,...,n_{+}+n_{-}\} \}$ is the basis of $P_{st}(n_{+}, n_{-})$.

It is clear that $W_{c}(n_{+}, n_{-})$ is generated by the adjacent transpositions $ \{ \sigma_{i} \}_{ i \in I } $ which we call simple reflections. Now we can endow $W_{c}(n_{+}, n_{-})$ with the $\mathbb{Z}_{2}$-grading. We do it in the iterative way. Set $|\id_{S_{n_{+}+n_{-}}}| = \bar{0}$ and $|\sigma_{i}| = |\epsilon_{i}| + |\epsilon_{i+1}|$ for all $i \in I$. Next assign for any permutation of the form $w = \sigma_{k} \circ \sigma_{k-1} \circ ... \circ \sigma_{2} \circ \sigma_{1}$ ($k \in I$) the degree $|w| = \sum_{j=1}^{k} |\sigma_{j}|$. Recall Lemma \ref{lm:udecperm}. Consider for any $w \in W_{c}(n_{+}, n_{-})$ a unique decomposition 
\[ w = w_1 \circ w_2 \circ ... \circ w_{n_{+}+n_{-}-1}, \]
where $ (w_{1}, w_{2}, ... , w_{n_{+}+n_{-}-1}) \in \Sigma_{1} \times \Sigma_{2} \times ... \times \Sigma_{n_{+}+n_{-}-1}$. Then set $|w| = \sum_{j=1}^{n_{+}+n_{-}-1} |w_{j}|$. It is clear that $\mathbb{Z}_{2}$-grading is well-defined.

Note that $\rho$ is the $\mathbb{Z}_{2}$-graded function and, moreover, it is the even function. The result follows.

The generators of $W_{c}(n_{+}, n_{-})$ satisfy the Coxeter relations 
\[ ( \sigma_{i} )^2 = \id_{S_{n_{+}+n_{-}}} \]
for all $i \in I$;
\[ \sigma_{i} \sigma_{j} = \sigma_{j} \sigma_{i} \text{ for } |i-j| > 1 \]
for $i, j \in I$;
\[ ( \sigma_{i} \sigma_{i+1} )^{3} = \id_{S_{n_{+}+n_{-}}} \]
for $i \in \{ 1, 2, ..., n_{+} + n_{-} - 2 \}$.

\begin{remark}
	\item We associate with each $P_{d} \in P(n_{+}, n_{-})$ ($d\in D$) its own $W_{c}(d)$ compelete Weyl group $(S_{n_{+}+n_{-}}, \rho_{d} : S_{n_{+}+n_{-}} \to Aut_{\Bbbk}(P_{d}))$ in such a way that representations $\rho_{d}$ and $\rho$ are isomorphic. It is possible as there exists a $\mathbb{Z}_{2}$-graded even isomorphism $\zeta_{d} : P_{d} \to P_{st}(n_{+}, n_{-})$ that sends basis elements $\{ \epsilon_{i,d} \}_{i \in \{1,...,n_{+}+n_{-}\}}$ of $P_{d}$ to elements of basis $\mathcal{B}_{st}$. Then set $\rho_{d} (?) := \zeta_{d}^{-1}  \circ \rho (?) \circ \zeta_{d}$. Thus if start from an arbitrary element of $P(n_{+}, n_{-})$ and define its complete Weyl group we get an equivalent construction to that described in the case of $P_{st}(n_{+}, n_{-})$. We select to describe $W_{c}(n_{+}, n_{-})$ in the case of $P_{st}(n_{+}, n_{-})$ for convenience.
	\item Note that the Weyl group $W(n_{+}, n_{-})$ is isomorphic to a subgroup of the complete Weyl group $W_{c}(n_{+}, n_{-})$. Indeed consider the subgroup $A$ generated by all elements $w \in W_{c}(n_{+}, n_{-})$ which permute elements of the basis $\mathcal{B}_{st}$ in such a way that $w(I_{\bar{0}}) = I_{\bar{0}}$ and $w(I_{\bar{1}}) = I_{\bar{1}}$. It is obvious that there is a group isomorphism $W(n_{+}, n_{-}) \cong A$.
	\label{rm:cwgfadd}
\end{remark}

\vspace{1cm}

\subsection{Completion}

\label{sub:compl}

Let $(J, \le)$ be a directed poset. Consider $\{ \mathfrak{m}_{j} \}_{j \in J}$ a descending sequence of filtration of $\mathbb{Z}_2$-graded ideals of a superalgebra $A$, i. e. if $i \le j$ in $J$, then $\mathfrak{m}_{i} \supseteq \mathfrak{m}_{j}$. We define the completion of $A$ with respect to the filtration $\{ \mathfrak{m}_{j} \}_{j \in J}$ to be
\begin{equation}
	\widehat{A}_{ \{ \mathfrak{m}_{j} \}_{j \in J}} =  \mathop{\lim_{\longleftarrow}}_{ j \in J } A / \mathfrak{m}_{j},
\end{equation}
where the compatability maps are the projections $ \phi_{i}^{j} ( \bar{x_{j}} ) = \phi_{i}^{j}(x + \mathfrak{m}_{j}) = x + \mathfrak{m}_{i} = \bar{x_{i}}$  ($x \in A$ and $i \le j$). Recall that the inverse limit of the inverse system $( (A / \mathfrak{m}_{j})_{j \in J} , ( \phi_{i}^{j} )_{i \le j \in J} )$ is
\begin{equation}
	\widehat{A}_{ \{ \mathfrak{m}_{j} \}_{j \in J}} = \{ ( \bar{x_{j}} ) \in \prod_{j \in J} A / \mathfrak{m}_{j} \; | \; \phi_{i}^{j}(\bar{x_{j}}) = \bar{x_{i}} \text{ for all } i \le j  \text{ in } J  \}.
\end{equation}
Recall that $\widehat{A}_{ \{ \mathfrak{m}_{j} \}_{j \in J}}$ has the natural inverse limit topology.

Let $A$ be an $\mathbb{N}_{0}$-graded superalgebra. Then $A = \bigoplus_{n \in \mathbb{N}_{0}} A_{n}$, where $A_{0} = \Bbbk$, and we have a descending filtration given by $A \supseteq F_{1} \supseteq F_{2} \supseteq ... \supseteq F_{i} \supseteq ... $, where $F_{i} = \bigoplus_{n \ge i } A_{n}$ ($i \in \mathbb{N}$). Thus we are able to consider the completion of $A$ with respect to the filtration $\{F_{i}\}_{i \in \mathbb{N}}$. 

Let $B$ be a subalgebra in $A$, i. e. $B = \bigoplus_{n \in \mathbb{N}_{0}} B_{n}$. Then we can construct in the same manner the completion $\widehat{B}$ with respect to the filtration $\{F_{i} = \bigoplus_{n \ge i } B_{n} \}_{i \in \mathbb{N}}$. More generally, we set $deg(v) = 1$ and $|v| = \bar{0}$ for a fomal variable $v$ and consider $B[v]$ which is an $\mathbb{N}_{0}$-graded superalgebra which extends the $\mathbb{N}_{0}$-grading on $B$. Then we can construct in the same manner the completion $\widehat{B[v]}$ with respect to the $\mathbb{N}_{0}$-grading on $B[v]$.

Let $I$ be a proper $\mathbb{Z}_2$-graded ideal of a superalgebra $A$. Consider the descending filtration given by $A \supseteq I \supseteq I^{2} \supseteq ... \supseteq I^{i} \supseteq ... $, where $i \in \mathbb{N}$. Thus we are able to consider the completion of $A$ with respect to the filtrarion $\{I^{i}\}_{i \in \mathbb{N}}$. The natural inverse limit topology in this case is called $I$-adic topology.

Let $A$ be an  $\mathbb{N}_{0}$-graded superalgebra and $B$ be any subset in $A$. Then by $B_{+}$ we denote a set of elements of positive degree in $B$ with respect to the grading introduced on $A$.

\vspace{1cm}

\section{Quantum structures}

In this Section it is given description of quatum superalgebras associated with Lie superalgebras from the Section \ref{sec:prelim}. We remind definitions of these quantum structures for arbitrary Dynkin diagrams and construct a minimalistic presentation for certain superalgebras.

\subsection{Quantized universal enveloping superalgebra}

\label{sb:ques}

Throughout this paper, $q$ and $\hbar$ are formal variables related by $q = exp(\frac{\hbar}{2}) \in \Bbbk[[\hbar]]$. For all $k, n \in \mathbb{N}_{0}$ and $k \le n $ we set
\[ [0]_{q}!:=1, \; [n]_q = \frac{q^n-q^{-n}}{q-q^{-1}}, \]
\[ [n]_q! = [n]_q [n-1]_q ... [1]_q, \; \begin{bmatrix}n\\k\end{bmatrix}_q = \frac{[n]_q!}{[k]_q! [n-k]_q!}. \]

A quantized universal enveloping (QUE) superalgebra $A_{\hbar}$ is a topologically free Hopf superalgebra over $\Bbbk[[\hbar]]$ such that $A_{\hbar} / \hbar A_{\hbar}$ is isomorphic as a Hopf superalgebra to universal enveloping superalgebra $U(\mathfrak{g})$ for some Lie superalgebra $\mathfrak{g}$. We use the following result proved in the non-super case in \cite{D87} and in the super case in \cite{A93}.

\begin{proposition}
	\label{pr:QUEliebi}
	Let $A_{\hbar}$ be a QUE superalgebra: $A_{\hbar} / \hbar A_{\hbar} \cong U(\mathfrak{g})$. Then the Lie superalgebra $\mathfrak{g}$ has a natural structure of a Lie superbialgebra defined by
	\begin{equation}
		\label{eq:squebialgsupcon}
		\delta(x) = \hbar^{-1} ( \Delta(\tilde{x}) - \Delta^{op}(\tilde{x}) ) \mod \hbar,
	\end{equation}
	where $x \in \mathfrak{g}$, $\tilde{x} \in A_{\hbar}$ is a preimage of $x$, $\Delta$ is a comultiplication in $A$ and $\Delta^{op} := \tau_{U(\mathfrak{g}), U(\mathfrak{g})} \circ \Delta$ (for the definition of $\tau_{U(\mathfrak{g}), U(\mathfrak{g})}$ see \eqref{eq:taudef}).
\end{proposition}

\begin{definition}
	\normalfont
	Let $A$ be a QUE superalgebra and let $(\mathfrak{g}, [\cdot, \cdot], \delta)$ be the Lie superbialgebra defined in Proposition \ref{pr:QUEliebi}. We say that $A$ is a quantization of the Lie superbialgebra $\mathfrak{g}$.
\end{definition}

Thus a quantization of a Lie superbialgebra $(\mathfrak{g}, [\cdot, \cdot], \delta)$ is called an associative unital Hopf superalgebra $A_{\hbar}$ over the ring of formal power series $\Bbbk[[\hbar]]$ such that the following conditions hold:
\begin{enumerate}
	\item $A_{\hbar}$ is isomorphic to $U(\mathfrak{g})[[\hbar]]$ as a superspace;
	\item $A_{\hbar} / \hbar A_{\hbar} \cong U(\mathfrak{g})$ as Hopf superalgebras;
	\item $\delta(x) = \hbar^{-1} ( \Delta(\tilde{x}) - \Delta^{op}(\tilde{x}) ) \mod \hbar$, where $x \in \mathfrak{g}$, $\tilde{x} \in A_{\hbar}$ is a preimage of $x$, $\Delta$ is a comultiplication in $A_{\hbar}$.
\end{enumerate}

We use results about formal power series proved in \cite{N69}. Let $A$ be a ring. Consider the ring of formal power series $A[[X]]$. Define for each $f \in 1+XA[[X]]$
\[ \log(f) := \sum_{n=1}^{\infty} \frac{(-1)^{n+1}}{n} (f-1)^{n} \in XA[[X]]. \]
Define for each $f \in XA[[X]]$
\[ exp(f) := \sum_{n=0}^{\infty} \frac{f^{n}}{n!} \in 1 + XA[[X]]. \]

\vspace{1cm}

\subsection{The Drinfeld super Yangian}
\label{sect:DSY}

Following \cite{T19a,T19b} (see also \cite{D88}, \cite{S94} and \cite{G07}), define the Drinfeld super Yangian of $\mathfrak{g}_d$ ($d \in D$), denoted by $Y_{\hbar}(\mathfrak{g}_d)$, to be the unital, associative $\Bbbk[[\hbar]]$-Hopf superalgebra generated by $\{ h_{i,r,d} , x_{i,r,d}^{\pm} \}^{r \ge 0}_{i \in I}$ with the $\mathbb{Z}_2$-grading $|h_{i,r,d}| = \bar{0}$, $|x_{i,r,d}^{\pm}| = |\alpha_{i,d}|$, and subject to the following defining relations:
\begin{equation}
	\label{eq:Cer}
	[h_{i,r,d}, h_{j,s,d}] = 0,
\end{equation}
\begin{equation}
	\label{eq:SYrc1}
	[h_{i,0,d}, x_{j,s,d}^{\pm}] = \pm c_{ij}^{d} x_{j,s,d}^{\pm}
\end{equation}
\begin{equation}
	\label{eq:SYrc3}
	[x_{i,r,d}^{+}, x_{j,s,d}^{-}] = \delta_{i,j} h_{i,r+s,d},
\end{equation}
\begin{equation}
	\label{eq:SYrc2}
	[ h_{i,r+1,d} , x_{j,s,d}^{\pm} ] - [h_{i,r,d} , x_{j,s+1,d}^{\pm}] = \pm \frac{c_{ij}^d \hbar}{2} \{h_{i,r,d}, x_{j,s,d}^{\pm}\} \text{ unless } i=j \text{ and } |\alpha_{i,d}|=\bar{1},
\end{equation}
\begin{equation}
	\label{eq:SYrc4}
	[x_{i,r+1,d}^{\pm}, x_{j,s,d}^{\pm}] - [x_{i,r,d}^{\pm}, x_{j,s+1,d}^{\pm}] = \pm \frac{c_{ij}^{d} \hbar}{2} \{ x_{i,r,d}^{\pm} , x_{j,s,d}^{\pm} \} \text{ unless } i=j \text{ and } |\alpha_{i,d}| = \bar{1},
\end{equation}
\begin{equation}
	\label{eq:SYrc5}
	[h_{i,r,d}, x_{i,s,d}^{\pm}] = 0 \text{ if } |\alpha_{i,d}| = \bar{1},
\end{equation}
\begin{equation}
	\label{eq:bssr}
	[x_{i,r,d}^{\pm}, x_{j,s,d}^{\pm}] = 0 \text{ if } c_{ij}^d = 0,
\end{equation}
as well as cubic super Serre relations
\begin{equation}
	\label{eq:cssr}
	[ x_{i,r,d}^{\pm} , [ x_{i,s,d}^{\pm} , x_{j,t,d}^{\pm} ] ] + [ x_{i,s,d}^{\pm} , [x_{i,r,d}^{\pm} , x_{j,t,d}^{\pm}] ] = 0 \text{ if } j=i \pm 1 \text{ and } |\alpha_{i,d}|=\bar{0},
\end{equation}
and quartic super Serre relations
\begin{equation}
	\label{eq:qssr}
	[ [ x_{j-1,r,d}^{\pm} , x_{j,0,d}^{\pm} ] , [ x_{j,0,d}^{\pm}, x_{j+1,s,d}^{\pm} ] ] = 0 \text{ if } |\alpha_{j,d}| = \bar{1}.
\end{equation}

\begin{remark}
	\begin{enumerate}
		\item Similar to Remark \ref{rm:lsr}, cubic super Serre relations \eqref{eq:cssr} also hold for all other parities, but in those cases, they already follow from \cref{eq:sje,eq:bssr}.
		\item Similar to Remark \ref{rm:lsr}, quartic super Serre relations \eqref{eq:qssr} also hold for all other parities, but in those cases, they already follow from \cref{eq:sje,eq:bssr,eq:cssr}.
		\item Generalizing the quartic super Serre relations \eqref{eq:qssr}, the following relations also hold \cite{T19b}:
		\[ [ [x_{j-1,r,d}^{\pm}, x_{j,k,d}^{\pm}], [x_{j,l,d}^{\pm}, x_{j+1,s,d}^{\pm}]] + [ [x_{j-1,r,d}^{\pm}, x_{j,l,d}^{\pm}], [x_{j,k,d}^{\pm},x_{j+1,s,d}^{\pm}]] = 0. \]
		This, in turn, using \eqref{eq:cssr} can be rewritten as
		\[ [[[x^{\pm}_{j-1,r,d}, x^{\pm}_{j,k,d}] , x^{\pm}_{j+1,s,d}], x^{\pm}_{j,l,d}] + [[[x^{\pm}_{j-1,r,d}, x^{\pm}_{j,l,d}] , x^{\pm}_{j+1,s,d}], x^{\pm}_{j,k,d}] = 0. \]
	\end{enumerate}
	\label{rm:yangrelrew}
\end{remark}

We note that the universal enveloping superalgebra $U(\mathfrak{g}_d)$ is naturally embedded in $Y_{\hbar}(\mathfrak{g}_d)$ as Hopf superalgebra, and the embedding is given by the formulas $h_{i,d} \to h_{i,0,d}$, $e_{i,d} \to x_{i,0,d}^{+}$, $f_{i,d} \to x_{i,0,d}^{-}$. We shall identify the universal enveloping superalgebra $U(\mathfrak{g}_d)$ with its image in the Drinfeld super Yangian.

Let $Y_{\hbar}^{0}(\mathfrak{g}_d)$, $Y_{\hbar}^{\pm}(\mathfrak{g}_d) \subset Y_{\hbar}(\mathfrak{g}_d)$ be the subalgebras generated by the elements $\{ h_{i,r} \}_{i \in I, r \in \mathbb{N}_{0}}$ (resp. $\{ x_{i,r}^{\pm} \}_{i \in I, r \in \mathbb{N}_{0}}$) and $Y_{\hbar}(\mathfrak{g}_d)^{\ge 0}$, $Y^{\le 0}_{\hbar}(\mathfrak{g}_d) \subset Y_{\hbar}(\mathfrak{g}_d)$ be the subalgebras generated by $Y^{0}_{\hbar}(\mathfrak{g}_d)$, $Y^{+}_{\hbar}(\mathfrak{g}_d)$ and $Y^{0}_{\hbar}(\mathfrak{g}_d)$, $Y^{-}_{\hbar}(\mathfrak{g}_d)$ respectively.

Denote by $Y^{0}_{i}(\mathfrak{g}_d) \subset Y^{0}_{\hbar}(\mathfrak{g}_d)$ the subalgebra generated by $\{ h_{ir} \}_{r \in \mathbb{N}_{0}}$ for a fixed $i \in I$. Let $ Y_{ \hbar}(\mathfrak{b}_{+}) $ and $ Y_{ \hbar}(\mathfrak{b}_{-}) $ be the subalgebras generated respectively by $ \{ x_{ir}^{+} , h_{ir} \}_{i \in I, r \in \mathbb{Z}} $ and $ \{ x_{ir}^{-} , h_{ir} \}_{i \in I, r \in \mathbb{Z}} $. Denote by $Y_{\hbar}(\mathfrak{sl}(2|1)^{i})$ the subalgebra generated by $\{ x_{ir}^{ \pm}, h_{ir} \}_{r \in \mathbb{Z}}$ for a fixed $i \in I$.

\vspace{1cm}

\subsection{A minimalistic presentation of $Y_{\hbar}(\mathfrak{g}_d)$}

\label{sb:mpsy}

In order to verify that the comultiplication given in Section \ref{sect:DSY} is a homomorphism of superalgebras we describe a more convenient form of superalgebras $Y_{\hbar}(\mathfrak{g}_d)$. Such work was done in \cite{S94} for $Y_{\hbar}(\mathfrak{g}_d)$ associated only with the distinguished Dynkin diagram (the result is stated without proof), and for non-super case in \cite{GNW18} and \cite{L93}. \footnote{from now on we use subscript $d \in D$ only when it is needed in order to simplify notations}

From defining relations we can see that $Y_{\hbar}(\mathfrak{g})$ is generated by elements $\{ h_{ir} , x_{i0}^{\pm} \}^{r \in \{0,1\}}_{i \in I}$. We use equations \eqref{eq:SYrc1}, \eqref{eq:SYrc2} and \eqref{eq:SYrc3} to get recurrent formulas
\begin{equation}
	\label{eq:rec1}
	x_{i,r+1}^{\pm} = \pm (c_{ii^{'}})^{-1} [ h_{i^{'}1} - \frac{\hbar}{2} h_{i^{'}0}^2 , x_{ir}^{\pm} ],
\end{equation}
where $r \ge 0$; $i^{'}=i$, if $|\alpha_{i}|=\bar{0}$, and $|i-i^{'}|=1$, if $|\alpha_{i}|=\bar{1}$ ($i,i^{'} \in I$);
\begin{equation}
	\label{eq:rec2}
	h_{ir} = [ x_{ir}^{+}, x_{i0}^{-} ],
\end{equation}
where $r \ge 2$ and $i \in I$.

We introduce the auxiliary generators for $i \in I$ by setting
\begin{equation}
	\widetilde{h}_{i 1} \overset{\operatorname{def}}{=} h_{i1} - \frac{\hbar}{2} h_{i0}^2.
\end{equation}

From now we exclude cases $|I|=1$; $|I|=3$ and $|\alpha_{1}|=|\alpha_{3}|$. The main reason is that in these cases \eqref{eq:rec1} and Lemmas \ref{lm:SYrc34}, \ref{lm:fassrс}, \ref{lm:hii12} and \ref{lm:leqmp} are not valid. Moreover, we consider Dynkin diagrams that satisfy the following condition: every odd vertex has as a neighbor at least one even vertex. This is needed in order Lemma \ref{lm:cartelrelii} to be true. Note that the distinguished Dynkin diagrams satisfy this condition.

\begin{theorem}
	\label{th:mpy}
	Suppose that our constraints on $I$ and on Dynkin diagrams are satisfied. Then $Y_{\hbar}(\mathfrak{g})$ is isomorphic to the superalgebra generated by $\{ h_{ir} , x_{ir}^{\pm} \}^{r \in \{0,1\}}_{i \in I}$ with the $\mathbb{Z}_2$-grading $|h_{ir}| = \bar{0}$, $|x_{ir}^{\pm}| = |\alpha_{i}|$, subject only to the relations
	\begin{equation}
	\label{eq:mpy1}
	[h_{ir}, h_{js}] = 0 \; (0 \le r,s \le 1),
	\end{equation}
	\begin{equation}
		\label{eq:mpy2}
	[h_{i0}, x_{js}^{\pm}] = \pm (\alpha_{i},\alpha_{j}) x_{js}^{\pm} \; (s = 0,1),
	\end{equation}
	\begin{equation}
		\label{eq:mpy3}
		[x_{ir}^{+}, x_{js}^{-}] = \delta_{ij} h_{i,r+s} \; (0 \le r+s \le 1),
	\end{equation}
	\begin{equation}
		\label{eq:mpy4}
		[\widetilde{h}_{i 1}, x_{j0}^{\pm}] = \pm (\alpha_{i}, \alpha_{j}) x_{j1}^{\pm} \text{ unless } i=j \text{ and } |\alpha_{i}| = \bar{1},
	\end{equation}
	\begin{equation}
		\label{eq:mpy5}
		[x_{i1}^{\pm}, x_{j0}^{\pm}] - [x_{i0}^{\pm}, x_{j1}^{\pm}] = \pm \frac{c_{ij} \hbar}{2} \{ x_{i0}^{\pm} , x_{j0}^{\pm} \} \text{ unless } i=j \text{ and } |\alpha_{i}| = \bar{1},
	\end{equation}
	\begin{equation}
		\label{eq:mpy6}
		[h_{i1}, x_{is}^{\pm}] = 0 \text{ if } |\alpha_{i}| = \bar{1} \; (s=0,1),
	\end{equation}
	\begin{equation}
		\label{eq:mpy7}
		[x_{i0}^{\pm}, x_{j0}^{\pm}] = 0 \text{ if } c_{ij} = 0,
	\end{equation}
	\begin{equation}
		\label{eq:mpy8}
		[ x_{i0}^{\pm} , [ x_{i0}^{\pm} , x_{j0}^{\pm} ] ] = 0 \text{ if } j=i \pm 1 \text{ and } |\alpha_{i}|=\bar{0},
	\end{equation}
	\begin{equation}
		\label{eq:mpy9}
		[ [ x_{j-1,0}^{\pm} , x_{j0}^{\pm} ] , [ x_{j0}^{\pm}, x_{j+1,0}^{\pm} ] ] = 0 \text{ if } |\alpha_{j}| = \bar{1}.
	\end{equation}
\end{theorem}

In this superalgebra, we also define elements $x_{ir}^{\pm}$ ($r \ge 2$) and $h_{ir}$ ($r \ge 2$) for $i \in I$ using \eqref{eq:rec1} and \eqref{eq:rec2}.

\begin{remark}
	\label{rm:addrel}
	As noted in \cite{GNW18} in order to prove Theorem \ref{th:mpy} we must prove that equation
	\begin{equation}
		\label{rm:lsar}
		[ [\widetilde{h}_{i^{'}1},x_{i1}^{+}] , x_{i1}^{-} ] + [ x_{i1}^{+} , [ \widetilde{h}_{i^{'}1},x_{i1}^{-} ] ] = 0,
	\end{equation}
   where $i^{'}=i$, if $|\alpha_{i}|=\bar{0}$, and $|i-i^{'}|=1$, if $|\alpha_{i}|=\bar{1}$, can be deduced from relations \eqref{eq:mpy1} - \eqref{eq:mpy9}.

   Note that it follows from Lemmas \eqref{lm:l2} and \eqref{lm:cel02} that \eqref{rm:lsar} is equivalent to
    \begin{equation}
   	\label{eq:hia12}
    [h_{i^{'}1}, h_{i2}] = 0.
	\end{equation}
\end{remark}

The proof of Theorem \ref{th:mpy}  is splitted in several lemmas and propositions.

\begin{lemma}
	\label{lm:h0eq}
	The relation \eqref{eq:SYrc1} is satisfied for all $i,j \in I$ and $r \in \mathbb{N}_{0}$. Moreover, for the same parameters
	\begin{equation}
		\label{eq:h1eq}
		[\widetilde{h}_{i1} , x_{jr}^{\pm}] = \pm (\alpha_{i}, \alpha_{j}) x_{j,r+1}^{\pm}.
	\end{equation}
\end{lemma}
\begin{proof}
	We prove the first equality by induction on $r$. If $r=0$, then if follows from \eqref{eq:mpy2}. Note that from \eqref{eq:rec1}, \eqref{eq:mpy1} and induction hypothesis we have
	\[ [ h_{i0} , x_{j,r+1}^{\pm} ] = \pm (\alpha_{j^{'}}, \alpha_{j})^{-1} [ h_{i0} , [ \widetilde{h}_{j^{'} 1} , x_{jr}^{\pm} ] ] = \]
	\[ (\alpha_{i}, \alpha_{j}) (\alpha_{j^{'}}, \alpha_{j})^{-1} [ \widetilde{h}_{j^{'} 1} , x_{jr}^{\pm} ] = \pm (\alpha_{i}, \alpha_{j}) x_{j,r+1}^{\pm}. \]
	The general case follows then.

	We prove the second equality by induction on $r$. If $r=0$, then if follows from \eqref{eq:mpy4}. Note that from \eqref{eq:rec1}, \eqref{eq:mpy1} and induction hypothesis we have
	\[ [ \widetilde{h}_{i1} , x_{j,r+1}^{\pm} ] = \pm (\alpha_{j^{'}}, \alpha_{j})^{-1} [ \widetilde{h}_{i1} , [ \widetilde{h}_{j^{'} 1} , x_{jr}^{\pm} ] ] = \]
	\[ (\alpha_{i}, \alpha_{j}) (\alpha_{j^{'}}, \alpha_{j})^{-1} [ \widetilde{h}_{j^{'} 1} , x_{j,r+1}^{\pm} ] = \pm (\alpha_{i}, \alpha_{j}) x_{j,r+2}^{\pm}. \]
	The general case follows then.
\end{proof}

\begin{lemma}
	\label{lm:l2}
	The relation \eqref{eq:SYrc3} holds when $i=j$, $r +s \le 2$.
\end{lemma}
\begin{proof}
	The proof is analogous to that in \cite[Lemma 2.23]{GNW18}.
\end{proof}

\begin{lemma}
	\label{lm:sy10}
	The relation \eqref{eq:SYrc4} holds when $i=j$, $(r,s) =(1,0)$, i.e.,
	\[ [ x_{i2}^{\pm} , x_{i0}^{\pm} ] = \pm \frac{(\alpha_{i}, \alpha_{i}) \hbar}{2} \{ x_{i1}^{\pm} , x_{i0}^{\pm} \}, \]
	where $|\alpha_{i}| = \bar{0}$.
\end{lemma}
\begin{proof}
	The proof is analogous to that in \cite[Lemma 2.24]{GNW18}.
\end{proof}

\begin{lemma}
	\label{lm:hx20}
	The relation \eqref{eq:SYrc2} holds when $i=j$, $(r,s) =(1,0)$, i.e.,
	\[ [ h_{i2} , x_{i0}^{\pm} ] - [h_{i1} , x_{i1}^{\pm}] = \pm \frac{c_{ii} \hbar}{2} \{h_{i1}, x_{i0}^{\pm}\}, \]
	where $|\alpha_{i}| = \bar{0}$.
\end{lemma}
\begin{proof}
	The proof is analogous to that in \cite[Lemma 2.26]{GNW18}.
\end{proof}

\begin{lemma}
	\label{lm:SYrc34}
	Suppose that $i,j \in I$ and $i \ne j$. The relations \eqref{eq:SYrc3} and \eqref{eq:SYrc4} hold for any $r,s \in \mathbb{N}_{0}$.
\end{lemma}
\begin{proof}
	We prove \eqref{eq:SYrc4} by induction on $r$ and $s$. The initial case $r=s=0$ is our assumption \eqref{eq:mpy5}. Let $X^{\pm}(r,s)$ be the result of substracting the right-hand side of \eqref{eq:SYrc4} from the left-hand side. Note that if we apply $[\widetilde{h}_{m1} , \cdot]$ and $[\widetilde{h}_{n1} , \cdot]$ ($m,n \in I$) to $X^{\pm}(r,s)$ we get from \eqref{eq:h1eq}
	\[ 0 = (\alpha_{i}, \alpha_{m}) X^{\pm}(r+1,s) + (\alpha_{j}, \alpha_{m}) X^{\pm}(r,s+1), \]
	\[ 0 = (\alpha_{i}, \alpha_{n}) X^{\pm}(r+1,s) + (\alpha_{j}, \alpha_{n}) X^{\pm}(r,s+1). \]

	In order to determine when the determinant of the matrix $\bigl( \begin{smallmatrix}(\alpha_{i},\alpha_{m}) & (\alpha_{j},\alpha_{m})\\ (\alpha_{i},\alpha_{n}) & (\alpha_{j},\alpha_{n})\end{smallmatrix}\bigr)$ is nonzero depending on the grading of roots it is sufficient to consider the following Dynkin (sub)diagrams:
	\begin{enumerate}
	\item
	\begin{tikzpicture}[scale=2]
		\tikzset{/Dynkin diagram,root radius=.07cm}
		\dynkin[labels={\alpha_{i},\alpha_{j}}] A{xx}
	\end{tikzpicture}

	Then we select $m=i, n=j$ in order to get the nonzero determinant.

	\item
	\begin{tikzpicture}[scale=2]
		\tikzset{/Dynkin diagram,root radius=.07cm}
		\dynkin[labels={\alpha_{i},\alpha_{i^{'}},\alpha_{j}}] A{xxx}
	\end{tikzpicture}

	$|\alpha_i|=|\alpha_j|=\bar{0} \Rightarrow m=i, n=j$;

	$|\alpha_i|=\bar{1}, |\alpha_j|=\bar{0} \Rightarrow m=i^{'}, n=j$;

	$|\alpha_i|=\bar{0}, |\alpha_j|=\bar{1} \Rightarrow m=i, n=i^{'}$;

	if $|I|=3$ and $|\alpha_i|=|\alpha_j|=\bar{1}$, then we are not able to build an invertible matrix for our purposes. That's why we exclude this case from our consideration. If $|I| > 3$ then a Dynkin subdiagram may have the form

	\begin{tikzpicture}[scale=2]
		\tikzset{/Dynkin diagram,root radius=.07cm}
		\dynkin[labels={\alpha_{i},\alpha_{i^{'}},\alpha_{j},\alpha_{j^{'}}}] A{txtx}
	\end{tikzpicture}

	or

	\begin{tikzpicture}[scale=2]
		\tikzset{/Dynkin diagram,root radius=.07cm}
		\dynkin[labels={\alpha_{i^{'}},\alpha_{i},\alpha_{j^{'}},\alpha_{j}}] A{xtxt}
	\end{tikzpicture}

	In these cases we take $m=i^{'}, n=j^{'}$.

	\item
	\begin{tikzpicture}[scale=2]
		\tikzset{/Dynkin diagram,root radius=.07cm}
		\dynkin[labels={\alpha_{i},\alpha_{i^{'}},\alpha_{j^{'}},\alpha_{j}}] A{xxxx}
	\end{tikzpicture}

	$|\alpha_i|=|\alpha_j|=\bar{0} \Rightarrow m=i, n=j$;

	$|\alpha_i|=\bar{1}, |\alpha_j|=\bar{0} \Rightarrow m=i^{'}, n=j$;

	$|\alpha_i|=\bar{0}, |\alpha_j|=\bar{1} \Rightarrow m=i, n=j^{'}$;

	$|\alpha_i|=\bar{1}, |\alpha_j|=\bar{1} \Rightarrow m=i^{'}, n=j^{'}$.

	\end{enumerate}
	When the determinant is nonzero, we have $X^{\pm}(r+1,s)=X^{\pm}(r,s+1)=0$. The result follows by induction hypothesis.

	We prove \eqref{eq:SYrc3} by induction on $r$ and $s$. The initial case $r=s=0$ is our assumption \eqref{eq:mpy3}. Let $X(r,s)$ be the result of substracting the right-hand side of \eqref{eq:SYrc4} from the left-hand side. Note that if we apply $[\widetilde{h}_{m1} , \cdot]$ and $[\widetilde{h}_{n1} , \cdot]$ ($m,n \in I$) to $X(r,s)$ we get from \eqref{eq:h1eq} that
	\[ 0 = (\alpha_{i}, \alpha_{m}) X(r+1,s) + (-1) (\alpha_{j}, \alpha_{m}) X(r,s+1), \]
	\[ 0 = (\alpha_{i}, \alpha_{n}) X(r+1,s) + (-1) (\alpha_{j}, \alpha_{n}) X(r,s+1). \]

	It is obvious that in order to determine when the determinant of the matrix $\bigl( \begin{smallmatrix}(\alpha_{i},\alpha_{m}) & -(\alpha_{j},\alpha_{m})\\ (\alpha_{i},\alpha_{n}) & -(\alpha_{j},\alpha_{n})\end{smallmatrix}\bigr)$ is non-zero depending on the grading of roots we apply the same arguments as above. The result follows by induction hypothesis.
\end{proof}

\begin{lemma}
	Suppose that $i,j \in I$. Then \eqref{eq:SYrc2} holds for $(r,s)=(1,0)$.
\end{lemma}
\begin{proof}
	The proof is analogous to that in \cite[Lemma 2.29]{GNW18}.
\end{proof}

For each $i \in I$, we define $\widetilde{h}_{i2}$ by $\widetilde{h}_{i2} = h_{i2}- \hbar h_{i0} h_{i1} + \frac{\hbar^2}{3} h_{i0}^3$.

\begin{proposition}
	For any $i,j \in I$, the following identity holds:
	\[ [ \widetilde{h}_{i2} , x_{j0}^{\pm} ] = \pm(\alpha_{i}, \alpha_{j}) x_{j2}^{\pm} \pm \frac{\hbar^2}{12} (\alpha_{i}, \alpha_{j})^3 x_{j0}^{\pm}. \]
\end{proposition}
\begin{proof}
	If $i=j$ and $|\alpha_{i}|=\bar{1}$ then the proof follows from \eqref{eq:mpy2}, \eqref{eq:mpy6}, \eqref{eq:mpy7}, \eqref{eq:rec1} and Lemma \eqref{lm:l2} that
	\[ [h_{i2}, x_{i0}^{\pm}] = (-1) ( [x_{i1}^{+} , [x_{i0}^{\pm} , x_{i1}^{-}] ] + [x_{i1}^{-} , [x_{i0}^{\pm} , x_{i1}^{+}] ] ) = \]
	\[ (-1) [ x_{i1}^{\mp} , [x_{i0}^{\pm} , x_{i1}^{\pm}] ] = 0. \]
	Otherwise the proof is analogous to that in \cite[Lemma 2.31]{GNW18}.
\end{proof}

\begin{lemma}
	\label{lm:fassrс}
	Relations \eqref{eq:bssr} (unless $i=j$ and $|\alpha_{i}|=\bar{1}$) and \eqref{eq:cssr} hold for the following cases:
	\begin{enumerate}
		\item $(r,s)=(0, s)$ and $(r,s,t)=(0,0,s) \; (s \in \mathbb{N}_{0})$;
		\item $(r,s)=(1,s)$ and $(r,s,t)=(1,0,s) \; (s \in \mathbb{N}_{0})$,
		\item $(r,s)=(2, s)$ and $(r,s,t)=(2,0,s) \; (s \in \mathbb{N}_{0})$;
		\item $(1,1,s) \; (s \in \mathbb{N}_{0})$,
	\end{enumerate}
	respectively.
\end{lemma}
\begin{proof}
	The proof is analogous to that in \cite[Lemma 2.33]{GNW18}.	Note that in order to prove cases 1 and 2 we arrive to a system of homogeneous linear equations. Then we apply the same arguments as in cases 2 and 3 in Lemma \ref{lm:SYrc34} to determine when this system has a unique trivial solution.
\end{proof}

\begin{lemma}
	\label{lm:hxji}
	We have
	\[ [ h_{j1} , x_{i1}^{\pm} ] = \frac{(\alpha_i , \alpha_{j})}{(\alpha_{i^{'}}, \alpha_{i})} [ h_{i^{'}1} , x_{i1}^{\pm} ] \pm \frac{(\alpha_{i}, \alpha_{j}) \hbar}{2} ( \{ h_{j0} , x_{i1}^{\pm} \} - \{ h_{i^{'}0} , x_{i1}^{\pm} \}  ) \]
	for all $i,j \in I$, where  $i^{'}=i$, if $|\alpha_{i}|=\bar{0}$, and $|i-i^{'}|=1$, if $|i|= \bar{1} \; (i^{'} \in I)$.
\end{lemma}
\begin{proof}
	The proof is analogous to that in \cite[Lemma 2.34]{GNW18}.
\end{proof}

\begin{lemma}
	\label{lm:cel02}
	For all $i,j \in I$, we have for any $s \in \mathbb{N}_{0}$
	\[ [h_{i0}, h_{js}] = 0. \]
\end{lemma}
\begin{proof}
By Lemma \ref{lm:h0eq}
\[ [ h_{i0} , h_{js} ] = [ h_{i0} , [x_{js}^{+} , x_{j0}^{-}] ] = (\alpha_{i} , \alpha_{j}) ( h_{js} - h_{js} ) = 0. \]
\end{proof}

\begin{lemma}
	\label{lm:hii12}
	Let $i, j \in I$ be such that $|\alpha_{i}|=\bar{0}$, $(\alpha_{i}, \alpha_{j})= \pm 1$ and $(\alpha_{i}, \alpha_{i})= \mp 2$ or $|\alpha_{i}|=\bar{1}$. Then
	\[ [h_{i1}, h_{i2}] = [ h_{i1}, [x_{i1}^{+}, x_{i1}^{-}] ] =  0. \]
\end{lemma}
\begin{proof}
	The first equality follows from Lemma \ref{lm:l2}, so we prove the second equality.

	If $|\alpha_{i}| = \bar{1}$ then by \eqref{eq:mpy6} we have
	\[ [ h_{i1}, [x_{i1}^{+}, x_{i1}^{-}] ] = (-1) ( [ x_{i1}^{+} , [ x_{i1}^{-} , h_{i1} ] ] + (-1)^{|i|} [  x_{i1}^{-} , [ h_{i1} , x_{i1}^{+} ] ] ) = 0. \]

	Now suppose that $|\alpha_{i}|=\bar{0}$. By Lemma \ref{lm:fassrс} (4) we have
	\[ 0 = [ x_{i1}^{+} , [ x_{i1}^{+} , x_{j0}^{+} ] ]. \]
	By Lemmas \ref{lm:l2} and \ref{lm:SYrc34}
	\[ 0 = [[ x_{i1}^{+} , [ x_{i1}^{+} , x_{j0}^{+} ] ] , x_{j1}^{-}] = [ x_{i1}^{+} , [x_{i1}^{+} , h_{j1}] ]. \]
	We then apply $[\cdot, x_{i0}^{-}]$:
	\[ 0 = [ [ x_{i1}^{+} , [x_{i1}^{+} , h_{j1}] ] , x_{i0}^{-} ] = \]
	\[ [ h_{i1} , [ x_{i1}^{+} , h_{j1} ] ] + (-1)^{1+|i|} (\alpha_{i}, \alpha_{j}) ( [ x_{i1}^{+} , [x_{i1}^{+} , x_{i1}^{-}] ] + \frac{\hbar}{2} [ x_{i1}^{+} , (-1) (\alpha_{i} , \alpha_{j}) \{ x_{i1}^{+} , x_{i0}^{-} \} + \{ h_{j0} , h_{i1} \} ] ). \]
	We apply $[\cdot, x_{i0}^{-}]$ again to obtain by Lemmas \ref{lm:sy10} and \ref{lm:hxji}:
	\[ 0 = [ [ [ x_{i1}^{+} , [x_{i1}^{+} , h_{j1}] ] , x_{i0}^{-} ] , x_{i0}^{-} ] =  \]
	\[ (-1) (\alpha_{i}, \alpha_{i}) [ x_{i1}^{-}, [ x_{i1}^{+} , h_{j1} ] ] + (-1) ( \alpha_{i}, \alpha_{j} ) 2 [ h_{i1} , [ x_{i1}^{+} , x_{i1}^{-} ] ] + (-1) (\alpha_{i}, \alpha_{j}) [ x_{i1}^{+} , [ h_{i1}, x_{i1}^{-} ] ] + \]
	\[ \frac{\hbar}{2} (-1) (\alpha_{i}, \alpha_{i}) [ \{ h_{i0}, x_{i0}^{-} \}, [ x_{i1}^{+} , h_{j1} ] ] + \hbar [ h_{i1} , \{ x_{i1}^{+} , x_{i0}^{-} \} ] + \]
	\[ \frac{\hbar}{2} (-1) (\alpha_{i} , \alpha_{j}) [ x_{i1}^{+} , \{ h_{j0} , [h_{i1}, x_{i0}^{-}] \} ] = \]
	\[ (-1) ( \alpha_{i}, \alpha_{j} ) 3 [ h_{i1} , [ x_{i1}^{+} , x_{i1}^{-} ] ] +  \]
	\[ (\alpha_{i},\alpha_{i}) (\alpha_{i}, \alpha_{j}) \frac{\hbar}{2} ( (  (\alpha_{i}, \alpha_{j}) - (\alpha_{i}, \alpha_{i}) ) \{ x_{i1}^{-} , x_{i1}^{+} \} + \{ h_{j0} , [x_{i1}^{-} , x_{i1}^{+} ] \} + (-1) \{ h_{i0} , [ x_{i1}^{-} , x_{i1}^{+} ] \} ) + \]
	\[ \frac{\hbar}{2} (\alpha_{i}, \alpha_{j}) (\alpha_{i}, \alpha_{i}) \{ [h_{i1}, x_{i1}^{+}], x_{i0}^{-} \} + \]
	\[ \frac{\hbar}{2} (\alpha_{i}, \alpha_{j}) (\alpha_{i},\alpha_{i}) ( \{ h_{i0} , [ x_{i1}^{-} , x_{i1}^{+} ] \} + (-1) 2 \hbar h_{i0}^2 h_{i1} + \frac{\hbar}{2} (\alpha_{i}, \alpha_{i}) \{ h_{i0} , \{ x_{i1}^{+} , x_{i0}^{-} \} \} ) + \]
	\[ \frac{\hbar^2}{4} (\alpha_{i}, \alpha_{i}) (\alpha_{i}, \alpha_{j}) ( (\alpha_{i}, \alpha_{i}) \{ \{ h_{j0}, x_{i1}^{+} \} , x_{i0}^{-}  \} + \{ h_{i0} , - \{ h_{j0} , h_{i1} \} + (\alpha_{i} , \alpha_{j}) \{ x_{i1}^{+} , x_{i0}^{-} \}  \} ) + \]
	\[ (-1) \frac{\hbar^2}{4} (\alpha_{i}, \alpha_{i}) (\alpha_{i}, \alpha_{j}) (  (\alpha_{i}, \alpha_{i}) \{ \{ h_{i0}, x_{i1}^{+} \} , x_{i0}^{-}  \} + (-1) 4 h_{i0}^2 h_{i1} + (\alpha_{i} , \alpha_{i}) \{ h_{i0} , \{ x_{i1}^{+} , x_{i0}^{-} \}  \} )+ \]
	\[ \hbar \{ [h_{i1}, x_{i1}^{+}] , x_{i0}^{-} \} + \hbar \frac{3}{2} ( -(\alpha_{i}, \alpha_{i}) \{ x_{i1}^{+} , x_{i1}^{-} \} - \frac{\hbar ( \alpha_{i} , \alpha_{i} )}{2} \{ x_{i1}^{+} , \{ h_{i0} , x_{i0}^{-} \} \} )  +  \]
	\[ \frac{\hbar}{2} (\alpha_{i},\alpha_{i}) (\alpha_{i} , \alpha_{j}) \{ h_{j0} , [ x_{i1}^{+} , x_{i1}^{-} ] \} + \frac{\hbar^2 ( \alpha_{i} , \alpha_{i} ) (\alpha_{i} , \alpha_{j})}{4} \{ h_{j0} , -(\alpha_{i},\alpha_{i}) \{ x_{i1}^{+} , x_{i0}^{-} \} + \{ h_{i0} , h_{i1} \}  \}. \]
	Thus
	\[ ( \alpha_{i}, \alpha_{j} ) 3 [ h_{i1} , [ x_{i1}^{+} , x_{i1}^{-} ] ] =  \]
	\[ -\frac{\hbar^2}{2} (\alpha_{i}, \alpha_{i}) \{ \{ h_{j0}, x_{i1}^{+} \} , x_{i0}^{-}   \} + (-1)  \frac{\hbar^2}{2} (\alpha_{i}, \alpha_{j}) \{ h_{i0} , \{ x_{i1}^{+} , x_{i0}^{-} \} \} + \]
	\[ \frac{\hbar^2}{2}  (\alpha_{i}, \alpha_{i}) \{ \{ h_{i0}, x_{i1}^{+} \} , x_{i0}^{-}  \} + (-1) \frac{3 \hbar^2 ( \alpha_{i} , \alpha_{i} )}{4} \{ x_{i1}^{+} , \{ h_{i0} , x_{i0}^{-} \} \}  +  \]
	\[ (\alpha_{i},\alpha_{i}) \frac{\hbar^2}{2} \{ h_{j0} , \{ x_{i1}^{+} , x_{i0}^{-} \} \} = (-1) \hbar^2 h_{i1} + \hbar^2 h_{i1} = 0. \]
	This is nothing but the assertion.
\end{proof}

\begin{lemma}
	\label{lm:h12oe}
	Let $[h_{i1}, h_{i2}] = 0$, where $i \in I$, and $|\alpha_{i}|= \bar{0}$. Then $[h_{i1}, h_{j2}] = 0$, where $(\alpha_{i},\alpha_{j}) \ne 0$ $(j \in I)$.
\end{lemma}
\begin{proof}
	By \eqref{eq:mpy1} and Lemmas \ref{lm:cel02}, \ref{lm:hii12}, \ref{lm:r233} it follows that for we can apply Lemma \ref{lm:ijtrancart} with $p=n=i$, $m=j$ ($|i-j|=1$), $z=1$, $s^{'}=2$. The result follows.
\end{proof}

\begin{lemma}
	The equation \eqref{rm:lsar} holds for all $i \in I$.
\end{lemma}
\begin{proof}
	The result follows from \eqref{eq:hia12}, Lemma \ref{lm:hii12}, Lemma \ref{lm:h12oe} and our constraints on Dynkin diagrams introduced before Theorem \ref{th:mpy}.
\end{proof}

\begin{lemma}
	\label{lm:hxinej}
	Suppose that $i,j \in I$ and $i \ne j$. The equation \eqref{eq:SYrc2} holds for any $r,s \in \mathbb{N}_{0}$.
\end{lemma}
\begin{proof}
	We prove \eqref{eq:SYrc2} by induction on $r$ and $s$. The initial case $r=s=0$ is our assumption \eqref{eq:mpy4}. Let $X^{\pm}(r,s)$ be the result of substracting the right-hand side of \eqref{eq:SYrc2} from the left-hand side. Note that by Lemma \ref{lm:SYrc34} and \eqref{eq:rec2}
	\begin{equation}
		\label{eq:hxrs+}
		[h_{i,r+1} , x^{+}_{js} ] = [ [ x^{+}_{i,r+1} , x^{-}_{i0} ] , x^{+}_{js} ] = (-1)^{|i||j|} [ [ x_{i,r+1}^{+} , x_{js}^{+} ] , x_{i0}^{-} ] =
	\end{equation}
	\[ (-1)^{|i||j|} [ [x_{ir}^{+}, x_{j,s+1}^{+}] + \frac{c_{ij} \hbar}{2} \{ x_{ir}^{+} , x_{js}^{+} \} , x_{i0}^{-} ] = [h_{ir} , x^{+}_{j,s+1} ] + \frac{c_{ij} \hbar}{2} \{ h_{ir} , x_{js}^{+} \}; \]

	\begin{equation}
		\label{eq:hxrs-}
	[h_{i,r+1} , x^{-}_{js} ] = [ [ x^{+}_{i0} , x^{-}_{i,r+1} ] , x^{-}_{js} ] = [ x_{i0}^{+} , [ x_{i,r+1}^{-} , x_{js}^{-} ] ] =
	\end{equation}
	\[ [ x_{i0}^{+} , [x_{ir}^{-}, x_{j,s+1}^{-}] - \frac{c_{ij} \hbar}{2} \{ x_{ir}^{-} , x_{js}^{-} \} ] =  [h_{ir} , x^{-}_{j,s+1} ] - \frac{c_{ij} \hbar}{2} \{ h_{ir} , x_{js}^{-} \}, \]
	for any $r,s \in \mathbb{N}_{0}$. Thus the result follows by induction hypothesis.
\end{proof}

\begin{lemma}
	\label{lm:r233}
	Let $i \in I$. Equations \eqref{eq:Cer}, \eqref{eq:SYrc3} $(0 \le r+s \le 3)$ and \eqref{eq:SYrc2} $(r=0,1; \; s \in \mathbb{N}_{0})$ hold for $i=j$.
\end{lemma}
\begin{proof}
	Consider the equation \eqref{eq:Cer}. The proof follows from \eqref{eq:mpy1} and Lemmas \ref{lm:cel02}, \ref{lm:hii12}.

	Consider the equation \eqref{eq:SYrc3}. It follows from Lemma \ref{lm:l2} for $0 \le r+s \le 2$. Let $i^{'}=i$, if $|\alpha_{i}|=\bar{0}$, and $|i-i^{'}|=1$, if $|\alpha_{i}|=\bar{1}$ ($i, i^{'} \in I$). By Lemmas \eqref{lm:hii12} and \eqref{lm:h12oe} we have
	\[ 0 = [ h_{i2} , \widetilde{h}_{i^{'}1} ] = [ [ x_{i2}^{+} , x_{i0}^{-} ] , \widetilde{h}_{i^{'}1} ] = ( \alpha_{i^{'}} , \alpha_{i} ) ( [ x_{i2}^{+} , x_{i1}^{-} ] - [ x_{i3}^{+} , x_{i0}^{-} ] ). \]
	Apply $[\widetilde{h}_{i^{'}1}, \cdot]$ to
	\[ [ x_{i2}^{+} , x_{i0}^{-} ] = [ x_{i1}^{+} , x_{i1}^{-} ] = [ x_{i0}^{+} , x_{i2}^{-} ]. \]
	Then
	\[ 0 = [ x_{i2}^{+} , x_{i1}^{-} ] - [ x_{i3}^{+} , x_{i0}^{-} ] = [ x_{i1}^{+} , x_{i2}^{-} ] - [ x_{i2}^{+} , x_{i1}^{-} ] = [ x_{i0}^{+} , x_{i3}^{-} ] - [ x_{i1}^{+} , x_{i2}^{-} ]. \]
	Thus we get
	\[ [ x_{i3}^{+} , x_{i0}^{-} ] = [ x_{i2}^{+} , x_{i1}^{-} ] = [ x_{i1}^{+} , x_{i2}^{-} ] = [ x_{i0}^{+} , x_{i3}^{-} ]. \]

	Consider the equation \eqref{eq:SYrc2}. Recall that here $|\alpha_{i}| = \bar{0}$. The case $1 \le r+s \le 2$ follows from \eqref{eq:mpy4}, Lemma \ref{lm:hx20} and Lemma \ref{lm:h0eq}. Note that if \eqref{eq:SYrc2} holds for $0 \le r \le 1$ and $s \in \mathbb{N}_{0}$ we have
	\[ [ \widetilde{h}_{i1} , [h_{i,r+1} , x_{is}^{\pm}] ] = [ \widetilde{h}_{i1} , [h_{ir} , x_{i,s+1}^{\pm}] \pm \frac{c_{ii} \hbar}{2} \{h_{ir}, x_{is}^{\pm}\} ] \Leftrightarrow \]
	\[ [h_{i,r+1} , x_{i,s+1}^{\pm}] = [h_{ir} , x_{j,s+2}^{\pm}] \pm \frac{c_{ii} \hbar}{2} \{h_{ir}, x_{j,s+1}^{\pm}\}. \]
	Thus \eqref{eq:SYrc2} is true for $0 \le r \le 1$ and all $s \in \mathbb{N}_{0}$.
\end{proof}

We use arguments established in \cite{L93} and \cite{GT13} to prove the relation \eqref{eq:Cer}. It is necessary to introduce auxiliary elements and prove a few intermediate lemmas. Our primary goal is the construction of elements $\dbtilde{h}_{ij,r}$ ($i,j \in I$, $r \in \mathbb{N}_{0}$) that satisfy equations \eqref{eq:thxo} and \eqref{eq:thxs}.

Set for any $i \in I$
\begin{equation}
	\label{eq:hln}
	\widetilde{h}_{i}(t) := \hbar \sum_{r \ge 0} \widetilde{h}_{i,r} t^{-r-1} = \log ( 1 + \hbar \sum_{r \ge 0} h_{i,r} t^{-r-1} ) \in Y_{\hbar}(\mathfrak{g})^{0}[[t^{-1}]].
\end{equation}

\begin{lemma}
	\label{lm:thfdpr}
	Let \eqref{eq:Cer} hold for $i,j \in I$ and $0 \le r,s \le v$ and let \eqref{eq:SYrc2} hold for $0 \le r \le v$ and $s \in \mathbb{N}_{0}$. Then, unless $i=j$ and $|\alpha_{i}|=\bar{1}$, for $i,j \in I$, $0 \le r \le v$, $s \in \mathbb{N}_{0}$,
	\begin{equation}
		\label{eq:thx}
		[ \widetilde{h}_{i,r}, x_{j,s}^{\pm} ] = \pm (\alpha_{i},\alpha_{j}) x_{j,r+s}^{\pm} \pm (\alpha_{i},\alpha_{j}) \sum_{p=1}^{\lfloor r/2 \rfloor} {{r}\choose{2p}} \frac{(\hbar (\alpha_{i},\alpha_{j}) / 2)^{2p}}{2p+1} x_{j,r+s-2p}^{\pm}.
	\end{equation}
\end{lemma}
\begin{proof}
	Note that from \eqref{eq:hln} follows that for arbitrary $r \in \mathbb{N}_{0}$ we have $\widetilde{h}_{i,r} = f(h_{i,0}, h_{i,1} , ..., h_{i,r})$ for some element of free algebra $f \in \Bbbk \langle x_{1},...,x_{r+1} \rangle$. Hence, while deriving \eqref{eq:thx}, we may and shall assume that $\eqref{eq:Cer}$ holds for all $r,s \in \mathbb{N}_{0}$. Then the result follows from \cite[Lemma 2.7, Lemma 2.9, Remark 3.1]{GT13}.
\end{proof}

Suppose that elements $\widetilde{h}_{i,r}$ are defined for $i \in I$ and $0 \le r \le v$. For each $\widetilde{h}_{i,r}$ fix $j \in I$ unless $i=j$ and $|\alpha_{i}|=\bar{1}$. Set $\dbtilde{h}_{ij,0} = h_{i,0}$, and define inductively for $1 \le r \le v$
\[ \dbtilde{h}_{ij,r} = \widetilde{h}_{i,r} - \sum_{p=1}^{\lfloor r/2 \rfloor} {{r}\choose{2p}} \frac{(\hbar (\alpha_{i},\alpha_{j}) / 2)^{2p}}{2p+1} \dbtilde{h}_{ij,r-2p}. \]
Thus we have

\begin{lemma}
	\label{lm:hwdpr}
	Suppose that Lemma \ref{lm:thfdpr} holds. Then in the same notations
	\begin{equation}
		\label{eq:thxo}
		[ \dbtilde{h}_{ij,r} , x_{j,s}^{\pm} ] = \pm (\alpha_{i}, \alpha_{j})x_{j,r+s}^{\pm},
	\end{equation}
	for $i, j \in I$, $0 \le r \le v$, $s \in \mathbb{N}_{0}$, and
	\begin{equation}
		\label{eq:thxs}
		\dbtilde{h}_{ij,r} = h_{i,r} + f(h_{i,0}, h_{i,1}, ... , h_{i,r-1})
	\end{equation}
	for some polynomial $f(x_{1},x_{2},...,x_{r}) \in \Bbbk[x_{1},x_{2},...,x_{r}]$.
\end{lemma}
\begin{proof}
	Equation \eqref{eq:thxs} holds from the definition of $\dbtilde{h}_{ij,r}$ and observation that elements $\widetilde{h}_{i,r}$ ($0 \le r \le v$) are of the form $\widetilde{h}_{i,r} = f(h_{i,0}, h_{i,1} , ..., h_{i,r})$ for some polynomial $f \in \Bbbk [ x_{1},...,x_{r+1} ]$.

	Next inductively for $0 \le r \le v$ we verify that
	\[ [ \dbtilde{h}_{ij,r} , x_{j,s}^{\pm} ] = [ \widetilde{h}_{i,r} - \sum_{p=1}^{\lfloor r/2 \rfloor} {{r}\choose{2p}} \frac{(\hbar (\alpha_{i},\alpha_{j}) / 2)^{2p}}{2p+1} \dbtilde{h}_{ij,r-2p} , x_{j,s}^{\pm} ] = \]
	\[ \pm (\alpha_{i},\alpha_{j}) x_{j,r+s}^{\pm} \pm (\alpha_{i},\alpha_{j}) \sum_{p=1}^{\lfloor r/2 \rfloor} {{r}\choose{2p}} \frac{(\hbar (\alpha_{i},\alpha_{j}) / 2)^{2p}}{2p+1} x_{j,r+s-2p}^{\pm} + \]
	\[ \mp (\alpha_{i},\alpha_{j}) \sum_{p=1}^{\lfloor r/2 \rfloor} {{r}\choose{2p}} \frac{(\hbar (\alpha_{i},\alpha_{j}) / 2)^{2p}}{2p+1} x_{j,r+s-2p}^{\pm} = \pm (\alpha_{i},\alpha_{j}) x_{j,r+s}^{\pm}. \]
\end{proof}

\begin{lemma}
	\label{lm:ijtrancart}
	Let $p, n, m \in I$ and $z \in \mathbb{N}_{0}$. Suppose that $[h_{pz} , h_{nv}] = 0$, $[h_{pz}, h_{mv}] = 0$ for $0 \le v \le s^{'}-1$ $(s^{'} \in \mathbb{N}_{0})$; $(\alpha_{p}, \alpha_{n}) \ne 0$ and $(\alpha_{p}, \alpha_{m}) \ne 0$. Moreover, let $[h_{nv_1} , h_{n v_2}] = 0$ and $[h_{mv_1} , h_{m v_2}] = 0$ for $0 \le v_1, v_2 \le s^{'}$ and \eqref{eq:SYrc2} hold for $(i,j)=(n,p)=(m,p)$, $0 \le r \le s^{'}$ and any $s \in \mathbb{N}_{0}$. Then $[h_{pz}, h_{ns^{'}}] = \frac{(\alpha_{n} , \alpha_{p})}{(\alpha_{m} , \alpha_{p})} [h_{pz}, h_{ms^{'}}]$.
\end{lemma}
\begin{proof}
	From conditions in the statement and Lemmas \ref{lm:thfdpr}, \ref{lm:hwdpr} it follows that we can define elements $\dbtilde{h}_{np,r}$ and $\dbtilde{h}_{mp,r}$. Note that
	\[ \frac{(\alpha_{n} , \alpha_{p})}{(\alpha_{m} , \alpha_{p})} [\dbtilde{h}_{mp,r} , x_{ps}^{\pm}] = \pm (\alpha_{n} , \alpha_{p}) x_{p,r+s}^{\pm} = [\dbtilde{h}_{np,r} , x_{ps}^{\pm}]. \]
	Thus
	\[ [h_{pz} , h_{ns^{'}}] = [h_{ir} , \dbtilde{h}_{np,s^{'}}] = [ [x_{pz}^{+} , x_{p0}^{-}] , \dbtilde{h}_{np,s^{'}} ] = \]
	\[ (-1) [x_{pz}^{+} , [\dbtilde{h}_{np,s^{'}}, x_{p0}^{-}]] + (-1)^{|p|} [x_{p0}^{-} , [\dbtilde{h}_{np,s^{'}}, x_{pz}^{+}]] = \]
	\[ \frac{(\alpha_{n} , \alpha_{p})}{(\alpha_{m} , \alpha_{p})} ( (-1) [x_{pz}^{+} , [\dbtilde{h}_{mp,s^{'}}, x_{p0}^{-}]] + (-1)^{|p|} [x_{p0}^{-} , [\dbtilde{h}_{mp,s^{'}}, x_{pz}^{+}]] ) = \]
	\[ \frac{(\alpha_{n} , \alpha_{p})}{(\alpha_{m} , \alpha_{p})} [h_{pz} , \dbtilde{h}_{mp,s^{'}}] = \frac{(\alpha_{n} , \alpha_{p})}{(\alpha_{m} , \alpha_{p})} [h_{pz} , h_{ms^{'}}] . \]
\end{proof}

Now we are ready to prove
\begin{lemma}
	\label{lm:cartelrelii}
	Let $i, j \in I$. Equations \eqref{eq:Cer}, \eqref{eq:SYrc3} and \eqref{eq:SYrc2} hold for $i=j$ and all $r,s \in \mathbb{N}_{0}$.
\end{lemma}
\begin{proof}
	The proof is completely the same as in \cite{L93} (see considerations after Lemma 2.2) where the proof is based on mathematical induction. That's why we only explain here how to deal with equations (2.25)-(2.29) in \cite{L93} where some extra arguments are needed to deal with the case $|\alpha_{i}| = \bar{1}$. First note that the base of induction is obtained by Lemma \ref{lm:r233}. In each equation we use the same notations as in mentioned paper. From now on suppose that $|\alpha_{i}| = \bar{1}$. Let $i^{'} \in I$ be such that $|i^{'}-i|=1$ and $|\alpha_{i^{'}}| = \bar{0}$. Notice that such $i^{'}$ exists by our constraints on Dynkin diagrams. $p \in \mathbb{N}_{0}$ is a fixed number as in the paper.

	Consider the equation (2.25). By the induction hypothesis and the proof of Lemma \ref{lm:cartelrelinej} it follows that $[h_{i^{'}p} , h_{iv}] = 0$ for $0 \le v \le p-1$. We are able to define $\dbtilde{h}_{i^{'}i,p}$ by Lemma \ref{lm:thfdpr} and to use Lemma \ref{lm:ijtrancart} to get
	\[ 0 = [ h_{i^{'}p} , h_{i^{'}p} ] = [ h_{ip} , h_{i^{'}p} ] = [ h_{ip} , \dbtilde{h}_{i^{'}i,p} ]. \]
	Further steps are the same as in the paper. In the equation (2.26) we use $[\widetilde{h}_{i^{'}1}, \cdot]$. Further steps are the same.

	 In the equation between (2.26) and (2.27):
	\[ [ h_{i^{'},r-q} , h_{i^{'}, q+1} ] = \frac{(\alpha_{i^{'}} , \alpha_{i^{'}})}{(\alpha_{i} , \alpha_{i^{'}})} [ h_{i,r-q} , h_{i^{'}, q+1} ] = \frac{(\alpha_{i^{'}} , \alpha_{i^{'}})}{(\alpha_{i} , \alpha_{i^{'}})} [ h_{i,r-q} , \dbtilde{h}_{i^{'}i,q+1} ]. \]
	Further steps are the same.

	Consider the equation (2.27). Here we use $[\widetilde{h}_{i^{'}1},
	\cdot]$. Further steps are the same. Consider the equation (2.28). By the induction hypothesis and the proof of Lemma \ref{lm:cartelrelinej} it follows that $[h_{i,p+1} , h_{i^{'}v}] = 0$ for $0 \le v \le p-1$; $[h_{iv} , h_{i^{'}p}] = 0$ for $0 \le v \le p$. We are able to define $\dbtilde{h}_{i^{'}i,p}$ by Lemma \ref{lm:thfdpr} and to use Lemma \ref{lm:ijtrancart} to get
	\[ [ h_{i^{'}, p+1} , h_{i^{'}p} ]  = \frac{(\alpha_{i^{'}} , \alpha_{i^{'}})}{(\alpha_{i} , \alpha_{i^{'}})} [ h_{i, p+1} , h_{i^{'}p} ] = \frac{(\alpha_{i^{'}} , \alpha_{i^{'}})}{(\alpha_{i} , \alpha_{i^{'}})} [ h_{i, p+1} , \dbtilde{h}_{i^{'}i,p} ]. \]
	Further steps are the same.

	In the equation (2.29) we use the same arguments as in the (2.28) and Lemma \ref{lm:hxinej} in the following to get
	\[ [ h_{i^{'}, p+1} , h_{i^{'}p} ] = \frac{(\alpha_{i^{'}} , \alpha_{i^{'}})}{(\alpha_{i} , \alpha_{i^{'}})} [ h_{i^{'},p+1} , h_{ip} ]. \]
	Further steps are the same. It is easy to see that all considerations after the equation (2.29) are also true in our case.
\end{proof}

\begin{lemma}
	\label{lm:cartelrelinej}
	Suppose that $i,j \in I$ and $i \ne j$. The equation \eqref{eq:Cer} holds for any $r,s \in \mathbb{N}_{0}$.
\end{lemma}
\begin{proof}
	We proof by induction on $r \in \mathbb{N}_{0}$. Consider the case $r=0$ and any $s \in \mathbb{N}_{0}$. It is true by Lemma \ref{lm:cel02}.
	Suppose that $[h_{il}, h_{js}] = 0$ for $0 \le l \le r$ and all $s \in \mathbb{N}_{0}$. Then using Lemmas \ref{lm:hxinej} and \ref{lm:hwdpr} we can define element $\dbtilde{h}_{ij,r+1}$. Then by Lemma \ref{lm:cartelrelii}
	\[ [ h_{i,r+1} , h_{js} ] = [ \dbtilde{h}_{ij,r+1} , h_{js} ] = [ \dbtilde{h}_{ij,r+1} , [x_{js}^{+} , x_{j0}^{-}] ] = \]
	\[ (\alpha_{i} , \alpha_{j}) ( [ x_{js}^{+} , x_{j,r+1}^{-} ] - [ x_{j,r+s+1}^{+} , x_{j0}^{-} ] ) =  (\alpha_{i} , \alpha_{j}) ( h_{j,r+s+1} - h_{j,r+s+1} ) = 0. \]
	Thus the result follows by induction hypothesis.
\end{proof}

\begin{lemma}
	The relation \eqref{eq:SYrc4} is satisfied for all $i,j \in I$ and $r,s \in \mathbb{N}_{0}$.
\end{lemma}
\begin{proof}
	The case for $i \ne j$ is proved in Lemma \ref{lm:SYrc34}. Suppose that $i=j$. Then $|\alpha_{i}|=0$. We prove by induction on $r$ and $s \in \mathbb{N}_{0}$. The initial case $(r,s)=(0,0)$ is our initial assumption \eqref{eq:mpy5}. Let $X^{\pm}(r,s)$ be the result of substracting the right-hand side of \eqref{eq:SYrc4} from the left-hand side. Using the relation \eqref{eq:thxo} we have for an arbitrary $r \in \mathbb{N}_{0}$:
	\[ 0 = [ \dbtilde{h}_{ii,r}, X^{\pm}(0,0) ] = (\alpha_{i}, \alpha_{i}) X^{\pm}(r,0) + (\alpha_{i}, \alpha_{i}) X^{\pm}(0,r) = \]
	\[ 2 (\alpha_{i}, \alpha_{i}) X^{\pm}(r,0) \Rightarrow X^{\pm}(r,0) = 0. \]
	Now for an arbitrary $s \in \mathbb{N}_{0}$:
	\[ 0 = [ \dbtilde{h}_{ii,s}, X^{\pm}(r,0) ] = (\alpha_{i}, \alpha_{i}) X^{\pm}(r+s,0) + (\alpha_{i}, \alpha_{i}) X^{\pm}(r,s) = \]
	\[ (\alpha_{i}, \alpha_{i}) X^{\pm}(r,s) \Rightarrow X^{\pm}(r,s) = 0. \]
	By induction hypothesis the result follows.
\end{proof}

\begin{lemma}
	\label{lm:bcssr}
	Relations \eqref{eq:bssr} and \eqref{eq:cssr} are satisfied for all $i,j \in I$ and $r,s,t \in \mathbb{N}_{0}$.
\end{lemma}
\begin{proof}
	Let $X^{\pm}(r,s)$ be the left-hand side of \eqref{eq:bssr}. We prove by induction on $r$ and $s \in \mathbb{N}_{0}$. The initial case $(r,s)=(0,0)$ is our initial assumption \eqref{eq:mpy7}. Note that
	\[ [ \widetilde{h}_{m1} , X^{\pm} (r,s) ] = \pm ( (\alpha_{m} , \alpha_{j}) X^{\pm} (r, s+1) + (\alpha_{m}, \alpha_{i}) X^{\pm} (r+1,s) ), \]
	where $m \in I$. Unless $i=j$ and $|\alpha_{i}|=\bar{1}$ we succesively apply $[\widetilde{h}_{i1}, \cdot]$ and $[\widetilde{h}_{j1}, \cdot]$ to get the result.

	If $i=j$ and $|\alpha_{i}|=\bar{1}$ we note that
	\[ [ \widetilde{h}_{i^{'}i,r} , X^{\pm} (0,0) ] = \pm (\alpha_{i^{'}} , \alpha_{i}) ( X^{\pm} (r, 0) + X^{\pm} (0,r) ) = \pm 2 (\alpha_{i^{'}} , \alpha_{i}) X^{\pm} (r, 0), \]
	where $|i-i^{'}|=1$ ($i^{'} \in I$). Thus $X^{\pm} (r, 0) = 0$ for all $r \in \mathbb{N}_{0}$. Now for an arbitrary $s \in \mathbb{N}_{0}$:
	\[ 0 = [ \dbtilde{h}_{i^{'}i,s}, X^{\pm}(r,0) ] = (\alpha_{i^{'}}, \alpha_{i}) X^{\pm}(r+s,0) + (\alpha_{i^{'}}, \alpha_{i}) X^{\pm}(r,s) = \]
	\[ (\alpha_{i^{'}}, \alpha_{i}) X^{\pm}(r,s) \Rightarrow X^{\pm}(r,s) = 0. \]
	By induction hypothesis the result follows.

	Let $X^{\pm}(r,s,t)$ be the left-hand side of \eqref{eq:cssr}. We prove by induction on $r$, $s$ and $t \in \mathbb{N}_{0}$. The initial case $(r,s,t)=(0,0,0)$ is our initial assumption \eqref{eq:mpy8}. We have proved in Lemma \ref{lm:fassrс} that $X^{\pm}(0,0,t)$ for all $t \in \mathbb{N}_{0}$. Note that for any $r,t \in \mathbb{N}_{0}$
	\[ 0 = [ \dbtilde{h}_{ii,r} , X^{\pm}(0,0,t) ] = \sum_{z=t}^{t+r} a_{z} X^{\pm}(0,0,z) + \]
	\[ (\alpha_{i}, \alpha_{i}) ( X^{\pm}(r,0,t) + X^{\pm}(0,r,t) ) = 2 (\alpha_{i}, \alpha_{i}) X^{\pm}(r,0,t) \Rightarrow X^{\pm}(r,0,t) = 0, \]
	where $a_z \in \Bbbk$ ($r \le z \le t+r$). Suppose that for any $r,t \in \mathbb{N}_{0}$ and for $0 \le l \le s$ ($s \in \mathbb{N}_{0}$), $X^{\pm}(r,s,t)=0$. Then
	\[ 0 = [\widetilde{h}_{i1} , X^{\pm}(r,s,t)] = (\alpha_{i}, \alpha_{j}) X^{\pm} ( r,s,t+1 ) + (\alpha_{i}, \alpha_{i}) ( X^{\pm}(r+1,s,t) + X^{\pm} (r,s+1,t) ) =  \]
	\[ (\alpha_{i}, \alpha_{i}) X^{\pm} (r,s+1,t) \Rightarrow X^{\pm} (r,s+1,t) = 0. \]
	Thus the result follows by induction hypothesis.

\end{proof}

\begin{lemma}
	The relation \eqref{eq:SYrc5} is satisfied for all $i,j \in I$ and $r,s \in \mathbb{N}_{0}$.
\end{lemma}
\begin{proof}
	We prove by induction on $r+s \in \mathbb{N}_{0}$. The case $0 \le r+s \le 1$ is our initial assumptions \eqref{eq:mpy2}, \eqref{eq:mpy6} and Lemma \ref{lm:h0eq}. Suppose that the result holds for all $1 \le l \le r+s$. Then by induction hypothesis and Lemma \ref{lm:cartelrelinej}
	\[ [ h_{i,r+s-q} , x_{i,q+1}^{\pm} ] = \pm ( \alpha_{i^{'}} , \alpha_{i} )^{-1} [ h_{i,r+s-q} , [ \widetilde{h}_{i^{'}1} , x_{iq}^{\pm}] ] = \]
	\[ \pm (-1) ( \alpha_{i^{'}} , \alpha_{i} )^{-1} ( [ \widetilde{h}_{i^{'}1} , [x_{iq}^{\pm} , h_{i,r+s-q}] ] + [ x_{iq}^{\pm} , [h_{i,r+s-q} , \widetilde{h}_{i^{'}1} ] ] ) = 0, \]
	where $0 \le q \le r+s$, $|i-i^{'}|=1$ ($i^{'} \in I$).
	Also by Lemmas \ref{lm:bcssr}, \ref{lm:cartelrelii} and \ref{lm:cartelrelinej}
	\[ [ h_{i,q+1} , x_{i,r+s-q}^{+} ] = (-1) ( [x_{i,q+1}^{+} , [ x_{i,r+s-q}^{+} , x_{i0}^{-} ] ] + [ x_{i0}^{-} , [ x_{i,r+s-q}^{+} , x_{i,q+1}^{+} ] ] ) =  \]
	\[ [ h_{i,r+s-q} ,x_{i,q+1}^{+} ] = (-1) ( \alpha_{i^{'}} , \alpha_{i} )^{-1} ( [ \widetilde{h}_{i^{'}1} , [x_{iq}^{+} , h_{i,r+s-q}] ] + [ x_{iq}^{+} , [h_{i,r+s-q} , \widetilde{h}_{i^{'}1} ] ] ) = 0, \]

	\[ [ h_{i,q+1} , x_{i,r+s-q}^{-} ] = (-1) ( [x_{i0}^{+} , [ x_{i,r+s-q}^{-} , x_{i,q+1}^{-} ] ] + [ x_{i,q+1}^{-} , [ x_{i,r+s-q}^{-} , x_{i0}^{+} ] ] ) =  \]
	\[ [ h_{i,r+s-q} ,  x_{i,q+1}^{-} ] = ( \alpha_{i^{'}} , \alpha_{i} )^{-1} ( [ \widetilde{h}_{i^{'}1} , [x_{iq}^{-} , h_{i,r+s-q}] ] + [ x_{iq}^{-} , [h_{i,r+s-q} , \widetilde{h}_{i^{'}1} ] ] ) = 0, \]
	where $0 \le q \le r+s$, $|i-i^{'}|=1$ ($i^{'} \in I$). The result follows by induction hypothesis.
\end{proof}

\begin{lemma}
	\label{lm:leqmp}
	The relation \eqref{eq:qssr} is satisfied for all $j \in I$ and $r,s \in \mathbb{N}_{0}$.
\end{lemma}
\begin{proof}
	Let $X^{\pm}(r,0,0,s)$ be the left-hand side of \eqref{eq:qssr}. We prove by induction on $r$ and $s \in \mathbb{N}_{0}$. The initial case $(r,0,0,s)=(0,0,0,0)$ is our initial assumption \eqref{eq:mpy9}. Note that if we apply $[\widetilde{h}_{m1} , \cdot]$, $[\widetilde{h}_{n1} , \cdot]$ and $[\widetilde{h}_{k1} , \cdot]$ ($m,n,k \in I$) to $X^{\pm}(r,0,0,s)$ we get from \eqref{eq:h1eq}
	\[ 0 = (\alpha_{m}, \alpha_{j-1}) X^{\pm}(r+1,0,0,s) + (\alpha_{m}, \alpha_{j}) ( X^{\pm}(r,1,0,s) + X^{\pm}(r,0,1,s) ) + (\alpha_{m}, \alpha_{j+1}) X^{\pm}(r,0,0,s+1), \]
	\[ 0 = (\alpha_{n}, \alpha_{j-1}) X^{\pm}(r+1,0,0,s) + (\alpha_{n}, \alpha_{j}) ( X^{\pm}(r,1,0,s) + X^{\pm}(r,0,1,s) ) + (\alpha_{n}, \alpha_{j+1}) X^{\pm}(r,0,0,s+1), \]
	\[ 0 = (\alpha_{k}, \alpha_{j-1}) X^{\pm}(r+1,0,0,s) + (\alpha_{k}, \alpha_{j}) ( X^{\pm}(r,1,0,s) + X^{\pm}(r,0,1,s) ) + (\alpha_{k}, \alpha_{j+1}) X^{\pm}(r,0,0,s+1). \]
	In order to determine when the determinant of the matrix
	\[ A = \begin{pmatrix}(\alpha_{m},\alpha_{j-1}) & (\alpha_{m},\alpha_{j}) & (\alpha_{m},\alpha_{j+1})\\ (\alpha_{n},\alpha_{j-1}) & (\alpha_{n},\alpha_{j}) & (\alpha_{n},\alpha_{j+1}) \\ (\alpha_{k},\alpha_{j-1}) & (\alpha_{k},\alpha_{j}) & (\alpha_{k},\alpha_{j+1}) \end{pmatrix}\]
	is nonzero depending on the grading of roots it is sufficient to consider the following Dynkin (sub)diagrams:

		\begin{tikzpicture}[scale=2]
			\tikzset{/Dynkin diagram,root radius=.07cm}
			\dynkin[labels={\alpha_{j-1},\alpha_{j},\alpha_{j+1}}] A{xtx}
		\end{tikzpicture}

		First suppose that $|I|=3$. Then select $m=j-1, n=j, k=j+1$. Then $\det(A) = - ( ( \alpha_{j-1} , \alpha_{j-1}) + ( \alpha_{j+1} , \alpha_{j+1} ) )$. In cases $|\alpha_{j-1}|=|\alpha_{j+1}|$ we get $\det(A)=0$. That's why we don't consider these Dynkin diagrams. In other cases the determinant is nonzero.

		If $|I|>3$ then there could be only the following Dynkin (sub)diagrams:

		\begin{enumerate}
		\item
		\begin{tikzpicture}[scale=2]
			\tikzset{/Dynkin diagram,root radius=.07cm}
			\dynkin[labels={\alpha_{j-2},\alpha_{j-1},\alpha_{j},\alpha_{j+1}}] A{xxtx}
		\end{tikzpicture}

		$|\alpha_{j-1}|=|\alpha_{j+1}|=\bar{0} \Rightarrow m=j-2, n=j, k=j+1$;

		$|\alpha_{j-1}|=\bar{1}, |\alpha_{j+1}|=\bar{0} \Rightarrow m=j-1, n=j, k=j+1$;

		$|\alpha_{j-1}|=\bar{0}, |\alpha_{j+1}|=\bar{1} \Rightarrow m=j-1, n=j, k=j+1$;

		$|\alpha_{j-1}|=|\alpha_{j+1}|=\bar{1} \Rightarrow m=j-2, n=j, k=j+1$.

		\item
		\begin{tikzpicture}[scale=2]
			\tikzset{/Dynkin diagram,root radius=.07cm}
			\dynkin[labels={\alpha_{j-1},\alpha_{j},\alpha_{j+1},\alpha_{j+2}}] A{xtxx}
		\end{tikzpicture}

		$|\alpha_{j-1}|=|\alpha_{j+1}|=\bar{0} \Rightarrow m=j+2, n=j, k=j+1$;

		$|\alpha_{j-1}|=\bar{1}, |\alpha_{j+1}|=\bar{0} \Rightarrow m=j-1, n=j, k=j+1$;

		$|\alpha_{j-1}|=\bar{0}, |\alpha_{j+1}|=\bar{1} \Rightarrow m=j-1, n=j, k=j+1$;

		$|\alpha_{j-1}|=|\alpha_{j+1}|=\bar{1} \Rightarrow m=j+2, n=j, k=j+1$.

	\end{enumerate}
	When the determinant is nonzero, we have $X^{\pm}(r+1,0,0,s)=X^{\pm}(r,0,0,s+1)=0$. The result follows by induction hypothesis.

\end{proof}

\begin{proof}[Proof of Theorem \ref{th:mpy}]
	Proof of the theorem follows from Remark \ref{rm:addrel} and Lemmas \ref{lm:h0eq} - \ref{lm:leqmp}.
\end{proof}

\vspace{1cm}

\subsection{The quantum loop superalgebra}

\label{sect:QLS}

We will use the following definition of a quantum loop superalgebra in terms of a system of generators analogous to the Drinfeld new system of generators \cite{Y99, T19b}.

Let $U_{\hbar}(L\mathfrak{g}_d)$ be the unital, associative Hopf superalgebra over  $\Bbbk[[\hbar]]$ topologically generated by $\{E_{i,k,d}, F_{i,k,d}, H_{i,k,d}\}_{i \in I, k \in \mathbb{Z}}$ with the $\mathbb{Z}_2$-grading $|H_{i,r,d}|=0$, $|E_{i,r,d}|=|F_{i,r,d}|=|\alpha_{i,d}|$, and subject to the following defining relations:

\noindent 1) For any  $i, j \in I$, $r,s \in \mathbb{Z}$
\begin{equation}
	[H_{i,r}, H_{j,s}] = 0. \label{QL1}
\end{equation}
2) For any $i, j \in I$, $k \in \mathbb{Z}$
\begin{equation}
	[H_{i,0}, E_{j,k}] = c^d_{i,j}E_{j,k}, \quad [H_{i,0}, F_{j,k}] = -c^d_{i,j}F_{j,k}. \label{QL2}
\end{equation}
3) For any $i, j \in I$ and $r, k \in \mathbb{Z}\backslash\{0\}$
\begin{equation}
	[H_{i,r}, E_{j,k}] = \frac{[rc^d_{i,j}]_{q}}{r}E_{j,r+k}, \quad [H_{i,r}, F_{j,k}] = -\frac{[rc^d_{i,j}]_{q}}{r}F_{j,r+k}. \label{QL3}
\end{equation}
4) For any $i, j \in I$ and $k, l  \in \mathbb{Z}$ ($c_{ij}^{d} \ne 0$)
\begin{eqnarray}
	&E_{i,k+1}E_{j,l} - (-1)^{|\alpha_{i,d}||\alpha_{j,d}|} q^{c^d_{ij}}E_{j,l}E_{i,k+1} = (-1)^{|\alpha_{i,d}||\alpha_{j,d}|} q^{c^d_{ij}}E_{i,k}E_{j,l+1} -  E_{j,l+1}E_{i,k},  \nonumber \quad \\
	&F_{i,k+1}F_{j,l} - (-1)^{|\alpha_{i,d}||\alpha_{j,d}|} q^{-c^d_{ij}}F_{j,l}F_{i,k+1} =  (-1)^{|\alpha_{i,d}||\alpha_{j,d}|} q^{-c^d_{ij}}F_{i,k}F_{j,l+1} -  F_{j,l+1}F_{i,k}, \quad \label{QL4}
\end{eqnarray}
5) For any $i, j \in I$, $k, l  \in \mathbb{Z}$
\begin{equation}
	[E_{i,k}, F_{j,l}] = \delta_{i,j} \frac{\psi_{i, k+l} - \varphi_{i, k+l}}{q - q^{-1}}.
	\label{QL5}
\end{equation}
6) For any $i, j \in I$, $k, l  \in \mathbb{Z}$, if $c^d_{ij}=0$, then
\begin{equation}
	[E_{i,k}, E_{j,l}] = 0, \quad [F_{i,k}, F_{j,l}] = 0.
	\label{QL6}
\end{equation}
7) For any $i,j \in I$, $r,k,l \in \mathbb{Z}$, if $c^d_{ij} \ne 0$ and $|\alpha_{i,d}|=\bar{0}$, then
\begin{eqnarray}
	&[ E_{i,r}, [ E_{i,k}, E_{j,l} ]_{q^{c^d_{ij}}} ]_{q^{-c^d_{ij}}} + [ E_{i,k}, [ E_{i,r}, E_{j,l} ]_{q^{c^d_{ij}}} ]_{q^{-c^d_{ij}}} = 0, \nonumber \quad \\
	&[ F_{i,r}, [ F_{i,k}, F_{j,l} ]_{q^{-c^d_{ij}}} ]_{q^{c^d_{ij}}} + [ F_{i,k}, [ F_{i,r}, F_{j,l} ]_{q^{-c^d_{ij}}} ]_{q^{c^d_{ij}}} = 0.
	\label{QL7}
\end{eqnarray}
8) For any $i \in I$, $r,s,k,l \in \mathbb{Z}$, if $|\alpha_{i,d}|=\bar{1}$, then
\begin{eqnarray}
	&[[[E_{i-1,r}, E_{i,s}]_{q^{c^d_{ij}}} , E_{i+1,k}]_{q^{-c^d_{ij}}}, E_{i,l}] + [[[E_{i-1,r}, E_{i,l}]_{q^{c^d_{ij}}} , E_{i+1,k}]_{q^{-c^d_{ij}}}, E_{i,s}] = 0, \nonumber \quad \\
	&[[[F_{i-1,r}, F_{i,s}]_{q^{-c^d_{ij}}} , F_{i+1,k}]_{q^{c^d_{ij}}}, F_{i,l}] + [[[F_{i-1,r}, F_{i,l}]_{q^{-c^d_{ij}}} , F_{i+1,k}]_{q^{c^d_{ij}}}, F_{i,s}] = 0.
	\label{QL8}
\end{eqnarray}

Here the elements  $\psi_{i,r}, \varphi_{i,r}$ are defined by the following formulas:
\begin{equation}
	\psi_i(z) = \sum_{r \geq 0} \psi_{i,r}z^{-r} = \exp(\frac{\hbar}{2}H_{i,0}) \exp((q - q^{-1})\sum_{s \geq 1} H_{i,s} z^{-s}), \label{eq:qlelpsi}
\end{equation}
\begin{equation}
	\varphi_i(z) = \sum_{r \geq 0} \varphi_{i,-r}z^{r} = \exp(-\frac{\hbar}{2}H_{i,0}) \exp(-(q - q^{-1})\sum_{s \geq 1} H_{i,-s} z^{s}), \label{eq:qlelphi}
\end{equation}
with  $\psi_{i,-k} = \varphi_{i,k} = 0$ for every $k \geq 1$.

Also $[a,b]_{x} = ab - (-1)^{|a||b|} x b a$ for homogeneous $a$ and $b$.

\begin{remark}
	\begin{enumerate}
		\item Relations \eqref{QL4} are equivalent to
		\[ [ E_{i,k+1} , E_{j,l} ]_{q^{c_{ij}^d}} = - [ E_{j,l+1} , E_{i,k} ]_{q^{c_{ij}^{d}}}, \]
		and
		\[ [ F_{i,k+1} , F_{j,l} ]_{q^{-c_{ij}^d}} = - [ F_{j,l+1} , F_{i,k} ]_{q^{-c_{ij}^{d}}} \]
		respectively. If $c_{ij}^{d} = 0$ then the equations follow by \eqref{QL6}.
		\item Relations \eqref{QL7} are also valid for $|\alpha_{i,d}| = \bar{1}$, but in that case, they already follow from \eqref{QL6}.
		\item Relations \eqref{QL8} are also valid for $|\alpha_{i,d}| = \bar{0}$, but in that case, they already follow from eqs. \eqref{QL4}, \eqref{QL6} and \eqref{QL7}.
	\end{enumerate}
\end{remark}

Let $\widetilde{U}^{0}_{\hbar}(L\mathfrak{g}_d)$ be the polynomial super ring on the generators $ \{ H_{ir} \}_{i \in I, r \in \mathbb{Z}}$. Let $ U_{ \hbar}(L \mathfrak{b}_{+}) $ and $ U_{ \hbar}(L \mathfrak{b}_{-}) $ be the subalgebras generated respectively by $ \{ E_{ir} , H_{ir} \}_{i \in I, r \in \mathbb{Z}} $ and $ \{ F_{ir} , H_{ir} \}_{i \in I, r \in \mathbb{Z}} $. Denote by $U_{\hbar}(L \mathfrak{sl}(2|1)^{i})$ the subalgebra generated by $\{ E_{ir}, F_{ir}, H_{ir} \}_{r \in \mathbb{Z}}$ for a fixed $i \in I$.

Recall that $L \mathfrak{g} = \mathfrak{g} \otimes \Bbbk [z, z^{-1}]$. Let $\mathcal{J} \subset U_{\hbar}( L \mathfrak{g} )$ be the kernel of the composition
\[ U_{\hbar}( L \mathfrak{g} ) \xrightarrow{\hbar \to 0} U( L \mathfrak{g} ) \xrightarrow{z \to 1} U(\mathfrak{g}) \]
and let
\[ \widehat{ U_{\hbar}( L \mathfrak{g} ) } = \mathop{\lim_{\longleftarrow}}_{ n \in \mathbb{N} } U_{\hbar} ( L \mathfrak{g} ) / \mathcal{J}^{n} \]
be the completion of $U_{\hbar}( L \mathfrak{g} )$ with respect to the ideal $\mathcal{J}$.

\begin{proposition}
	$U_{\hbar}(L\mathfrak{g}_d)$ is generated by elements $\{ H_{i,0}, H_{i,\pm 1} , E_{i,0}, F_{i,0} \}_{i, \in I}$.
\end{proposition}
\begin{proof}
	Note that using $H_{i,\pm 1}, E_{i,0}, F_{i,0}$ and equations \eqref{QL3} we can easy build elements $\{ E_{i,r}, F_{i,r} \}_{i \in I}^{r \in \mathbb{Z} \setminus \{0\}}$.
	
	Consider the equation \eqref{eq:qlelpsi}. We have
	\[ \sum_{r \geq 0} \psi_{i,r} z^{-r} = \exp(\frac{\hbar}{2}H_{i,0}) \exp((q - q^{-1})\sum_{s \geq 1} H_{i,s} z^{-s}) = \]
	\[ \exp(\frac{\hbar}{2}H_{i,0}) \sum_{n \ge 0} \frac{(q-q^{-1})^{n}}{n!} ( \sum_{s \ge 1} H_{is} z^{-s} )^{n} = \]
	\[ \exp(\frac{\hbar}{2}H_{i,0}) ( 1 + \sum_{n \ge 1} \frac{(q-q^{-1})^{n}}{n!} \sum_{m \ge n} \sum_{ \substack{ 1 \le r_{1}, r_{2}, ... , r_{n} \le m, \\ r_{1} + r_{2} + ... + r_{n} = m } } H_{i,r_{1}} H_{i,r_2} ... H_{i,r_n} z^{-m} ) = \]
	\[ \exp(\frac{\hbar}{2}H_{i,0}) ( 1 + \sum_{n \ge 1} \frac{(q-q^{-1})^{n}}{n!} \sum_{m \ge n} \sum_{ \substack{ 1 \le r_{1} \le r_{2} \le ... \le r_{n} \le m, \\ r_{1} + r_{2} + ... + r_{n} = m } } C(r_1, r_2, ... , r_{n}) H_{i,r_{1}} H_{i,r_2} ... H_{i,r_n} z^{-m} ) = \]
	\[ \exp(\frac{\hbar}{2}H_{i,0}) ( 1 + \sum_{n \ge 1} \sum_{ \substack{ \text{partition of } n \\ \lambda = \{ \lambda_{1} , \lambda_{2} , ... , \lambda_{m} \} , \\ \lambda_{1} \ge \lambda_{2} 
			\ge ... \ge \lambda_{m} \ge 1 } } \frac{(q-q^{-1})^{m}}{m!} C(\lambda_{1} , \lambda_{2} , ... , \lambda_{m} ) H_{i,\lambda_{1}} , H_{i,\lambda_{2}} , ... , H_{i,\lambda_{m}} z^{-n} ), \]
	where $C(r_1, r_2, ... , r_{n}) \in \mathbb{N}$ for $r_1, r_{2}, ... , r_{n} \in \mathbb{N}$. Thus
	\[ \psi_{i,0} = \exp(\frac{\hbar}{2}H_{i,0}), \]
	\[ \psi_{i,r} = \exp(\frac{\hbar}{2}H_{i,0}) \sum_{ \substack{ \text{partition of } r \\ \lambda = \{ \lambda_{1} , \lambda_{2} , ... , \lambda_{m} \} , \\ \lambda_{1} \ge \lambda_{2} \ge ... \ge \lambda_{m} \ge 1 } } \frac{(q-q^{-1})^{m}}{m!} C(\lambda_{1} , \lambda_{2} , ... , \lambda_{m} ) H_{i,\lambda_{1}} , H_{i,\lambda_{2}} , ... , H_{i,\lambda_{m}} \]
	for all $r \in \mathbb{N}$.
	
	Using the equation \eqref{eq:qlelphi} we get by analogy that
	\[ \varphi_{i,0} = \exp(-\frac{\hbar}{2}H_{i,0}), \]
	\[ \varphi_{i,-r} = \exp(-\frac{\hbar}{2}H_{i,0}) \sum_{ \substack{ \text{partition of } r \\ \lambda = \{ \lambda_{1} , \lambda_{2} , ... , \lambda_{m} \} , \\ \lambda_{1} \ge \lambda_{2} \ge ... \ge \lambda_{m} \ge 1 } } (-1)^{m} \frac{ (q-q^{-1})^{m}}{m!} C(\lambda_{1} , \lambda_{2} , ... , \lambda_{m} ) H_{i,-\lambda_{1}} , H_{i,-\lambda_{2}} , ... , H_{i,-\lambda_{m}} \]
	for all $r \in \mathbb{N}$.
	
	Recall the equation \eqref{QL5}. Thus for all $k \in \mathbb{N}$
	\[ [ E_{i,k} , F_{i,0} ] =  \frac{\psi_{i,k} - \varphi_{i,0} }{q- q^{-1}} \iff \]
	\[ \psi_{i,k} = (q- q^{-1}) [ E_{i,k} , F_{i,0} ] + \exp(-\frac{\hbar}{2}H_{i,0}) \iff \]
	\[ H_{i,k} =  \exp(-\frac{\hbar}{2}H_{i,0}) [ E_{i,k} , F_{i,0} ] + \frac{1}{q-q^{-1}} \exp(- \hbar H_{i,0}) + \]
	\[ (-1) \sum_{ \substack{ \text{partition of } k \\ \lambda = \{ \lambda_{1} , \lambda_{2} , ... , \lambda_{m} \} , \\ k > \lambda_{1} \ge \lambda_{2} \ge ... \ge \lambda_{m} \ge 1 } } \frac{(q-q^{-1})^{m-1}}{m!} C(\lambda_{1} , \lambda_{2} , ... , \lambda_{m} ) H_{i,\lambda_{1}} , H_{i,\lambda_{2}} , ... , H_{i,\lambda_{m}}. \]
	In this way we are able to build elements $\{ H_{i,r} \}_{i \in I}^{r \ge 2}$.
	
	On the other hand we have for all $l \in \mathbb{N}$
	\[ [ E_{i,0} , F_{i,l} ] =  \frac{\psi_{i,0} - \varphi_{i,l} }{q- q^{-1}} \iff \]
	\[ \varphi_{i,l} = \exp(\frac{\hbar}{2}H_{i,0}) - (q- q^{-1}) [ E_{i,0} , F_{i,l} ] \iff \]
	\[ H_{i,-l} = \exp(\frac{\hbar}{2}H_{i,0}) [ E_{i,0} , F_{i,l} ] -  \frac{1}{q-q^{-1}} \exp(\hbar H_{i,0}) + \]
	\[  \sum_{ \substack{ \text{partition of } l \\ \lambda = \{ \lambda_{1} , \lambda_{2} , ... , \lambda_{m} \} , \\ l > \lambda_{1} \ge \lambda_{2} \ge ... \ge \lambda_{m} \ge 1 } } (-1)^{m} \frac{(q-q^{-1})^{m-1}}{m!} C(\lambda_{1} , \lambda_{2} , ... , \lambda_{m} ) H_{i,-\lambda_{1}} , H_{i,-\lambda_{2}} , ... , H_{i,-\lambda_{m}}. \]
	In this way we are able to build elements $\{ H_{i,-r} \}_{i \in I}^{r \ge 2}$. The result follows.
\end{proof}

\vspace{1cm}

\section{Comultplication on Yangians and their completion}

In this Section we explicitly describe Hopf superalgebra structure on Yangians and their completions.

\subsection{Hopf superalgebra structure on Yangians}
\label{sub:hssy}

Let
\[ \Omega^{+} = \sum_{k \in I} h^{(k)} \otimes h^{(k)} + \sum_{ \alpha \in \Delta^{+}} x_{\alpha}^{-} \otimes x_{\alpha}^{+} \]
be the "half" of the Casimir operator $\Omega$ \eqref{eq:CE}. Let $\square$ be the operator defined by $\square(X) = X \otimes 1 + 1 \otimes X$. It satisfies $\square([X, Y]) = [\square(X), \square(Y)]$ for all $X, Y \in Y_{\hbar}(\mathfrak{g})$.

Note the following properties of $\Omega^{+}$ (\cite[Lemma 4.2]{GNW18} in not super case)
\begin{lemma}
	We have
	\[ [ \square(h_{i0}) , \Omega^{+} ] = 0, \]
	\[ [ \square(x_{i0}^{+}) , \Omega^{+} ] = - x_{i0}^{+} \otimes h_{i0},  \]
	\[ [ \square(x_{i0}^{-}) , \Omega^{+} ] = h_{i0} \otimes x_{i0}^{-}, \]
	for all $i \in I$.
\end{lemma}
\begin{proof}
	The first formula is a simple consequence of the definition:
	\[ [ \square(h_{i0}) , \Omega^{+} ] = \sum_{ \alpha \in \Delta^{+}} (-1) (\alpha_{i}, \alpha) x_{\alpha}^{-} \otimes x_{\alpha}^{+} + \sum_{ \alpha \in \Delta^{+}} (\alpha_{i}, \alpha) x_{\alpha}^{-} \otimes x_{\alpha}^{+} = 0. \]
	
	For the next formulas we use assumptions made in the end of Subsection \ref{subs:LS} and \cite[Lemmas 1.3, 2.4]{K90}. First consider the case $\pm= +$. Then
	\[ [ x_{i0}^{+} \otimes 1 , \Omega^{+} ] = \sum_{k \in I} [ x_{i0}^{+} , h^{(k)} ] \otimes h^{(k)} + \sum_{ \alpha \in \Delta^{+}} [ x_{i0}^{+} , x_{\alpha}^{-} ] \otimes x_{\alpha}^{+} = \]
	\[ - \sum_{k \in I} \langle h^{(k)} , h_{i0} \rangle x_{i0}^{+} \otimes h^{(k)} + h_{i0} \otimes x_{i0}^{+} +  \sum_{ \alpha \in \Delta^{+} \setminus \{ \alpha_{i} \} }  [ x_{i0}^{+} , x_{\alpha}^{-} ] \otimes x_{\alpha}^{+} = \]
	\[ - x_{i0}^{+} \otimes h_{i0} + h_{i0} \otimes x_{i0}^{+} - \sum_{ \alpha \in \Delta^{+} \setminus \{ \alpha_{i} \} } (-1)^{|\alpha - \alpha_{i}||\alpha_i|} [ x_{\alpha - \alpha_{i}}^{-} , x_{i0}^{+} ] \otimes x_{\alpha - \alpha_{i}}^{+} = \]
	\[ - x_{i0}^{+} \otimes h_{i0} + h_{i0} \otimes x_{i0}^{+} - \sum_{ \alpha \in \Delta^{+} \setminus \{ \alpha_{i} \} } (-1)^{|\alpha||\alpha_{i}|} x_{\alpha}^{-} \otimes [ x_{i0}^{+} , x_{\alpha}^{+} ] = \]
	\[ - x_{i0}^{+} \otimes h_{i0} - \sum_{k \in I} h^{(k)} \otimes [ x_{i0}^{+} , h^{(k)} ]  - \sum_{ \alpha \in \Delta^{+} } (-1)^{|\alpha||\alpha_{i}|} x_{\alpha}^{-} \otimes [ x_{i0}^{+} , x_{\alpha}^{+} ] = \]
	\[ - x_{i0}^{+} \otimes h_{i0} - [  1 \otimes x_{i0}^{+}  , \Omega^{+} ]. \]
	
	By analogy for the case $\pm = -$, we have
	\[ [ 1 \otimes x_{i0}^{-} , \Omega^{+} ] = \sum_{k \in I} h^{(k)} \otimes [ x_{i0}^{-} , h^{(k)} ] + \sum_{ \alpha \in \Delta^{+}} (-1)^{|\alpha||\alpha_{i}|} x_{\alpha}^{-} \otimes [ x_{i0}^{-} , x_{\alpha}^{+} ] = \]
	\[ \sum_{k \in I} h^{(k)} \otimes \langle h^{(k)}, h_{i0} \rangle x_{i0}^{-} -  x_{i0}^{-} \otimes h_{i0} + \sum_{ \alpha \in \Delta^{+} \setminus \{ \alpha_{i} \}} (-1)^{|\alpha||\alpha_{i}|} x_{\alpha}^{-} \otimes [ x_{i0}^{-} , x_{\alpha}^{+} ] = \]
	\[ h_{i0} \otimes x_{i0}^{-} - x_{i0}^{-} \otimes h_{i0} + \sum_{ \alpha \in \Delta^{+} \setminus \{ \alpha_{i} \}} (-1)^{|\alpha-\alpha_{i}||\alpha_{i}|} [ x_{\alpha - \alpha_{i}}^{-} , x_{i0}^{-} ] \otimes x_{\alpha - \alpha_{i}}^{+} = \]
	\[  h_{i0} \otimes x_{i0}^{-} - \sum_{k \in I} [ x_{i0}^{-} , h^{(k)} ] \otimes h^{(k)} - \sum_{ \alpha \in \Delta^{+}} [ x_{i0}^{-} , x_{\alpha}^{-} ] \otimes x_{\alpha}^{+} = \]
	\[ h_{i0} \otimes x_{i0}^{-} - [  x_{i0}^{-} \otimes 1  , \Omega^{+} ]. \]
\end{proof}

The counit $\epsilon: Y_{\hbar}(\mathfrak{g}) \to \Bbbk$ is defined on generators in the following way:
\begin{equation}
	\label{eq:epsdef}
	\epsilon(1) = 1, \; \epsilon(x) = 0 \text{ for } x \in \{ h_{ir} , x_{ir}^{\pm} \}^{r \ge 0}_{i \in I}.
\end{equation}

The comultiplication is defined on generators $\{ h_{ir} , x_{i0}^{\pm} \}^{r \in \{0,1\}}_{i \in I}$ by the following formulas (see \cite{S13}, \cite{S04} and \cite{S07}):
\begin{equation}
	\label{eq:comg}
	\Delta(x) = \square(x) \text{ for } x \in \mathfrak{g},
\end{equation}
\begin{eqnarray}
	\label{eq:comh1}
	&\Delta(h_{i1}) = \square(h_{i1}) + \hbar(  h_{i0} \otimes h_{i0} + [ h_{i0} \otimes 1, \Omega^{+}] ) =  \nonumber \quad \\
	& \square(h_{i1}) + \hbar( h_{i0} \otimes h_{i0} - \sum_{ \alpha \in \Delta^{+}} (\alpha_{i}, \alpha) x_{\alpha}^{-} \otimes x_{\alpha}^{+} ).
\end{eqnarray}
It follows that
\begin{equation}
	\label{eq:comyh1}
	\Delta(\widetilde{h}_{i1}) = \square(\widetilde{h}_{i1}) + \hbar [ h_{i0} \otimes 1 , \Omega^{+} ].
\end{equation}

Now we need the isomorphism $\tau:  U( \mathfrak{g}[s] ) \to Y_{\hbar}( \mathfrak{g}) / \hbar Y_{\hbar}( \mathfrak{g}) $ from Definition \ref{df:isomzerogr}.  Note that our formulas are motivated by the equation \eqref{eq:squebialgsupcon} which is satisfied for \eqref{eq:comg} and \eqref{eq:comh1}. Indeed, for \eqref{eq:comg} it follows trivially from \eqref{eq:cosuperbr1}; for \eqref{eq:comh1} we have from \eqref{eq:cosuperbr2}
\[ \delta( \tau^{-1} ( h_{i1} )) = \delta(h_{i} u) = \sum_{\alpha \in \Delta^{+}} (-1)^{|\alpha|} (\alpha_{i}, \alpha) x_{\alpha}^+ \otimes x^{-}_{\alpha} - \sum_{\alpha \in \Delta^{+}} (\alpha_{i}, \alpha) x_{\alpha}^{-} \otimes x_{\alpha}^{+}. \]
On the other hand,
\[ \delta( \tau^{-1} ( h_{i1} )) = \delta(h_{i} u) = \hbar^{-1} ( \Delta(h_{i1}) - \Delta^{op}(h_{i1}) ) \pmod \hbar = \]
\[ [ h_{i0} \otimes 1 , \Omega^{+} ] - [ 1 \otimes h_{i0} , \tau (\Omega^{+}) ] = [ h_{i0} \otimes 1 , \Omega^{+} + \tau(\Omega^{+}) ]= \]
\[ [ h_{i0} \otimes 1 , \Omega + \sum_{k \in I} h^{(k)} \otimes h^{(k)} ] = [ h_{i0} \otimes 1 , \Omega ] = \]
\[ \sum_{\alpha \in \Delta^{+}} (-1)^{|\alpha|} (\alpha_{i}, \alpha) x_{\alpha}^+ \otimes x^{-}_{\alpha} - \sum_{\alpha \in \Delta^{+}} (\alpha_{i}, \alpha) x_{\alpha}^{-} \otimes x_{\alpha}^{+}. \]

The following result works for any Lie superalgebra satisfying the assumptions of Theorem \ref{th:mpy}.

\begin{theorem}
	\label{th:comhomdef}
	The operator $\Delta$ defines a superalgebra homomorphism $\Delta: Y_{\hbar}(\mathfrak{g}) \to Y_{\hbar}(\mathfrak{g}) \otimes Y_{\hbar}(\mathfrak{g})$.
\end{theorem}
\begin{proof}
	Obviously $\Delta$ is the linear $\mathbb{Z}_{2}$-graded even function by definition.
		
	It is enough to check the compatability of the relations listed in Theorem \ref{th:mpy}. First we deduce formulas for $\Delta(x_{i1}^{+})$ for all $i \in I$. From \eqref{eq:rec1} with $r=0$, we obtain:
	\[ \Delta(x_{i1}^{\pm}) = \pm (\alpha_{i}, \alpha_{i'})^{-1} [ \Delta(\widetilde{h}_{i' 1}) , \Delta(x_{i0}^{\pm})  ] = \]
	\[ \pm (\alpha_{i}, \alpha_{i'})^{-1} [ \square(\widetilde{h}_{i'1}) + \hbar [ h_{i'0} \otimes 1 , \Omega^{+} ] , \square( x_{i0}^{\pm} ) ] =  \]
	\[ \square( x_{i1}^{\pm} ) \pm (\alpha_{i}, \alpha_{i'})^{-1} \hbar [ [ h_{i' 0} \otimes 1 , \Omega^{+} ] , \square( x_{i0}^{\pm} ) ]. \]
	We consider the $+$ case first. Note that
	\[ [ [ h_{i' 0} \otimes 1 , \Omega^{+} ] , \square( x_{i0}^{+} ) ] = - [ [ 1 \otimes h_{i' 0} , \Omega^{+} ] , \square( x_{i0}^{+} ) ] = \]
	\[ - [ [ 1 \otimes h_{i' 0} , \square( x_{i0}^{+} ) ] , \Omega^{+} ] - [ 1 \otimes h_{i' 0} , [ \Omega^{+} , \square ( x_{i0}^{+} ) ] ] = \]
	\[ - ( \alpha_{i} , \alpha_{i'} ) [ 1 \otimes x_{i0}^{+} , \Omega^{+} ] - [ 1 \otimes h_{i' 0} , x_{i0}^{+} \otimes h_{i0} ] = - ( \alpha_{i} , \alpha_{i'} ) [ 1 \otimes x_{i0}^{+} , \Omega^{+} ]. \]
	Therefore
	\begin{equation}
		\Delta(x_{i1}^{+}) = \square( x_{i1}^{+} ) - \hbar [ 1 \otimes x_{i0}^{+} , \Omega^{+} ].
		\label{eq:comxbp}
	\end{equation}
	
	Similarly, we have
	\[ [ [ h_{i' 0} \otimes 1 , \Omega^{+} ] , \square( x_{i0}^{-} ) ] =  \]
	\[ [ h_{i' 0} \otimes 1 , [ \Omega^{+} , \square(x_{i0}^{-}) ] ] + [ \Omega^{+} , [ \square(x_{i0}^{-} ) , h_{i^{'} 0} \otimes 1 ] ] = \]
	\[ - [ h_{i' 0} \otimes 1 , h_{i0} \otimes x_{i0}^{-} ] -  (\alpha_{i'} , \alpha_{i}) [ x_{i0}^{-} \otimes 1 , \Omega^{+} ] = -  (\alpha_{i'} , \alpha_{i}) [ x_{i0}^{-} \otimes 1 , \Omega^{+} ]. \]
	Therefore
	\begin{equation}
		\Delta(x_{i1}^{-}) = \square( x_{i1}^{-} ) + \hbar [ x_{i0}^{-} \otimes 1 , \Omega^{+} ].
		\label{eq:comxbm}
	\end{equation}
	
	We do not need to check the relations involving  only $h_{i0}$ or $x_{i0}^{\pm}$. The compatability with relations \eqref{eq:mpy1} - \eqref{eq:mpy5} is proved in the same way as in \cite[Theorem 4.9: Part 1]{GNW18}. Let us consider \eqref{eq:mpy6} with $(r,s)=(1,0)$:
	\[ [ \Delta(h_{i1}) ,  \Delta(x_{i0}^{\pm}) ] = [ \square(h_{i1}) + \hbar(  h_{i0} \otimes h_{i0} + [ h_{i0} \otimes 1, \Omega^{+}] ) , \square( x_{i0}^{\pm} ) ] = \]
	\[ \square( [ h_{i1} , x_{i0}^{\pm} ] ) + \hbar ( [h_{i0} , x_{i0}^{\pm}] \otimes h_{i0} + h_{i0} \otimes [h_{i0} , x_{i0}^{\pm}] + [ [ h_{i0} \otimes 1, \Omega^{+}] , \square( x_{i0}^{\pm}) ] ) = \]
	\[ \hbar ( [ [ h_{i0} \otimes 1 , \square ( x_{i0}^{\pm} ) ] , \Omega^{+} ] + [ h_{i0} \otimes 1, [ \Omega^{+} ,  \square( x_{i0}^{\pm}) ] ] ) = \hbar [ h_{i0} \otimes 1, [ \Omega^{+} ,  \square( x_{i0}^{\pm}) ] ]. \]
	In the case $+$, the above is equal to
	\[ \hbar [ h_{i0} \otimes 1 , x_{i0}^{+} \otimes h_{i0} ] = \hbar [ h_{i0} , x_{i0}^{+} ] \otimes h_{i0} = 0. \]
	In the case $-$, we have
	\[ - \hbar [ h_{i0} \otimes 1, h_{i0} \otimes x_{i0}^{-} ] = 0. \]

	Consider \eqref{eq:mpy6} with $(r,s)=(1,1)$. Consider the case $+$:
	\[ [ \Delta(h_{i1}) ,  \Delta(x_{i1}^{+}) ] = [ \square(h_{i1}) + \hbar(  h_{i0} \otimes h_{i0} + [ h_{i0} \otimes 1, \Omega^{+}] ) , \square( x_{i1}^{+} ) - \hbar [ 1 \otimes x_{i0}^{+} , \Omega^{+} ] ] = \]
	\[ [ \square(h_{i1}) + \hbar h_{i0} \otimes h_{i0} , \square( x_{i1}^{+} ) - \hbar [ 1 \otimes x_{i0}^{+} , \Omega^{+} ] ] = \]
	\[ \square( [ h_{i1} , x_{i1}^{+} ] ) - \hbar^2 [ h_{i0} \otimes h_{i0} , [ 1 \otimes x_{i0}^{+} , \Omega^{+} ] ] = 0. \]
	Consider the case $-$:
	\[ [ \Delta(h_{i1}) ,  \Delta(x_{i1}^{-}) ] = [ \square(h_{i1}) + \hbar(  h_{i0} \otimes h_{i0} + [ h_{i0} \otimes 1, \Omega^{+}] ) ,  \square( x_{i1}^{-} ) + \hbar [ x_{i0}^{-} \otimes 1 , \Omega^{+} ] ] = \]
	\[ [ \square(h_{i1}) + \hbar  h_{i0} \otimes h_{i0}  ,  \square( x_{i1}^{-} ) + \hbar [ x_{i0}^{-} \otimes 1 , \Omega^{+} ] ] = \]
	\[ \square ( [ h_{i1} , x_{i1}^{-} ] ) + \hbar^2 [ h_{i0} \otimes h_{i0}  , [ x_{i0}^{-} \otimes 1 , \Omega^{+} ] ] = 0. \]
	
	Finally, we consider the relation \eqref{eq:mpy1} in the case $(r,s) = (1,1)$. It suffices to show that $ [ \widetilde{h}_{i1} , \widetilde{h}_{j1} ] = 0 $ for all $i, j \in I$. Since \eqref{eq:comyh1} and $ \square( [ \widetilde{h}_{i1} , \widetilde{h}_{j1} ] ) = 0 $, the latter relation reduces to
	\begin{equation}
		\label{eq:caryrelcom}
		\hbar^2 [ [ h_{i0} \otimes 1 , \Omega^{+} ] , [ h_{j0} , \Omega^{+} ] ] = \hbar ( [ \square( \widetilde{h}_{j1} ) , [ h_{i0} \otimes 1 , \Omega^{+} ] ] - [ \square( \widetilde{h}_{i1} ) , [ h_{j0} \otimes 1 , \Omega^{+} ] ] ).
	\end{equation}
	The left-hand side is the sum over $k \in \mathbb{N}$ of
	\[ \hbar^2 \sum_{ ht( \alpha + \beta ) = k } ( \alpha_{i} , \alpha ) ( \alpha_{j} , \beta ) [ x^{-}_{ \alpha } \otimes x^{+}_{ \alpha } , x_{ \beta}^{-} \otimes x_{ \beta}^{+} ] = \]
	\[ \frac{\hbar^2}{2} \sum_{ ht( \alpha + \beta ) = k } ( \alpha_{i} , \alpha ) ( \alpha_{j} , \beta ) (-1)^{ | \alpha| | \beta| } ( \{ x_{\alpha}^{-} , x_{\beta}^{-} \} \otimes [ x_{\alpha}^{+} , x_{\beta}^{+} ] + [ x_{\alpha}^{-} , x_{\beta}^{-} ] \otimes \{ x_{\alpha}^{+} , x_{\beta}^{+} \} ), \]
	where the sum $ \sum_{ ht( \alpha + \beta ) = k } $ is taken over all $ \alpha , \beta \in \Delta^{+} $ such that $ ht( \alpha + \beta ) = k  $. The right-hand side of \eqref{eq:caryrelcom} is the sum over $k \in \mathbb{N}$ of 
	\[ \hbar ( \sum_{ \alpha \in \Delta^{+}(k) } ( \alpha , \alpha_{j} ) [ \square ( \widetilde{h}_{i1} ) , x_{ \alpha}^{-} \otimes x_{ \alpha}^{+} ] - \sum_{ \alpha \in \Delta^{+}(k) } ( \alpha , \alpha_{i} ) [ \square ( \widetilde{h}_{j1} ) , x_{ \alpha}^{-} \otimes x_{ \alpha}^{+} ] ) , \]
	where $ \Delta^{+}(k) = \{ \alpha \in \Delta^{+} : ht( \alpha ) = k \}$. Therefore, \eqref{eq:caryrelcom} will hold if the following equalities are established for all $k \in \mathbb{N}$:
	\begin{equation}
		\label{eq:fccarcomf}
		\hbar \sum_{ ht( \alpha + \beta ) = k } (-1)^{ |\alpha| |\beta| } ( \alpha_{i} , \alpha ) ( \alpha_{j} , \beta ) \{ x_{ \alpha }^{-} , x_{ \beta }^{-} \} \otimes [ x_{ \alpha}^{+} , x_{ \beta}^{+} ] = 2  \sum_{ \alpha \in \Delta^{+}(k) } [ h_{ij} ( \alpha ) , x_{ \alpha}^{-} ] \otimes x_{ \alpha}^{+} ,
	\end{equation}	
	and
	\begin{equation}
		\label{eq:fccarcoms}
		\hbar \sum_{ ht( \alpha + \beta ) = k } (-1)^{ |\alpha| |\beta| } ( \alpha_{i} , \alpha ) ( \alpha_{j} , \beta ) [ x_{ \alpha }^{-} , x_{ \beta }^{-} ] \otimes \{ x_{ \alpha}^{+} , x_{ \beta}^{+} \} = 2  \sum_{ \alpha \in \Delta^{+}(k) } x_{ \alpha}^{-} \otimes [ h_{ij} ( \alpha ) , x_{ \alpha}^{+} ] ,
	\end{equation}	
	where $ h_{ij} ( \alpha ) = ( \alpha , \alpha_{j} ) \widetilde{h}_{i1} - ( \alpha , \alpha_{i} ) \widetilde{h}_{j1} $ for all $i, j \in I$ and $ \alpha \in \Delta^{+} $.
	
	Now we need the result proved in \cite[Proposition 3.21]{GNW18} (that is also true is super case as one may verify). Suppose that $ \alpha \in \Delta^{+}$. The result states that for all $k \in I$ we have 
	\[ [ J( h_k) , x_{ \alpha}^{ \pm} ] = [ h_k , J( x_{ \alpha }^{ \pm } ) ] = \pm ( \alpha_{k} , \alpha ) J( x_{ \alpha }^{ \pm } ) \]
	where $J(h_{k}) = \widetilde{h}_{k1} + \widetilde{v}_{k} $, $ \widetilde{v}_{k} = \frac{\hbar}{2} \sum_{ \alpha \in \Delta^{+} } ( \alpha , \alpha_{i} ) x_{\alpha}^{-} x_{\alpha}^{+}  $ and $ J( x_{ \alpha }^{ \pm } ) $ is an well-defined element in $Y_{\hbar}(\mathfrak{g})$. Consequently,
	\[ [ h_{ij} ( \alpha ) , x_{ \alpha }^{\mp} ] = [ ( \alpha , \alpha_{j} ) J(h_i) - ( \alpha , \alpha_{i} ) J( h_j ) , x_{ \alpha }^{ \mp } ] = [ x_{ \alpha}^{ \mp } , v_{ij} ( \alpha) ], \]
	where $ v_{ij} ( \alpha) = ( \alpha , \alpha_{j} ) \widetilde{v}_{i} - ( \alpha , \alpha_{i} ) \widetilde{v}_{j} $. We claim that
	\begin{equation}
		\label{eq:pva}
		[ v_{ij} ( \alpha) , x_{ \alpha}^{ - } ] = - \frac{\hbar}{2} \sum_{ \substack { \beta , \gamma \in \Delta^{+}, \\ \beta + \gamma = \alpha } } (-1)^{ |\beta| |\gamma| + |\alpha| ( |\beta| + |\gamma| ) } A_{ij, \beta , \gamma}  \langle x_{\alpha}^{ -} , [ x_{ \beta}^{ + } , x_{\gamma}^{ + } ]  \rangle x_{ \beta }^{ -} x_{ \gamma}^{ - },
	\end{equation}
	and
	\begin{equation}
		\label{eq:mva}
		[ v_{ij} ( \alpha) , x_{ \alpha}^{ + } ] = - \frac{\hbar}{2} \sum_{ \substack { \beta , \gamma \in \Delta^{+}, \\ \beta + \gamma = \alpha } } (-1)^{ |\beta| |\gamma| } A_{ij, \beta , \gamma}  \langle x_{\alpha}^{ +} , [ x_{ \beta}^{ - } , x_{\gamma}^{ - } ]  \rangle x_{ \beta }^{ +} x_{ \gamma}^{ + },
	\end{equation}
	where $A_{ij, \beta , \gamma} = ( \alpha_i , \beta ) ( \alpha_j , \gamma ) - ( \alpha_j , \beta ) ( \alpha_i , \gamma )$ for each $ \beta , \gamma \in \Delta^{+}$.
	
	First we prove \eqref{eq:pva}. By definition of $\widetilde{v}_{i}$, we have
	\[ [ v_{ij}( \alpha ) , x_{ \alpha}^{-} ] = \frac{\hbar}{2} \sum_{ \beta \in \Delta^{+}} A_{ij, \beta, \alpha} ( (-1)^{|\alpha| |\beta|} [ x_{ \beta}^{-} , x_{ \alpha}^{-} ] x_{ \beta}^{+} + x_{\beta}^{-} [ x_{\beta}^{+} , x_{\alpha}^{-} ]  ).  \]
	Consider $x_{\beta}^{-} [ x_{ \beta}^{+} , x_{\alpha}^{-} ]$. Assume first that $ \gamma = \beta - \alpha $ is positive, that is $ \gamma \in \Delta^{+}$. Then
	\[ x_{ \beta}^{-} [ x_{ \beta}^{+} , x_{ \alpha}^{-} ] = \langle [ x_{ \beta}^{+} , x_{ \alpha}^{-} ] , x_{ \gamma}^{-} \rangle x_{ \beta}^{-} x_{ \gamma}^{+} = (-1)^{ | \gamma| ( | \beta| + | \alpha| ) } \langle x_{ \gamma}^{-} , [ x_{ \beta}^{+} , x_{ \alpha}^{-} ] \rangle x_{ \beta}^{-} x_{ \gamma}^{+} . \] 
	If $\beta - \alpha$ is negative, we set instead $ \gamma = \alpha - \beta$ and deduce that
	\[ x_{ \beta}^{-} [ x_{ \beta}^{+} , x_{ \alpha}^{-} ] = \langle x_{ \gamma}^{+} , [ x_{ \beta}^{+} , x_{ \alpha}^{-} ] \rangle x_{ \beta}^{-} x_{ \gamma}^{-} = (-1)^{1+|\beta||\gamma|+|\alpha|( |\beta| + |\gamma| )} \langle x_{ \alpha}^{-} , [ x_{ \beta}^{+} , x_{ \gamma}^{+} ] \rangle x_{ \beta}^{-} x_{ \gamma}^{-}. \]
	We thus have
	\begin{equation}
		\label{eq:sumaijpmcase}
		\sum_{ \beta \in \Delta^{+}} A_{ij, \beta, \alpha} x_{\beta}^{-} [ x_{\beta}^{+} , x_{\alpha}^{-} ] = \sum_{ \substack { \beta, \gamma \in \Delta^{+} , \\ \gamma + \alpha = \beta } } A_{ij, \gamma , \beta} (-1)^{ | \gamma| ( | \beta| + | \alpha| ) } \langle x_{ \gamma}^{-} , [ x_{ \beta}^{+} , x_{ \alpha}^{-} ] \rangle x_{ \beta}^{-} x_{ \gamma}^{+} +  
	\end{equation}
    \[ (-1) \sum_{ \substack { \beta, \gamma \in \Delta^{+} , \\ \beta + \gamma = \alpha } } A_{ij, \beta , \gamma} (-1)^{|\beta||\gamma|+|\alpha|( |\beta| + |\gamma| )} \langle x_{ \alpha}^{-} , [ x_{ \beta}^{+} , x_{ \gamma}^{+} ] \rangle x_{ \beta}^{-} x_{ \gamma}^{-}, \]
    where we have used that $ A_{ij, \beta, \beta - \gamma} = A_{ij, \gamma, \beta} $ and $A_{ij, \beta, \beta + \gamma} = A_{ij, \beta, \gamma}$. Also we have
    \begin{equation}
    	\label{eq:sumaijpmcaset}
     \sum_{ \beta \in \Delta^{+} } A_{ij, \beta, \alpha} (-1)^{| \alpha| | \beta|} [ x_{ \beta}^{-} , x_{ \alpha}^{-} ] x_{ \beta}^{+} = \sum_{ \substack{ \beta, \gamma \in \Delta^{+} , \\ \beta + \alpha = \gamma } } A_{ij, \beta, \gamma - \beta} (-1)^{| \alpha| | \beta|} \langle x_{\gamma}^{+} , [ x_{ \beta}^{-} , x_{ \alpha}^{-} ] \rangle x_{\gamma}^{-} x_{ \beta}^{+} = 
	\end{equation}
    \[ \sum_{ \substack{ \beta, \gamma \in \Delta^{+} , \\ \gamma + \alpha = \beta } } A_{ij, \gamma, \beta} (-1)^{|\alpha| |\gamma|} \langle x_{\beta}^{+} , [ x_{ \gamma}^{-} , x_{ \alpha}^{-} ] \rangle x_{\beta}^{-} x_{\gamma}^{+} = \]
    \[ \sum_{ \substack{ \beta, \gamma \in \Delta^{+} , \\ \gamma + \alpha = \beta } } A_{ij, \gamma, \beta} (-1)^{1+ |\gamma| ( |\alpha| + |\beta| ) } \langle x_{\gamma}^{-} , [ x_{\beta}^{+} , x_{\alpha}^{-} ] \rangle x_{\beta}^{-} x_{\gamma}^{+}, \]
    where we have used that $A_{ij, \beta, \gamma - \beta} = A_{ij, \beta, \gamma}$. To obtain the second equation we have exchanged $\beta$ and $\gamma$. Thus if we add \eqref{eq:sumaijpmcase} and \eqref{eq:sumaijpmcaset} we prove \eqref{eq:pva}.
    
    Thus from \eqref{eq:pva} we have
    \[ 2 [ h_{ij} ( \alpha ) , x_{ \alpha}^{-} ] \otimes x_{ \alpha}^{+} = \hbar \sum_{ \substack { \beta , \gamma \in \Delta^{+}, \\ \beta + \gamma = \alpha } } (-1)^{ |\beta| |\gamma| + |\alpha| ( |\beta| + |\gamma| ) } A_{ij, \beta , \gamma}  \langle x_{\alpha}^{ -} , [ x_{ \beta}^{ + } , x_{\gamma}^{ + } ]  \rangle x_{ \beta }^{ -} x_{ \gamma}^{ - } \otimes x_{ \alpha}^{+} = \]
    \[ \hbar \sum_{ \substack { \beta , \gamma \in \Delta^{+}, \\ \beta + \gamma = \alpha } } (-1)^{ |\beta| |\gamma|  } A_{ij, \beta , \gamma}  x_{ \beta }^{ -} x_{ \gamma}^{ - } \otimes [ x_{ \beta}^{ + } , x_{\gamma}^{ + } ]. \]
    Then
    \begin{equation}
    	\label{eq:rhcpepm}
    	2 \sum_{ \alpha \in \Delta^{+}(k) } [ h_{ij} ( \alpha ) , x_{ \alpha}^{-} ] \otimes x_{ \alpha}^{+} = \frac{\hbar}{2} \sum_{ \alpha \in \Delta^{+}(k) } ( \sum_{ \substack { \beta , \gamma \in \Delta^{+}, \\ \beta + \gamma = \alpha } } (-1)^{ |\beta| |\gamma|  } A_{ij, \beta , \gamma}  x_{ \beta }^{ -} x_{ \gamma}^{ - } \otimes [ x_{ \beta}^{ + } , x_{\gamma}^{ + } ] +
    \end{equation}
    \[ \sum_{ \substack { \beta , \gamma \in \Delta^{+}, \\ \beta + \gamma = \alpha } } (-1)^{ |\beta| |\gamma|  } A_{ij, \beta , \gamma}  x_{ \gamma}^{ - } x_{ \beta }^{ -} \otimes [ x_{\gamma}^{ + } , x_{ \beta}^{ + } ] ) = \frac{\hbar}{2} \sum_{ \alpha \in \Delta^{+}(k) } \sum_{ \substack { \beta , \gamma \in \Delta^{+}, \\ \beta + \gamma = \alpha } } A_{ij, \beta , \gamma} \{  x_{ \gamma}^{ - } , x_{ \beta }^{ -} \} \otimes [ x_{ \beta}^{ + } , x_{\gamma}^{ + } ] . \]
    
    From the left-hand side of \eqref{eq:fccarcomf} we have
    \[ \frac{\hbar}{2} ( \sum_{ ht( \alpha + \beta ) = k } (-1)^{ |\alpha| |\beta| } ( \alpha_{i} , \alpha ) ( \alpha_{j} , \beta ) \{ x_{ \alpha }^{-} , x_{ \beta }^{-} \} \otimes [ x_{ \alpha}^{+} , x_{ \beta}^{+} ] + \]
    \[ \sum_{ ht( \alpha + \beta ) = k } (-1)^{ |\alpha| |\beta| } ( \alpha_{i} , \beta ) ( \alpha_{j} , \alpha ) \{ x_{ \beta }^{-} , x_{ \alpha }^{-} \} \otimes [ x_{ \beta}^{+} , x_{ \alpha}^{+} ] ) =  \]
    \[ \frac{\hbar}{2} \sum_{ ht( \alpha + \beta ) = k } A_{ij, \alpha, \beta} \{ x_{ \beta }^{-} , x_{ \alpha }^{-} \} \otimes [ x_{ \alpha}^{+} , x_{ \beta}^{+} ]. \]
    This coincides with the right-hand side of \eqref{eq:rhcpepm}. This completes the proof of the identity \eqref{eq:fccarcomf}.
    
    Now we prove \eqref{eq:mva}. By definition of $\widetilde{v}_{i}$, we have
    \[ [ v_{ij}( \alpha ) , x_{ \alpha}^{+} ] = \frac{\hbar}{2} \sum_{ \beta \in \Delta^{+}} A_{ij, \beta, \alpha} ( (-1)^{|\alpha| |\beta|} [ x_{ \beta}^{-} , x_{ \alpha}^{+} ] x_{ \beta}^{+} + x_{\beta}^{-} [ x_{\beta}^{+} , x_{\alpha}^{+} ]  ).  \]
    Consider $[ x_{ \beta}^{-} , x_{ \alpha}^{+} ] x_{ \beta}^{+}$. Assume first that $ \gamma = \alpha -\beta $ is positive, that is $ \gamma \in \Delta^{+}$. Then
    \[ [ x_{ \beta}^{-} , x_{ \alpha}^{+} ] x_{ \beta}^{+} = \langle [ x_{ \beta}^{-} , x_{ \alpha}^{+} ] , x_{\gamma}^{-} \rangle x_{\gamma}^{+} x_{ \beta}^{+} = (-1)^{1+|\alpha||\beta|} \langle x_{ \alpha}^{+} , [ x_{\beta}^{-} , x_{\gamma}^{-} ] \rangle x_{\gamma}^{+} x_{ \beta}^{+}. \]
    If $\alpha -\beta$ is negative, we set instead $ \gamma = \beta - \alpha $ and deduce that
    \[ [ x_{ \beta}^{-} , x_{ \alpha}^{+} ] x_{ \beta}^{+} = (-1)^{1+|\alpha||\beta|} \langle x_{\gamma}^{+} ,  [ x_{ \alpha}^{+} , x_{ \beta}^{-} ] \rangle x_{\gamma}^{-} x_{ \beta}^{+} = (-1)^{1+|\alpha||\beta|} \langle [ x_{\gamma}^{+} , x_{\alpha}^{+} ] , x_{\beta}^{-} \rangle x_{\gamma}^{-} x_{ \beta}^{+}. \] 
    We thus have
    \begin{equation}
    	\label{eq:sumaijmpcase}
    	\sum_{ \beta \in \Delta^{+}} A_{ij, \beta, \alpha} (-1)^{|\alpha| |\beta|} [ x_{ \beta}^{-} , x_{ \alpha}^{+} ] x_{ \beta}^{+} = (-1) \sum_{ \substack { \beta, \gamma \in \Delta^{+} , \\ \gamma + \beta = \alpha } } A_{ij, \beta, \gamma} \langle x_{ \alpha}^{+} , [ x_{\beta}^{-} , x_{\gamma}^{-} ] \rangle x_{\gamma}^{+} x_{ \beta}^{+} + 
    \end{equation}
    \[ (-1) \sum_{ \substack { \beta, \gamma \in \Delta^{+} , \\ \gamma + \alpha = \beta } } A_{ij, \gamma, \beta} \langle [ x_{\gamma}^{+} , x_{\alpha}^{+} ] , x_{\beta}^{-} \rangle x_{\gamma}^{-} x_{ \beta}^{+} = \]
    \[ (-1) \sum_{ \substack { \beta, \gamma \in \Delta^{+} , \\ \gamma + \beta = \alpha } }  A_{ij, \beta, \gamma} (-1)^{|\beta| |\gamma|} \langle x_{ \alpha}^{+} , [ x_{\beta}^{-} , x_{\gamma}^{-} ] \rangle x_{\beta}^{+} x_{ \gamma}^{+} - \sum_{ \substack { \beta, \gamma \in \Delta^{+} , \\ \gamma + \alpha = \beta } } A_{ij, \gamma, \beta} \langle [ x_{\gamma}^{+} , x_{\alpha}^{+} ] , x_{\beta}^{-} \rangle x_{\gamma}^{-} x_{ \beta}^{+}, \]
    where we have used that $A_{ij, \gamma, \beta} = - A_{ij, \beta, \gamma}$. To obtain the second equation we have exchanged $\beta$ and $\gamma$. Also we have as in the previous case
    \begin{equation}
    	\label{eq:sumaijmpcaset}
    	\sum_{ \beta \in \Delta^{+}} A_{ij, \beta, \alpha} x_{\beta}^{-} [ x_{\beta}^{+} , x_{\alpha}^{+} ] = \sum_{ \substack { \beta, \gamma \in \Delta^{+} , \\ \alpha + \beta = \gamma } } A_{ij, \beta, \gamma - \beta} \langle [ x_{\beta}^{+} , x_{\alpha}^{+} ] , x_{\gamma}^{-} \rangle x_{\beta}^{-} x_{\gamma}^{+} = 
    \end{equation}
    \[ \sum_{ \substack { \beta, \gamma \in \Delta^{+} , \\ \alpha + \gamma = \beta } } A_{ij, \gamma, \beta} \langle [ x_{\gamma}^{+} , x_{\alpha}^{+} ] , x_{\beta}^{-} \rangle x_{\gamma}^{-} x_{\beta}^{+}.  \]
    Thus if we add \eqref{eq:sumaijmpcase} and \eqref{eq:sumaijmpcaset} we prove \eqref{eq:mva}.
    
    Thus from \eqref{eq:mva} we have
    \[ 2 x_{ \alpha}^{-} \otimes [ h_{ij} ( \alpha ) , x_{ \alpha}^{+} ] = \hbar \sum_{ \substack { \beta , \gamma \in \Delta^{+}, \\ \beta + \gamma = \alpha } } x_{ \alpha}^{-} \otimes (-1)^{ |\beta| |\gamma| } A_{ij, \beta , \gamma}  \langle x_{\alpha}^{ +} , [ x_{ \beta}^{ - } , x_{\gamma}^{ - } ]  \rangle x_{ \beta }^{ +} x_{ \gamma}^{ + } = \]
    \[ \hbar \sum_{ \substack { \beta , \gamma \in \Delta^{+}, \\ \beta + \gamma = \alpha } } (-1)^{ |\beta| |\gamma| } A_{ij, \beta , \gamma} [ x_{ \beta}^{ - } , x_{\gamma}^{ - } ] \otimes x_{ \beta }^{ +} x_{ \gamma}^{ + }. \]
    Then
    \begin{equation}
    	\label{eq:rhcpepms}
    	2 \sum_{ \alpha \in \Delta^{+}(k) } x_{ \alpha}^{-} \otimes [ h_{ij} ( \alpha ) , x_{ \alpha}^{+} ] = \frac{\hbar}{2} \sum_{ \alpha \in \Delta^{+}(k) } ( \sum_{ \substack { \beta , \gamma \in \Delta^{+}, \\ \beta + \gamma = \alpha } } (-1)^{ |\beta| |\gamma| } A_{ij, \beta , \gamma} [ x_{ \beta}^{ - } , x_{\gamma}^{ - } ] \otimes x_{ \beta }^{ +} x_{ \gamma}^{ + } +
    \end{equation}
    \[ \sum_{ \substack { \beta , \gamma \in \Delta^{+}, \\ \beta + \gamma = \alpha } } (-1)^{ |\beta| |\gamma| } A_{ij, \gamma , \beta} [ x_{ \gamma}^{ - } , x_{\beta}^{ - } ] \otimes x_{ \gamma }^{ +} x_{ \beta}^{ + } ) = \frac{\hbar}{2} \sum_{ \alpha \in \Delta^{+}(k) } \sum_{ \substack { \beta , \gamma \in \Delta^{+}, \\ \beta + \gamma = \alpha } } A_{ij, \beta , \gamma} [ x_{ \beta}^{ - } , x_{\gamma}^{ - } ] \otimes \{ x_{ \gamma }^{ +} , x_{ \beta}^{ + } \}. \]
    
    From the left-hand side of \eqref{eq:fccarcoms} we have
    \[ \frac{\hbar}{2} ( \sum_{ ht( \alpha + \beta ) = k } (-1)^{ |\alpha| |\beta| } ( \alpha_{i} , \alpha ) ( \alpha_{j} , \beta ) [ x_{ \alpha }^{-} , x_{ \beta }^{-} ] \otimes \{ x_{ \alpha}^{+} , x_{ \beta}^{+} \} +  \]
    \[ \sum_{ ht( \alpha + \beta ) = k } (-1)^{ |\alpha| |\beta| } ( \alpha_{i} , \beta ) ( \alpha_{j} , \alpha ) [ x_{ \beta }^{-} , x_{ \alpha }^{-} ] \otimes \{ x_{ \beta}^{+} , x_{ \alpha}^{+} \} ) = \]
    \[ \frac{\hbar}{2} \sum_{ ht( \alpha + \beta ) = k } A_{ij, \alpha, \beta} [ x_{ \alpha }^{-} , x_{ \beta }^{-} ] \otimes \{ x_{ \beta}^{+} , x_{ \alpha}^{+} \} . \]
    This coincides with the right-hand side of \eqref{eq:rhcpepms}. This completes the proof of the identity \eqref{eq:fccarcoms}. Thus we have proved the theorem.
    
\end{proof}

\begin{remark}
	We need again the isomorphism $\tau:  U( \mathfrak{g}[s] ) \to Y_{\hbar}( \mathfrak{g}) / \hbar Y_{\hbar}( \mathfrak{g}) $ from Definition \ref{df:isomzerogr}. Note that by \eqref{eq:cosuperbr2}
	\[ \delta( \tau^{-1}( x_{i1}^{\pm} ) ) = \delta(x_{i0}^{\pm} u) = [ x_{i0}^{\pm} \otimes 1 , \Omega ]. \]
	On the other hand,
	\[ \delta( \tau^{-1}( x_{i1}^{+} ) ) = \delta(x_{i0}^{+} u) = \hbar^{-1} ( \Delta(x_{i1}^{+}) - \Delta^{op}(x_{i1}^{+}) ) \pmod \hbar = \]
	\[ -[ 1 \otimes x_{i0}^{+} , \Omega^{+} ] + [ x_{i0}^{+} \otimes 1 , \tau( \Omega^{+} ) ] = \]
	\[ [ x_{i0}^{+} \otimes 1 , \Omega^{+} ] + x_{i0}^{+} \otimes h_{i0} + [ x_{i0}^{+} \otimes 1, \tau( \Omega^{+} ) ] = \]
	\[ [ x_{i0}^{+} \otimes 1 , \Omega + \sum_{k \in I} h^{(k)} \otimes h^{(k)} ] + x_{i0}^{+} \otimes h_{i0} = \]
	\[ [ x_{i0}^{+} \otimes 1 , \Omega ] +  x_{i0}^{+} \otimes h_{i0} - \sum_{k \in I} \langle h^{(k)} , h_{i0} \rangle x_{i0}^{+} \otimes h^{(k)} = \]
	\[ [ x_{i0}^{+} \otimes 1 , \Omega ] +  x_{i0}^{+} \otimes h_{i0} -  x_{i0}^{+} \otimes h_{i0} = [ x_{i0}^{+} \otimes 1 , \Omega ]. \]
	Similarly,
	\[ \delta( \tau^{-1}( x_{i1}^{-} ) ) =  \delta(x_{i0}^{-} u) = \hbar^{-1} ( \Delta(x_{i1}^{-}) - \Delta^{op}(x_{i1}^{-}) ) \pmod \hbar = \]
	\[ [ x_{i0}^{-} \otimes 1 , \Omega^{+} ] - [ 1 \otimes x_{i0}^{-} , \tau (\Omega^{+}) ] = \]
	\[ [ x_{i0}^{-} \otimes 1 , \Omega^{+} + \tau( \Omega^{+} ) ] - x_{i0}^{-} \otimes h_{i0} = \]
	\[ [ x_{i0}^{-} \otimes 1 , \Omega] + \sum_{ k \in I } \langle h^{(k)} , h_{i0} \rangle x_{i0}^{-} \otimes h^{(k)} - x_{i0}^{-} \otimes h_{i0} = \]
	\[ [ x_{i0}^{-} \otimes 1 , \Omega] + x_{i0}^{-} \otimes h_{i0} - x_{i0}^{-} \otimes h_{i0} = [ x_{i0}^{-} \otimes 1 , \Omega]. \]
\end{remark}

In order to define the comultiplication on generators $\{ h_{ir} , x_{i1}^{\pm} , x_{ir}^{\pm} \}^{r \ge 2}_{i \in I}$ we use recurrent formulas \eqref{eq:rec1} and \eqref{eq:rec2}. Thus
\begin{equation}
	\label{eq:comrec1}
	\Delta(x_{i,r+1}^{\pm}) = \pm (c_{ij})^{-1} [ \Delta(h_{j1}) - \frac{\hbar}{2} \Delta^2(h_{j0}) , \Delta(x_{ir}^{\pm}) ],
\end{equation}
where $r \ge 0$; $i=j$, if $|\alpha_{i}|=\bar{0}$ and $|i-j|=1$, if $|\alpha_{i}|=\bar{1}$ ($i,j \in I$);
\begin{equation}
	\label{eq:comrec2}
	\Delta(h_{ir}) = [ \Delta(x_{ir}^{+}), \Delta(x_{i0}^{-}) ],
\end{equation}
where $r \ge 2$ and $i \in I$.

We need to show that $\Delta$ is indeed the comultiplication
\begin{lemma}
	\label{lm:comcoass}
	The operator $\Delta: Y_{\hbar}(\mathfrak{g}) \to Y_{\hbar}(\mathfrak{g}) \otimes Y_{\hbar}(\mathfrak{g})$ defines the comultiplication on $Y_{\hbar}(\mathfrak{g})$.
\end{lemma}
\begin{proof}
	Since $Y_{\hbar}(\mathfrak{g})$ is generated by $\mathfrak{g}$ and $h_{i1}$ we need only to verify that defining equations of a comultiplication are satisfied for $h_{i1}$. Indeed, this follows by Theorem \ref{th:comhomdef}. Thus we have
	\[ ( id \otimes \Delta ) \Delta ( h_{i1} ) = ( id \otimes \Delta ) ( \square(h_{i1}) +  \hbar(  h_{i0} \otimes h_{i0} + [ h_{i0} \otimes 1, \Omega^{+}] ) ) = \] 
	\[ h_{i1} \otimes 1 \otimes 1 + 1 \otimes h_{i1} \otimes 1 +  1 \otimes 1 \otimes h_{i1} + ( id \otimes \Delta ) ( \hbar(  h_{i0} \otimes h_{i0} + [ h_{i0} \otimes 1, \Omega^{+}] ) ) = \]
	\[ ( \Delta \otimes id ) ( \square(h_{i1}) ) + ( \Delta \otimes id ) ( \hbar(  h_{i0} \otimes h_{i0} + [ h_{i0} \otimes 1, \Omega^{+}] ) ) = \]
	\[ ( \Delta \otimes id ) \Delta ( h_{i1} ). \]
	Moreover,
	\[ ( id \otimes \epsilon ) \Delta ( h_{i1} ) = h_{i1} \otimes 1, \;  ( \epsilon \otimes id ) \Delta ( h_{i1} ) = 1 \otimes h_{i1}. \]
	The result follows.
\end{proof}

The antipode $S$ is defined on generators $\{ h_{ir} , x_{ir}^{\pm} \}^{r \in \{0,1\}}_{i \in I}$ in the following way:
\begin{equation}
	\label{eq:antipodeY1}
	S(x) = -x \text{\; for \;} x \in \mathfrak{g},
\end{equation}
\begin{equation}
	S(h_{i1}) =	-h_{i1} + \hbar ( h_{i0}^2 + \sum_{ \alpha \in \Delta^{+}} (-1)^{1+|\alpha|} (\alpha_{i}, \alpha) x_{\alpha}^{-} x_{\alpha}^{+} ).
\end{equation}
Indeed, the result follows from the relation $\mu \circ (S \otimes \epsilon) \circ \Delta = \mu \circ (\epsilon \otimes S) \circ \Delta = \nu \circ \epsilon$, formulas \eqref{eq:comg}, \eqref{eq:comh1} and Theorem \ref{th:comhomdef}. Here $\mu : Y_{\hbar}(\mathfrak{g}) \otimes Y_{\hbar}(\mathfrak{g}) \to Y_{\hbar}(\mathfrak{g}) $ is the multiplication operator and $\nu : \Bbbk \to Z(Y_{\hbar}(\mathfrak{g}))$ is a ring homomorphism in the center of $Y_{\hbar}(\mathfrak{g})$.

In order to define the antipode on generators $\{ h_{ir} , x_{i1}^{\pm} , x_{ir}^{\pm} \}^{r \ge 2}_{i \in I}$ we use recurrent formulas \eqref{eq:rec1} and \eqref{eq:rec2}. Thus
\begin{equation}
	S(x_{i,r+1}^{\pm}) = \mp (c_{ij})^{-1} [ S(h_{j1}) - \frac{\hbar}{2} S^2(h_{j0}) , S(x_{ir}^{\pm}) ],
\end{equation}
where $r \ge 0$; $i=j$ if $|\alpha_{i}|=\bar{0}$ and $|i-j|=1$ if $|\alpha_{i}|=\bar{1}$ ($i,j \in I$);
\begin{equation}
	\label{eq:antipodeY2}
	S(h_{ir}) = - [ S(x_{ir}^{+}), S(x_{i0}^{-}) ],
\end{equation}
where $r \ge 2$ and $i \in I$.

\vspace{1cm}

\subsection{Hopf superalgebra structure on completion of Yangian}

\label{lb:HSSCY}

Recall Subsection \ref{sub:compl}.

$Y_{\hbar}(\mathfrak{g})$ is the $\mathbb{N}_{0}$-graded superalgebra where we set $\text{deg}(h_{i,r})= \text{deg}(x_{i,r}^{{\pm}}) =r$ for all $r \in \mathbb{N}_{0}$ and $\text{deg}(\hbar)=1$. Then $Y_{\hbar}(\mathfrak{g}) = \bigoplus_{n \in \mathbb{N}_{0}} Y_{\hbar}(\mathfrak{g})_{n}$, where $Y_{\hbar}(\mathfrak{g})_{0} = \Bbbk$ and $Y_{\hbar}(\mathfrak{g})_{n}$ is a subsuperalgebra generated by elements of degree $n$. Then we have a descending filtration of two-sided $\mathbb{Z}_2$-graded ideals $\{F_{i}\}_{i \in \mathbb{N}}$, where $F_{i} = \bigoplus_{n \ge i} Y_{\hbar}(\mathfrak{g})_{n}$. Note that $Y_{\hbar}(\mathfrak{g})/ F_{i}$ is a superalgebra for all $i \in \mathbb{N}$. We define the completion
\[ \widehat{Y_{\hbar}(\mathfrak{g})} := \widehat{Y_{\hbar}(\mathfrak{g})}_{\{F_{i}\}_{i \in \mathbb{N}}} = \mathop{\lim_{\longleftarrow}}_{ i \in \mathbb{N} } Y_{\hbar}(\mathfrak{g})/ F_{i}. \]
Let $ \iota : Y_{\hbar}(\mathfrak{g}) \to \widehat{Y_{\hbar}(\mathfrak{g})} $ be the homomorphism of superalgebras given by $X \to ( q_{n}(X) )_{n \in \mathbb{N}_{0} }$ for all $X \in Y_{\hbar}(\mathfrak{g})$. Note that $\iota$ is injective: if $X \in Ker( \iota )$ then $X \in \cap_{n \in \mathbb{N}} F_{n} = \{ 0 \} $, i. e. $\widehat{Y_{\hbar}(\mathfrak{g})}$ is a separable $\Bbbk[[ \hbar ]]$-module. It is easy to see from the definition of a completion that the following result holds. For each $n \in \mathbb{N}_{0}$, denote by $q_n$ the natural quotient map which is given by $q_n : Y_{\hbar}(\mathfrak{g}) \to Y_{\hbar}(\mathfrak{g}) / F_{n+1}$.

\begin{lemma}
	\label{lm:isomcompl}
	The embedding $ \iota $ extends to an isomorphism of superalgebras
	\[ \Phi : \prod_{k=0}^{ \infty }  Y_{\hbar}(\mathfrak{g})_{k} \to \widehat{Y_{\hbar}(\mathfrak{g})}, \; \sum_{k=0}^{ \infty } X_{k} \to ( \sum_{k=0}^{ n } q_{n} (X_{k}) )_{n \in \mathbb{N}_{0}}, \]
	where $X_k \in Y_{\hbar}(\mathfrak{g})_{k}$ for all $k \in \mathbb{N}_{0}$.
\end{lemma}
Thus we shall view the elements of  $\widehat{Y_{\hbar}(\mathfrak{g})}$ as infinite series $\sum_{k=0}^{ \infty } X_{k}$ with $X_k \in Y_{\hbar}(\mathfrak{g})_{k}$ for all $k \in \mathbb{N}_{0}$. Note that $\Phi$ preserves the $\mathbb{Z}_{2}$ grading (it is sufficient to check that generators preserve their grading under this homomorphism).

Note that $Y_{\hbar}(\mathfrak{g}) \otimes Y_{\hbar}(\mathfrak{g})$ is a $\mathbb{N}_{0}$-graded ring as $Y_{\hbar}(\mathfrak{g}) \otimes Y_{\hbar}(\mathfrak{g}) = \bigoplus_{n \in \mathbb{N}_{0}} (Y_{\hbar}(\mathfrak{g}) \otimes Y_{\hbar}(\mathfrak{g}) )_{n}$, where $(Y_{\hbar}(\mathfrak{g}) \otimes Y_{\hbar}(\mathfrak{g}) )_{n} = \bigoplus_{ \substack{ r, s \in \mathbb{N}_{0}, \\ r+s=n } } Y_{\hbar}(\mathfrak{g})_{r} \otimes Y_{\hbar}(\mathfrak{g})_{s}$. Then we have a descending filtration of two-sided $\mathbb{Z}_2$-graded ideals $\{F_{i}\}_{i \in \mathbb{N}}$, where $F_{i} = \bigoplus_{n \ge i} (Y_{\hbar}(\mathfrak{g}) \otimes Y_{\hbar}(\mathfrak{g}) )_{n}$. We define the completion
\[ Y_{\hbar}(\mathfrak{g}) \widehat{\otimes} Y_{\hbar}(\mathfrak{g}) := \mathop{\lim_{\longleftarrow}}_{ i \in \mathbb{N} } Y_{\hbar}(\mathfrak{g}) \otimes Y_{\hbar}(\mathfrak{g}) / F_{i}. \]
As in Lemma \ref{lm:isomcompl}, we can identify $Y_{\hbar}(\mathfrak{g}) \widehat{\otimes} Y_{\hbar}(\mathfrak{g})$ with the direct product
\[ \prod_{k=0}^{ \infty } (Y_{\hbar}(\mathfrak{g}) \otimes Y_{\hbar}(\mathfrak{g}) )_{k} = \prod_{k=0}^{ \infty } \left ( \bigoplus_{ \substack{ r, s \in \mathbb{N}_{0}, \\ r+s=k } } Y_{\hbar}(\mathfrak{g})_{r} \otimes Y_{\hbar}(\mathfrak{g})_{s} \right ) . \]
We proceed in the same manner to define  the completed $n$-th tensor power $Y_{\hbar}(\mathfrak{g})^{ \widehat{\otimes} n} = \widehat{\bigotimes}_{k=1}^{n} Y_{\hbar}(\mathfrak{g})$ for $ 3 \le n \in \mathbb{N}$.

We show that $\widehat{Y_{\hbar}(\mathfrak{g})}$ is the Hopf superalgebra using the structure defined on $Y_{\hbar}(\mathfrak{g})$ in Subsection \ref{sub:hssy}. Define a counit $\epsilon: \widehat{Y_{\hbar}(\mathfrak{g})} \to \Bbbk $ on generators as in the equation \eqref{eq:epsdef}. Note that the element $ \Omega^{+} $ can be viewed as an element of $Y_{\hbar}(\mathfrak{g}) \widehat{\otimes} Y_{\hbar}(\mathfrak{g})$, and therefore we may define the function
\[ \Delta : \{ x_{i0}^{ \pm } , h_{i0}, h_{i1} \;\ | \; i \in I \} \to Y_{\hbar}(\mathfrak{g}) \widehat{\otimes} Y_{\hbar}(\mathfrak{g}) \]
exactly as $\Delta$ has been defined in \eqref{eq:comg} and \eqref{eq:comh1}.

\begin{proposition}
	\label{pr:homcomcY}
	Assume that $\mathfrak{g}$ satisfies constraints listed in Theorem \ref{th:mpy}. The function $\Delta$ extends to an superalgebra homomorphism $ \Delta : \widehat{Y_{\hbar}(\mathfrak{g})} \to Y_{\hbar}(\mathfrak{g}) \widehat{\otimes} Y_{\hbar}(\mathfrak{g}) $.
\end{proposition}
\begin{proof}
	It follows from Lemma \ref{lm:isomcompl} that relations \eqref{eq:mpy1} - \eqref{eq:mpy6} can be proved without modifications as in Theorem \ref{th:comhomdef} when $Y_{\hbar}(\mathfrak{g}) \otimes Y_{\hbar}(\mathfrak{g})$ is replaced with $Y_{\hbar}(\mathfrak{g}) \widehat{\otimes} Y_{\hbar}(\mathfrak{g})$. To show that the formula \eqref{eq:mpy1} for $(r,s)=(1,1)$ holds we need some extra considerations that are the same as in the \cite[Proposition 5.18]{GNW18}.
\end{proof}

Using Propostion \ref{pr:homcomcY}, formulas \eqref{eq:comrec1}, \eqref{eq:comrec2} and the same considerations as in Lemma \ref{lm:comcoass} it is easy to show that $\Delta$ is the comultiplication on $\widehat{Y_{\hbar}(\mathfrak{g})}$. We define the antipode $S : \widehat{Y_{\hbar}(\mathfrak{g})} \to \widehat{Y_{\hbar}(\mathfrak{g})}$ by the same arguments as in Subsection \ref{sub:hssy}. Thus the antipode acts on generators in the same way as in equations \eqref{eq:antipodeY1} - \eqref{eq:antipodeY2}.

\vspace{1cm}

\section{Isomorphism}

In this section we construct an explicit homomorphism of superalgebras
\[ \Phi : U_{\hbar}(L\mathfrak{g}_d) \to \widehat{Y_{\hbar}(\mathfrak{g}_{d})}, \]
where $d \in D$, $\widehat{Y_{\hbar}(\mathfrak{g})}$ is the completion of $Y_{\hbar}(\mathfrak{g})$ introduced in Subsection \ref{lb:HSSCY}, and show that it induces an isomorphism of completed superalgebras. We use  approach described in \cite{GT13} where non super case is proved. Also the result is proved in \cite{S18} for $\mathfrak{sl}(1|1)$ and stated without complete proof in general case for Lie superalgebras of type A. We give the complete proof for this case.

The proof is organized as follows: in order to not repeat technical details given in \ref{lb:HSSCY} when super and non super cases are actually the same we skip some results or just state them without proof. But in doing so we give explicit proofs when super case differs significantly from the non super one.

In order to prove the result we need to assume that the PBW theorem holds for $Y_{\hbar}(\mathfrak{g}_{d})$.
The assumption is required in the following results: Corollary \ref{cl:Bseq} and Lemmas \ref{lm:ql5comp} - \ref{lm:ql8comp}.

\vspace{1cm}

\subsection{PBW theorem for super Yangians}

The PBW theorem was proved in \cite{T19b} for $Y_{\hbar}(\mathfrak{g}_{d})$ for all $d \in D$. For the convenience of the reader, we remind it here. We fix $d \in D$ and ommit it. Pick any total ordering $\preceq$ on $ \Delta^{+} \times \mathbb{N}_{0} $. For every $( \beta, r) \in \Delta^{+} \times \mathbb{N}_{0} $, we choose:
\begin{enumerate}
	\item a decomposition $\beta = \alpha_{i_1} + ... + \alpha_{i_p} $ such that $[ ... [ e_{\alpha_{i_1}} , e_{\alpha_{i_2}} ] , ... , e_{\alpha_{i_p}} ]$ is a nonzero root vector $e_{\beta}$ of $\mathfrak{g}$ (here, $e_{\alpha_{i}}$ denotes the standart Chevalley generator of $\mathfrak{g}$);
	\item a decomposition $r = r_1 + ... + r_p$ with $r_i \in \mathbb{N}_{0}$.
\end{enumerate}
Define the PBW basis elements $x_{\beta , r}^{\pm}$ of $Y_{\hbar}(\mathfrak{g})$ via
\[ x_{\beta, r }^{ \pm} := [ ... [ [ x_{i_1, r_1}^{\pm} , x_{i_2, r_2}^{\pm} ] , x_{i_3, r_3}^{\pm} ] , ... , x_{i_p, r_p}^{\pm} ]. \]
Let $H$ denote the set of all functions $h : \Delta^{+} \times \mathbb{N}_{0} \to \mathbb{N}_{0}$ with finite support and such that $h( \beta, r) \le 1$ if $|\beta|= \bar{1}$ (we set here $|\pm ( \alpha_j + ... + \alpha_{i} )| := |\alpha_j| + ... + |\alpha_{i}| \in \mathbb{Z}_2$). The monomials
\[ x_{h}^{\pm} := \prod_{ ( \beta, r) \in \Delta^{+} \times \mathbb{N} }^{ \to } (x_{ \beta, r }^{\pm})^{h( \beta, r)} \text{  with  } h \in H \]
will be called the ordered PBW monomials of $Y_{\hbar}^{\pm}(\mathfrak{g})$. Recall notations given in Subsection \ref{sect:DSY}. Here, the arrow $\to$ over the product sign refers to a total order $\preceq$ on $ \Delta^{+} \times \mathbb{N}_{0} $.

\begin{theorem}
	\label{th:pbwyang}
	\begin{enumerate}
		\item The ordered PBW monomials $ \{ x_{h}^{\pm} \}_{h \in H}$ form a $ \Bbbk $-basis of $Y_{\hbar}^{\pm}(\mathfrak{g})$. In other words $Y_{\hbar}^{\pm}(\mathfrak{g})$ is the superalgebra generated by elements $\{ x_{ir}^{\pm} \}_{i \in I, r \in \mathbb{N}}$ subject to the relations \eqref{eq:SYrc4}, \eqref{eq:bssr} - \eqref{eq:qssr};
		\item $Y_{\hbar}^{0}(\mathfrak{g})$ is a polynomial superalgebra in the generators $\{ h_{ir} \}_{i \in I, r \in \mathbb{N}_0}$;
		\item the products of ordered PBW monomials $ \{ x_{h}^{-} \}_{h \in H} $, $ \{ x_{h^{'}}^{+} \}_{h^{'} \in H} $, and the ordered (in any way) monomials in $ \{ h_{ir} \}_{i \in I, r \in \mathbb{N}_0} $ form a $\Bbbk$-basis of $Y_{\hbar}(\mathfrak{g})$.
	\end{enumerate}
\end{theorem}
As an important corollary, we obtain
\begin{corollary}
	\begin{enumerate}
		\item $Y_{\hbar}^{\ge 0}(\mathfrak{g})$ (resp. $Y_{\hbar}^{\le 0}(\mathfrak{g})$) is the superalgebra generated by elements $ \{ h_{ir}, x^{+}_{ir} \}_{i \in I, r \in \mathbb{N}_{0}} $ (resp. $\{ h_{ir}, x^{-}_{ir} \}_{i \in I, r \in \mathbb{N}_{0}}$) subject to the relations \eqref{eq:Cer} - \eqref{eq:SYrc1}, \eqref{eq:SYrc2} - \eqref{eq:qssr};
		\item the multiplication map
		\[ m : Y_{\hbar}^{-}(\mathfrak{g}) \otimes Y_{\hbar}^{0}(\mathfrak{g}) \otimes Y_{\hbar}^{+}(\mathfrak{g}) \to Y_{\hbar}(\mathfrak{g}) \]
		is an isomorphism of $\Bbbk$-vector superspaces.
	\end{enumerate}
    \label{cl:pbwcl}
\end{corollary}

\vspace{1cm}

\subsection{Auxiliary operators and algebraic structures}

We introduce functions and structures that are needed in the following. We use notations introduced in Subsections \ref{sub:compl}, \ref{sect:DSY} and \ref{lb:HSSCY}.

Fix $i \in I$. Denote by $\sigma_{i}^{+}: Y^{ \ge 0 }_{\hbar}(\mathfrak{g}) \to Y^{ \ge 0 }_{\hbar}(\mathfrak{g})$ (resp. $\sigma_{i}^{-}: Y^{ \le 0 }_{\hbar}(\mathfrak{g}) \to Y^{ \le 0 }_{\hbar}(\mathfrak{g})$) superalgebra homomorphisms
\begin{equation}
	\label{eq:sigmadef}
	\sigma^{ \pm }_{i}( h_{jr} ) = h_{jr}, \; \sigma^{ \pm }_{i}(x_{jr}^{\pm}) = x_{j,r+\delta_{i j}}^{\pm},
\end{equation}
where $j \in I$ and $r \in \mathbb{N}_{0}$. 

Fix $i \in I$. Denote by
\[ \xi_{i}(u) = 1 + \hbar \sum_{r \ge 0} h_{ir} u^{-r-1} \in Y_{\hbar}(\mathfrak{g})[[u^{-1}]]. \]

For any $i \in I$, define the formal power series
\[ t_i(u) = \hbar \sum_{r \ge 0} t_{ir} u^{-r-1} \in Y^{0}_{\hbar}(\mathfrak{g})[[u^{-1}]] \]
by
\begin{equation}
	\label{eq:tlogdef}
	t_{i}(u) = \log( \xi_i (u)) = \log( 1 + \hbar \sum_{r \ge 0} h_{ir} u^{-r-1} ).
\end{equation}
Since $t_{i}(u)$ is invertible, $\{ t_{ir} \}_{i \in I, r \in \mathbb{N}_{0}}$ is another system of generators of $Y^{0}_{\hbar}(\mathfrak{g})$. Moreover, notice that $t_{i0} = h_{i0}$ and $t_{ir} = h_{ir} \pmod{\hbar} $ for any $r \ge 1$. Also $\deg(t_{ir}) = r$.

Fix $i \in I$. We introduce the following formal power series
\begin{equation}
	\label{eq:borinv}
	B_{i}(v) = B(t_i(u)) = \hbar \sum_{r \ge 0} t_{ir} \frac{v^r}{r!} \in Y^{0}_{\hbar}(\mathfrak{g})[[v]].
\end{equation}

\begin{proposition}
	\emph{\cite{GT13}}
	\label{pr:lamdef}
	There are operators $ \{ \lambda_{i; s}^{ \pm} \}_{i \in I, s \in \mathbb{N}_{0}} $ on $ Y^{0}_{\hbar}(\mathfrak{g}) $ such that the following holds
	\begin{enumerate}
		\item for any $ h \in Y^{0}_{\hbar}(\mathfrak{g}) $, the elements $ \lambda_{i; s}^{ \pm}(h) \in Y^{0}_{\hbar}(\mathfrak{g}) $ are uniquely determined by the requirement that, for any $m \in \mathbb{N}_{0}$,
		\[ x_{im}^{ \pm } h = \sum_{s \ge 0} \lambda_{i; s}^{ \pm}(h) x_{i, m+s}^{ \pm}; \]
		\item for any $h, \eta \in Y^{0}_{\hbar}(\mathfrak{g})$,
		\[ \lambda_{i; s}^{ \pm } ( h \eta ) = \sum_{k+l=s} \lambda_{i; k}^{ \pm}(h) \lambda_{i; l}^{ \pm}( \eta); \]
		\item the operator $ \lambda_{i;s} : Y^{0}_{\hbar}(\mathfrak{g}) \to Y^{0}_{\hbar}(\mathfrak{g}) $ is homogeneous of degree $-s$;
		\item let $ \lambda_{i}^{ \pm } : Y^{0}_{\hbar}(\mathfrak{g}) \to Y^{0}_{\hbar}(\mathfrak{g})[v] $ be given by
		\begin{equation}
			\label{eq:lambdadefhom}
			\lambda_{i}^{ \pm }(v)(h) = \sum_{s \ge 0} \lambda_{i; s}^{ \pm}(h) v^s.
		\end{equation}
		Then $\lambda_{i}^{ \pm }(v)$ is a superalgebra homomorphism of degree $0$;
		\item $\lambda_{i_1}^{ \epsilon_1 }(v_1)$ and $\lambda_{i_2}^{ \epsilon_2 }(v_2)$ commute for any $i_1, i_2 \in I$ and $\epsilon_1, \epsilon_2 \in \{ \pm \}$;
		\item for any $i \in I$,
		\[ \lambda_{i}^{+} (v) \lambda_{i}^{-}(v) = id_{Y^{0}_{\hbar}(\mathfrak{g})} = \lambda_{i}^{-}(v) \lambda_{i}^{+} (v). \]
	\end{enumerate}
\end{proposition}

Consider the formal power series
\begin{equation}
	\label{eq:Gdef}
	G(v) = \log( \frac{v}{ e^{v/2} - e^{-v/2} } ) \in v \mathbb{Q}[[v]].
\end{equation}
Fix $i \in I$. Define $\gamma_i (v) \in \widehat{Y^{0}_{i}(\mathfrak{g})[v]}_{+} $ by
\[ \gamma_i (v) = \hbar \sum_{r \ge 0} \frac{t_{ir}}{r!} ( - \partial_v )^{r+1} G(v), \]
where $\partial_v \in End_{\mathbb{Q}}( \mathbb{Q}[[v]])$ is the derivative with respect to $v$.

Fix $i \in I$. Define $g_{i}^{\pm} (v) \in \widehat{Y^{0}_{\hbar}(\mathfrak{g})[v]} $ by
\begin{equation}
	\label{eq:gipm}
	g_i^{\pm}(v) = \sum_{m \ge 0} g_{im}^{\pm} v^m.
\end{equation}

We need the following
\begin{lemma}
	\emph{\cite{GT13}}
	\label{lm:invgpm}
	Let $\{ g_{i}^{\pm}(u) \}_{i \in I} \subset \widehat{Y^{0}_{\hbar}(\mathfrak{g})[u]}$ be elements satisfying condition \eqref{eq:bcomp} of Theorem \ref{th:thcondhom}. Then,
	\[ g_{i}^{\pm}(u) = \frac{1}{d_{i}^{ \pm}} \mod \widehat{Y^{0}_{\hbar}(\mathfrak{g})[u]}_{+}, \]
	where $\{ d_{i}^{ \pm} \}_{i \in I} \in \Bbbk^{*}$ satisfy $d_{i}^{+} d_{i}^{-} = d_i$ for each $i \in I$. In particular, each $g_{i}^{ \pm}(u)$ is invertible.
\end{lemma}
Note that for our choise of Cartan matrices we have $d_{i}^{+}=d_{i}^{-}=d_{i}=1$ for $g_{i}^{\pm}(u)$ for all $i \in I$.

Fix $i \in I$. Define $g_{i}(v) \in \widehat{Y^{0}_{\hbar}(\mathfrak{g})[v]} $ by
\begin{equation}
	\label{eq:giunvsol}
	g_{i}(v) = ( \frac{\hbar}{q - q^{-1}} )^{\frac{1}{2}} \exp(\frac{\gamma_i(v)}{2}).
\end{equation}

We shall need the following
\begin{lemma}
	\emph{\cite{GT13}}
	\label{lm:lgG}
	Let $i,j \in I$, and set $a = c_{ij} / 2$. Then,
	\[ \lambda_{i}^{ \pm}(u)( g_{j}(v) ) = g_{j}(v) \exp( \pm \frac{G(v - u + a \hbar) - G(v - u - a \hbar)}{2} ). \]
\end{lemma}

Fix $i \in I$ and $m \ge 2$. We shall use the following notation
\begin{enumerate}
	\item for a vector superspace $V$ and an operator $T \in End(V)$, $T_{(i)} \in End(V^{\otimes m})$ is defined as
	\[ T_{(i)} = id^{\otimes i-1} \otimes T \otimes id^{m-i}; \]
	\item for an superalgebra $A$, $ad^{(m)} : A^{\otimes m} \to End(A)$ is defined as
	\[ ad^{(m)} ( a_1 \otimes ... \otimes a_m ) = ad(a_1) \circ ... \circ ad(a_m), \]
\end{enumerate}
where $ad: A \to End(A)$ is the adjoint action defined by $ad(a)(b) = [a,b]$ for all $a, b \in A$.

\begin{proposition}
	\label{pr:y46eqc}
	\begin{enumerate}
		\item \label{eq:rel1} The relation \eqref{eq:SYrc4} for $i \ne j$ is equivalent to the requirement that the following holds for any $A(v_1, v_2) \in \Bbbk[[v_1, v_2]]$
		\[ A( \sigma_{i}^{ \pm} , \sigma_{j}^{ \pm} )( \sigma_{i}^{ \pm} - \sigma_{j}^{ \pm} \mp a \hbar ) x_{i0}^{ \pm} x_{j0}^{ \pm} = (-1)^{|\alpha_{i}||\alpha_{j}|} A( \sigma_{i}^{ \pm}, \sigma_{j}^{ \pm})( \sigma_{i}^{ \pm} - \sigma_{j}^{ \pm} \pm a \hbar ) x_{j0}^{ \pm} x_{i0}^{ \pm}, \]
		where $a = c_{ij} / 2$;
		\item \label{eq:rel2} the relation \eqref{eq:SYrc4} for $i = j$ is equivalent to the requirement that the following holds for any $B(v_1, v_2) \in \Bbbk[[v_1, v_2]]$ such that $B(v_1, v_2) = B(v_2, v_1)$
		\begin{equation}
			\label{y4ieqjreleq}
			\mu( B( \sigma_{i, (1)}^{ \pm}  , \sigma_{i, (2)}^{ \pm} ) ( \sigma_{i, (1)}^{ \pm} - \sigma_{i, (2)}^{ \pm} \mp a \hbar )  x_{i0}^{ \pm} \otimes x_{i0}^{ \pm}  ) = 0,
		\end{equation}
		where $\mu: Y_{\hbar}(\mathfrak{g})^{ \otimes 2} \to Y_{\hbar}(\mathfrak{g})$ is the multiplication, $a = c_{ii} / 2$ and $|\alpha_{i}| = \bar{0}$;
		\item \label{eq:rel3} the relations \eqref{eq:bssr} and \eqref{eq:cssr} are equivalent to the requirement that the following holds for any $i \ne j$ and $A \in \Bbbk[v_1, ..., v_m ]^{\mathfrak{S}_{m}}$ with $m = 1 + |c_{ij}| $
		\[ ad^{(m)} ( A( \sigma_{i, (1)}^{ \pm} , ... , \sigma_{i, (m)}^{ \pm} ) ( x_{i0}^{ \pm} )^{ \otimes m} ) x_{jl}^{ \pm} = 0, \]
		where $\Bbbk[v_1, ..., v_m ]^{\mathfrak{S}_{m}}$ is a subalgebra of elements invariant under action of symmetric group $\mathfrak{S}_{m}$ on polynomial superalgebra $\Bbbk[v_1, ..., v_m ]$, where $|v_{i}| = \bar{0}$ for all $i \in \{1,2,...,m\}$; $|\alpha_{i}| = \bar{0}$, if $|c_{ij}|=1$;
		\item \label{eq:rel4} the relation \eqref{eq:bssr} for $i \ne j$ is equivalent to the requirement that the following holds for any $A(v_1, v_2) \in \Bbbk[[v_1, v_2]]$
		\[ A(\sigma_{i}^{ \pm} , \sigma_{j}^{ \pm}) x_{i0}^{\pm} x_{j0}^{\pm} = (-1)^{|\alpha_{i}||\alpha_{j}|} A(\sigma_{i}^{ \pm} , \sigma_{j}^{ \pm}) x_{j0}^{\pm} x_{i0}^{\pm}, \]
		where $c_{ij} = 0$;
		\item \label{eq:rel40} the relation \eqref{eq:bssr} for $i = j$ is equivalent to the requirement that the following holds for any $B(v_1, v_2) \in \Bbbk[[v_1, v_2]]$ such that $B(v_1, v_2) = B(v_2, v_1)$
		\[ \mu ( B( \sigma_{i,(1)}^{\pm } , \sigma_{i,(2)}^{\pm } ) x_{i0}^{\pm} \otimes x_{i0}^{\pm} ) = 0, \]
		where $c_{ii} = 0$ and $\mu: Y_{\hbar}(\mathfrak{g})^{ \otimes 2} \to Y_{\hbar}(\mathfrak{g})$ is the multiplication;
		\item \label{eq:rel5} the relation \eqref{eq:cssr} for $r=s=t=0$ is equivalent to the requirement that 
		\[ \mu (\mu \otimes \id) (\sigma_{i,(1)}^{ \pm} - \sigma_{j,(2)}^{ \pm } \pm a \hbar) x_{i0}^{\pm} \otimes x_{j0}^{\pm} \otimes x_{i0}^{\pm} = \mu (\mu \otimes \id) (\sigma_{i,(3)}^{ \pm} - \sigma_{j,(2)}^{ \pm } \mp a \hbar) x_{i0}^{\pm} \otimes x_{j0}^{\pm} \otimes x_{i0}^{\pm};   \]
	    \item \label{eq:rel6} the relation \eqref{eq:qssr} for $r=s=0$ is equivalent to the requirement that 
		\[ \mu (\mu \otimes \id) (\mu \otimes \id \otimes \id) \circ \]
		\[ ( \sigma_{j-1,(2)}^{\pm}- \sigma_{j,(1)}^{\pm} \mp a \hbar) ( \sigma_{j,(1)}^{\pm}- \sigma_{j+1,(3)}^{\pm} \mp a \hbar) ( \sigma_{j-1,(2)}^{\pm}- \sigma_{j,(4)}^{\pm} \pm a \hbar  ) (\sigma_{j,(4)}^{\pm}- \sigma_{j+1,(3)}^{\pm} \mp a \hbar  )  \]
		\[ x_{j0}^{\pm} \otimes x_{j-1,0}^{\pm} \otimes x_{j+1,0}^{\pm} \otimes x_{j0}^{\pm} = 0. \]
	\end{enumerate}
\end{proposition}
\begin{proof}
	We prove \eqref{eq:rel1}. The relation \eqref{eq:SYrc4} can be rewritten as
	\[ \sigma_{i}^{\pm r} \sigma_{j}^{\pm s} ( \sigma_{i}^{\pm} - \sigma_{j}^{\pm} \mp a 
	\hbar ) ( x_{i0}^{\pm} x_{j0}^{\pm} ) = (-1)^{|\alpha_{i}||\alpha_{j}|} \sigma_{i}^{\pm r} \sigma_{j}^{\pm s} ( \sigma_{i}^{\pm} - \sigma_{j}^{\pm} \pm a 
	\hbar ) ( x_{j0}^{\pm} x_{i0}^{\pm} ). \]
	
	Relations \eqref{eq:rel2} and \eqref{eq:rel3} are proved as in \emph{\cite{GT13}}.
	
	We prove \eqref{eq:rel4}. The relation \eqref{eq:bssr} for $i \ne j$ can be rewritten as
	\[ \sigma_{i}^{\pm r} \sigma_{j}^{\pm s} ( x_{i0}^{\pm} x_{j0}^{\pm} ) = (-1)^{|\alpha_{i}||\alpha_{j}|} \sigma_{i}^{\pm r} \sigma_{j}^{\pm s} ( x_{j0}^{\pm} x_{i0}^{\pm} ).  \]
	
	We prove \eqref{eq:rel40}. The relation \eqref{eq:bssr} for $i = j$ can be rewritten as
	\[ \mu( \sigma_{i,(1)}^{\pm r} \sigma_{i,(2)}^{\pm s} ( x_{i0}^{\pm} \otimes x_{i0}^{\pm} ) ) = - \mu( \sigma_{i,(1)}^{\pm s} \sigma_{i,(2)}^{\pm r} ( x_{i0}^{\pm} \otimes x_{i0}^{\pm} ) ) \iff  \]
	\[ \mu( ( \sigma_{i,(1)}^{\pm r} \sigma_{i,(2)}^{\pm s} + \sigma_{i,(1)}^{\pm s} \sigma_{i,(2)}^{\pm r} ) ( x_{i0}^{\pm} \otimes x_{i0}^{\pm} ) ) = 0. \]
	
	We prove \eqref{eq:rel5}. The relation \eqref{eq:cssr} for $r=s=t=0$ can be rewritten as
	\[ x_{i0}^{\pm 2} x_{j0}^{\pm} - 2 x_{i0}^{\pm} x_{j0}^{\pm} x_{i0}^{\pm} + x_{j0}^{\pm} x_{i0}^{\pm 2} = 0. \]
	Using \eqref{eq:rel1} we get
	\[ x_{i0}^{\pm} \frac{\sigma_{i}^{ \pm} - \sigma_{j}^{ \pm } \pm a \hbar}{\sigma_{i}^{ \pm} - \sigma_{j}^{ \pm } \mp a \hbar} ( x_{j0}^{\pm}  x_{i0}^{\pm} ) - 2 x_{i0}^{\pm} x_{j0}^{\pm} x_{i0}^{\pm} + \frac{\sigma_{i}^{ \pm} - \sigma_{j}^{ \pm } \mp a \hbar}{\sigma_{i}^{ \pm} - \sigma_{j}^{ \pm } \pm a \hbar} ( x_{i0}^{\pm} x_{j0}^{\pm} ) x_{i0}^{\pm } = 0 \iff  \]
	\[ x_{i0}^{\pm} ( \frac{1}{\sigma_{i}^{ \pm} - \sigma_{j}^{ \pm } \mp a \hbar} ) (x_{j0}^{\pm} x_{i0}^{\pm}) = ( \frac{1}{\sigma_{i}^{ \pm} - \sigma_{j}^{ \pm } \pm a \hbar} ) ( x_{i0}^{\pm} x_{j0}^{\pm} ) x_{i0}^{\pm} \iff \]
	\[ \mu (\mu \otimes \id) ( \frac{1}{\sigma_{i,(3)}^{ \pm} - \sigma_{j,(2)}^{ \pm } \mp a \hbar} - \frac{1}{\sigma_{i,(1)}^{ \pm} - \sigma_{j,(2)}^{ \pm } \pm a \hbar} ) ( x_{i0}^{\pm} \otimes x_{j0}^{\pm} \otimes x_{i0}^{\pm} ) = 0 \iff \]
	\[ \mu (\mu \otimes \id) (\sigma_{i,(1)}^{ \pm} - \sigma_{j,(2)}^{ \pm } \pm a \hbar) x_{i0}^{\pm} \otimes x_{j0}^{\pm} \otimes x_{i0}^{\pm} = \]
	\[ \mu (\mu \otimes \id) (\sigma_{i,(3)}^{ \pm} - \sigma_{j,(2)}^{ \pm } \mp a \hbar) x_{i0}^{\pm} \otimes x_{j0}^{\pm} \otimes x_{i0}^{\pm}. \]
	
	We prove \eqref{eq:rel6}. Using \eqref{eq:rel1} by direct computations the relation \eqref{eq:qssr} for $r=s=0$ can be rewritten as	
	\[ ( \frac{ \pm 2 a \hbar}{ \sigma_{j-1}^{\pm}- \sigma_{j}^{\pm} \mp a \hbar} + \frac{\sigma_{j}^{\pm}- \sigma_{j+1}^{\pm} \mp a \hbar}{ \sigma_{j}^{\pm}- \sigma_{j+1}^{\pm} \pm a \hbar} ) ( x_{j,0}^{\pm} x_{j-1,0}^{\pm} ) x_{j+1,0}^{\pm} x_{j,0}^{\pm} + \]
	\[ x_{j,0}^{\pm} ( \frac{\mp 2 a \hbar}{ \sigma_{j-1}^{\pm}- \sigma_{j}^{\pm} \pm a \hbar} + \frac{\sigma_{j}^{\pm}- \sigma_{j+1}^{\pm} \pm a \hbar}{ \sigma_{j}^{\pm}- \sigma_{j+1}^{\pm} \mp a \hbar} ) ( x_{j-1,0}^{\pm} x_{j+1,0}^{\pm} x_{j,0}^{\pm} ) = 0 \iff \]
	
	\[ \mu (\mu \otimes \id) (\mu \otimes \id \otimes \id) \circ \]
	\[ ( ( \pm 2 a \hbar ( \sigma_{j,(1)}^{\pm}- \sigma_{j+1,(3)}^{\pm} \pm a \hbar ) + ( \sigma_{j-1,(2)}^{\pm}- \sigma_{j,(1)}^{\pm} \mp a \hbar) ( \sigma_{j,(1)}^{\pm}- \sigma_{j+1,(3)}^{\pm} \mp a \hbar) ) \cdot \]
	\[ ( \sigma_{j-1,(2)}^{\pm}- \sigma_{j,(4)}^{\pm} \pm a \hbar  ) (\sigma_{j,(4)}^{\pm}- \sigma_{j+1,(3)}^{\pm} \mp a \hbar  ) + \]
	\[ ( \mp 2 a \hbar ( \sigma_{j,(4)}^{\pm}- \sigma_{j+1,(3)}^{\pm} \mp a \hbar ) + ( \sigma_{j-1,(2)}^{\pm}- \sigma_{j,(4)}^{\pm} \pm a \hbar) ( \sigma_{j,(4)}^{\pm}- \sigma_{j+1,(3)}^{\pm} \pm a \hbar) ) \cdot \]
	\[ ( \sigma_{j-1,(2)}^{\pm}- \sigma_{j,(1)}^{\pm} \mp a \hbar  ) (\sigma_{j,(1)}^{\pm}- \sigma_{j+1,(3)}^{\pm} \pm a \hbar  ) ) \]
	\[ x_{j0}^{\pm} \otimes x_{j-1,0}^{\pm} \otimes x_{j+1,0}^{\pm} \otimes x_{j0}^{\pm} = 0. \]
	
	Using \eqref{eq:rel5} we have
	\[ \mu (\mu \otimes \id) (\mu \otimes \id \otimes \id) \circ \]
	\[ ( (\sigma_{j,(1)}^{\pm}- \sigma_{j+1,(3)}^{\pm} \pm a \hbar) - (\sigma_{j,(4)}^{\pm}- \sigma_{j+1,(3)}^{\pm} \mp a \hbar) ) x_{j0}^{\pm} \otimes x_{j-1,0}^{\pm} \otimes x_{j+1,0}^{\pm} \otimes x_{j0}^{\pm} = 0,  \]
	\[ \mu (\mu \otimes \id) (\mu \otimes \id \otimes \id) \circ \]
	\[ ( (\sigma_{j-1,(2)}^{\pm}- \sigma_{j,(1)}^{\pm} \mp a \hbar) - (\sigma_{j-1,(2)}^{\pm}- \sigma_{j,(4)}^{\pm} \pm a \hbar) ) x_{j0}^{\pm} \otimes x_{j-1,0}^{\pm} \otimes x_{j+1,0}^{\pm} \otimes x_{j0}^{\pm} = 0, \]
	\[ \mu (\mu \otimes \id) (\mu \otimes \id \otimes \id) \circ \]
	\[ ( (\sigma_{j,(1)}^{\pm}- \sigma_{j+1,(3)}^{\pm} \mp a \hbar) - (\sigma_{j,(4)}^{\pm}- \sigma_{j+1,(3)}^{\pm} \pm a \hbar) ) x_{j0}^{\pm} \otimes x_{j-1,0}^{\pm} \otimes x_{j+1,0}^{\pm} \otimes x_{j0}^{\pm} = 0.  \]
	Thus
	\[ \mu (\mu \otimes \id) (\mu \otimes \id \otimes \id) \circ \]
	\[ ( \sigma_{j-1,(2)}^{\pm}- \sigma_{j,(1)}^{\pm} \mp a \hbar) ( \sigma_{j,(1)}^{\pm}- \sigma_{j+1,(3)}^{\pm} \mp a \hbar) ( \sigma_{j-1,(2)}^{\pm}- \sigma_{j,(4)}^{\pm} \pm a \hbar  ) (\sigma_{j,(4)}^{\pm}- \sigma_{j+1,(3)}^{\pm} \mp a \hbar  )  \]
	\[ x_{j0}^{\pm} \otimes x_{j-1,0}^{\pm} \otimes x_{j+1,0}^{\pm} \otimes x_{j0}^{\pm} = 0. \]
\end{proof}

\begin{corollary}
	\emph{\cite{GT13}}
	\label{cl:Bseq}
	If \eqref{y4ieqjreleq} holds for some $B \in \Bbbk[[v_1, v_2]]$, then $B(v_1, v_2) = B(v_2, v_1)$.
\end{corollary}

\vspace{1cm}

\subsection{Homomorphism from quantum loop superalgebra to super Yangian}

In order to build a homomorphism we impose additional restrictions on Dynkin diagrams: we require that every odd vertex has only even verteces as neighbors. This condition is needed in order Theorem \ref{th:gsercindtrue} to be true. So from now on in this section we suppose that this agreement holds. 

\begin{definition}
	\label{phidef}
	Define a function
	\[ \Phi : \{ H_{ir}, E_{ir}, F_{ir} \}_{i \in I, r \in \mathbb{Z}} \to \widehat{Y_{\hbar}(\mathfrak{g})} \]
	by
	\begin{equation}
		\Phi( H_{i0} ) = t_{i0}
	\end{equation}
	and for $r \in \mathbb{Z}^{*}$
	\begin{equation}
		\Phi( H_{ir} ) = \frac{\hbar}{q - q^{-1}} \sum_{k \ge 0} t_{ik} \frac{r^k}{k!} = \frac{B_{i}(r)}{q - q^{-1}},
	\end{equation}
	where we've used \eqref{eq:tlogdef} and \eqref{eq:borinv}. Define further for $r \in \mathbb{Z}$
	\begin{equation}
		\label{eq:eir}
		\Phi(E_{ir}) = e^{r \sigma_{i}^{+}} g_{i}^{+}( \sigma_{i}^{+} ) x_{i0}^{+},
	\end{equation}
	\begin{equation}
		\label{eq:fir}
		\Phi(F_{ir}) = e^{r \sigma_{i}^{-}} g_{i}^{-}( \sigma_{i}^{-} ) x_{i0}^{-},
	\end{equation}
	where we've used \eqref{eq:sigmadef} and \eqref{eq:gipm}.
\end{definition}

The above fucntion extends to a homomorphism of superlagebras $ \Phi^{0} : \widetilde{U}^{0}_{\hbar}(L\mathfrak{g}) \to Y^{0}_{\hbar}(\mathfrak{g}) $.

\begin{proposition}
	\emph{\cite{GT13}}
	\label{pr:ql2ql3com}
	The function $\Phi$ is compatible with relations \eqref{QL2} and \eqref{QL3} if and only if equations \eqref{eq:eir} and \eqref{eq:fir} hold.
\end{proposition}

\begin{theorem}
	\label{th:thcondhom}
	The function $\Phi$ extends to a homomorphism of superalgebras ${U}_{\hbar}(L\mathfrak{g}) \to \widehat{Y_{\hbar}(\mathfrak{g})}$ if and only if the following conditions hold
	\begin{enumerate}
		\item for any $i, j \in I$
		\begin{equation}
			\label{eq:acomp}
			g_{i}^{+}(u) \lambda_{i}^{+}(u) ( g_{j}^{-} (v) ) = g_{j}^{-} (v) \lambda_{j}^{-}(v) ( g_{i}^{+}(u) );
		\end{equation}
		\item for any $i \in I$ and $k \in \mathbb{Z}$
		\begin{equation}
			\label{eq:bcomp}
			e^{ku} g_{i}^{+}(u) \lambda_{i}^{+}(u) ( g_{i}^{-}(u) ) \Big\rvert_{u^m = h_{im}} = \Phi^{0} ( \frac{\psi_{ik} - \phi_{ik}}{q-q^{-1}} );
		\end{equation}
		\item for any $i, j \in I$ such that $c_{ij} \ne 0$ ($a = c_{ij} / 2$)
		\begin{equation}
			\label{eq:ccomp}
			g_{i}^{ \pm}(u) \lambda_{i}^{ \pm} (u) ( g_{j}^{ \pm} (v) ) ( \frac{ (-1)^{|\alpha_i||\alpha_j|} e^{u} - e^{v \pm a \hbar}}{u-v \mp a \hbar} ) = g_{j}^{ \pm}(v) \lambda_{j}^{ \pm} (v) ( g_{i}^{ \pm} (u) ) ( \frac{ e^{v} - (-1)^{|\alpha_i||\alpha_j|} e^{u \pm a \hbar} }{v-u \mp a \hbar} );
		\end{equation}
		\item for any $i, j \in I$ such that $c_{ij} = 0$
		\begin{equation}
			\label{eq:qcomp}
			g_{i}^{ \pm}(u) \lambda_{i}^{ \pm} (u) ( g_{j}^{ \pm} (v) ) = g_{j}^{ \pm}(v) \lambda_{j}^{ \pm} (v) ( g_{i}^{ \pm} (u) );
		\end{equation}
	\end{enumerate}
	We've used notations \eqref{eq:lambdadefhom} and \eqref{eq:gipm}.
\end{theorem}
\begin{proof}
	By construction and Proposition \ref{pr:ql2ql3com}, $\Phi$ is compatible with the relations \eqref{QL1} - \eqref{QL3}. The result then follows from Lemmas \ref{lm:ql5comp}, \ref{lm:ql4comp}, \ref{lm:ql6comp}, \ref{lm:ql7comp} and \ref{lm:ql8comp} below. 
\end{proof}

\begin{corollary}
	The homomorphism of superalgebras $\Phi$ has the following properties:
	\begin{enumerate}
		\item it restricts to a homomorphism of superalgebras $U_{\hbar}(L \mathfrak{sl}(2|1)^{i}) \to \widehat{Y_{\hbar}(\mathfrak{sl}(2|1)^{i})}$ for a fixed $i \in I$;
		\item it restricts to a homomorphism of superalgebras $U_{ \hbar}(L \mathfrak{b}_{ \pm}) \to \widehat{Y_{\hbar}( \mathfrak{b}_{ \pm})}$. 
	\end{enumerate}
\end{corollary}
\begin{proof}
	The result follows from Definition \ref{phidef}.
\end{proof}

\begin{lemma}
	\emph{\cite{GT13}}
	\label{lm:ql5comp}
	$\Phi$ is compatible with the relation \eqref{QL5} if and only if \eqref{eq:acomp} and \eqref{eq:bcomp} hold.
\end{lemma}

\begin{lemma}
	\label{lm:ql4comp}
	$\Phi$ is compatible with the relation \eqref{QL4} if and only if \eqref{eq:ccomp} holds.
\end{lemma}
\begin{proof}
	In order to give simultaneously one proof for both cases we set $X^{+}_{ir} = E_{ir}$ and $X^{-}_{ir} = F_{ir}$ for any $i \in I$ and $r \in \mathbb{Z}$. Compatibility with \eqref{QL4} reads
	\[ \Phi( X^{ \pm}_{i,k+1} ) \Phi( X^{ \pm}_{j,l} ) - (-1)^{|\alpha_{i}||\alpha_{j}|} q^{\pm c_{ij}} \Phi ( X^{ \pm}_{i,k} ) \Phi ( X^{ \pm}_{j,l+1} ) = \]
	\[ (-1)^{|\alpha_{i}||\alpha_{j}|} q^{\pm c_{ij}} \Phi ( X^{ \pm}_{j,l} ) \Phi( X^{ \pm}_{i,k+1} ) - \Phi( X^{ \pm}_{j,l+1} ) \Phi ( X^{ \pm}_{i,k} ). \]
	
	Set $a = c_{ij} / 2$, so that $q^{\pm c_{ij}} = e^{\pm a \hbar}$ for all $i,j \in I$.
	
	Assume first $i \ne j$. Since
	\[ \Phi( X^{\pm}_{ir} ) \Phi( X^{ \pm}_{js} ) = e^{r \sigma_{i}^{ \pm}} e^{s \sigma_{j}^{ \pm}} g_{i}^{ \pm}( \sigma_{i}^{ \pm}) \lambda_{i}^{ \pm}( \sigma_{i}^{ \pm}) ( g_{j}^{ \pm}( \sigma_{j}^{ \pm} ) ) x_{i0}^{ \pm} x_{j0}^{ \pm}  \]
	the above reduces to
	\[ e^{k \sigma_{i}^{ \pm} } e^{ l \sigma_{j}^{ \pm} } g_{i}^{ \pm} ( \sigma_{i}^{ \pm} ) \lambda_{i}^{ \pm} ( \sigma_{i}^{ \pm} ) ( g_{j}^{ \pm} ( \sigma_{j}^{ \pm} ) ) ( e^{ \sigma_{i}^{\pm}} - (-1)^{|\alpha_{i}||\alpha_{j}|} e^{\sigma_{j}^{\pm} \pm a \hbar} ) x_{i0}^{\pm} x_{j0}^{\pm} = \]
	\[ e^{k \sigma_{i}^{ \pm} } e^{ l \sigma_{j}^{ \pm} } g_{j}^{ \pm} ( \sigma_{j}^{ \pm} ) \lambda_{j}^{ \pm} ( \sigma_{j}^{ \pm} ) ( g_{i}^{ \pm} ( \sigma_{i}^{ \pm} ) ) (  (-1)^{|\alpha_{i}||\alpha_{j}|} e^{\sigma_{i}^{\pm} \pm a \hbar } - e^{\sigma_{j}^{+}} ) x_{j0}^{\pm} x_{i0}^{\pm}. \]
	Using \eqref{eq:rel1} of Proposition \ref{pr:y46eqc}, we get
	\[ ( e^{ \sigma_{i}^{\pm}} - (-1)^{|\alpha_{i}||\alpha_{j}|} e^{\sigma_{j}^{\pm} \pm a \hbar} ) x_{i0}^{\pm} x_{j0}^{\pm} = \frac{e^{ \sigma_{i}^{\pm}} - (-1)^{|\alpha_{i}||\alpha_{j}|} e^{\sigma_{j}^{\pm} \pm a \hbar}}{ \sigma_{i}^{\pm} - \sigma_{j}^{\pm} \mp a \hbar } ( \sigma_{i}^{\pm} - \sigma_{j}^{\pm} \mp a \hbar ) x_{i0}^{\pm} x_{j0}^{\pm} = \]
	\[ (-1)^{|\alpha_{i}||\alpha_{j}|}  \frac{e^{ \sigma_{i}^{\pm}} - (-1)^{|\alpha_{i}||\alpha_{j}|} e^{\sigma_{j}^{\pm} \pm a \hbar}}{ \sigma_{i}^{\pm} - \sigma_{j}^{\pm} \mp a \hbar } ( \sigma_{i}^{\pm} - \sigma_{j}^{\pm} \pm a \hbar ) x_{j0}^{\pm} x_{i0}^{\pm} =  \]
	\[ \frac{ (-1)^{|\alpha_{i}||\alpha_{j}|} e^{ \sigma_{i}^{\pm}} - e^{\sigma_{j}^{\pm} \pm a \hbar}}{ \sigma_{i}^{\pm} - \sigma_{j}^{\pm} \mp a \hbar } ( \sigma_{i}^{\pm} - \sigma_{j}^{\pm} \pm a \hbar ) x_{j0}^{\pm} x_{i0}^{\pm}. \]
	On the other hand,
	\[ ( (-1)^{|\alpha_{i}||\alpha_{j}|} e^{\sigma_{i}^{\pm} \pm a \hbar } - e^{\sigma_{j}^{\pm}} ) x_{j0}^{\pm} x_{i0}^{\pm} = \frac{ (-1)^{|\alpha_{i}||\alpha_{j}|} e^{\sigma_{i}^{\pm} \pm a \hbar } - e^{\sigma_{j}^{\pm}}}{\sigma_{i}^{\pm} - \sigma_{j}^{\pm} \pm a \hbar} ( \sigma_{i}^{\pm} - \sigma_{j}^{\pm} \pm a \hbar ) x_{j0}^{\pm} x_{i0}^{\pm} = \]
	\[ \frac{ e^{\sigma_{j}^{\pm}} - (-1)^{|\alpha_{i}||\alpha_{j}|} e^{\sigma_{i}^{\pm} \pm a \hbar } }{\sigma_{j}^{\pm} - \sigma_{i}^{\pm} \mp a \hbar} ( \sigma_{i}^{\pm} - \sigma_{j}^{\pm} \pm a \hbar ) x_{j0}^{\pm} x_{i0}^{\pm}. \]
	The PBW Theorem \ref{th:pbwyang} shows that the above is equivalent to \eqref{eq:ccomp}.
	
	Assume now that $i = j$. Recall that in this case $|\alpha_{i}| = \bar{0}$ as $c_{ii} \ne 0$ by \eqref{QL4}. Then
	\[ \Phi( X^{\pm}_{ir} ) \Phi ( X^{\pm}_{is} ) = ( e^{r \sigma_{i}^{\pm}} g_{i}^{\pm} ( \sigma_{i}^{\pm} ) x_{i0}^{\pm} ) ( e^{s \sigma_{i}^{\pm}} g_{i}^{\pm} ( \sigma_{i}^{\pm} ) x_{i0}^{\pm} ) = \]
	\[ \mu( e^{r \sigma_{i,(1)}^{\pm} } e^{s \sigma_{i,(2)}^{\pm}} g_{i}^{\pm} ( \sigma_{i,(1)}^{\pm} ) \lambda_{i}^{\pm} ( \sigma_{i,(1)}^{\pm} ) ( g_{i}^{\pm} ( \sigma_{i, (2)}^{\pm} ) ) x_{i0}^{\pm} \otimes x_{i0}^{\pm} ). \]
	The compatibility with \eqref{QL4} therefore reduces to
	\[ \mu( e^{k \sigma_{i,(1)}^{\pm} } e^{l \sigma_{i,(2)}^{\pm}} g_{i}^{\pm} ( \sigma_{i,(1)}^{\pm} ) \lambda_{i}^{\pm} ( \sigma_{i,(1)}^{\pm} ) ( g_{i}^{\pm} ( \sigma_{i, (2)}^{\pm} ) ) ( e^{\sigma_{i, (1)}^{\pm} } - e^{\sigma_{i, (2)}^{\pm} \pm a \hbar} ) x_{i0}^{\pm} \otimes x_{i0}^{\pm} ) = \]
	\[ \mu( e^{l \sigma_{i,(1)}^{\pm} } e^{k \sigma_{i,(2)}^{\pm}} g_{i}^{\pm} ( \sigma_{i,(1)}^{\pm} ) \lambda_{i}^{\pm} ( \sigma_{i,(1)}^{\pm} ) ( g_{i}^{\pm} ( \sigma_{i, (2)}^{\pm} ) ) ( e^{\sigma_{i,(2)}^{\pm} \pm a \hbar } -  e^{\sigma_{i,(1)}^{+}} ) x_{i0}^{\pm} \otimes x_{i0}^{\pm} ) \]
	that is, to
	\[ \mu ( ( e^{k \sigma_{i,(1)}^{\pm} } e^{l \sigma_{i,(2)}^{\pm}} + e^{l \sigma_{i,(1)}^{\pm} } e^{k \sigma_{i,(2)}^{\pm}} ) g_{i}^{\pm} ( \sigma_{i,(1)}^{\pm} ) \lambda_{i}^{\pm} ( \sigma_{i,(1)}^{\pm} ) ( g_{i}^{\pm} ( \sigma_{i, (2)}^{\pm} ) ) \cdot \]
	\[ ( e^{\sigma_{i, (1)}^{\pm} } - e^{\sigma_{i, (2)}^{\pm} \pm a \hbar} ) x_{i0}^{\pm} \otimes x_{i0}^{\pm} ) = 0. \]
	By \eqref{eq:rel2} of Proposition \ref{pr:y46eqc} and Corollary \ref{cl:Bseq}, this equation is equivalent to the requirement that
	\[ \mu ( ( e^{k \sigma_{i,(1)}^{\pm} } e^{l \sigma_{i,(2)}^{\pm}} + e^{l \sigma_{i,(1)}^{\pm} } e^{k \sigma_{i,(2)}^{\pm}} ) g_{i}^{\pm} ( \sigma_{i,(1)}^{\pm} ) \lambda_{i}^{\pm} ( \sigma_{i,(1)}^{\pm} ) ( g_{i}^{\pm} ( \sigma_{i, (2)}^{\pm} ) ) \frac{e^{\sigma_{i, (1)}^{\pm} } -  e^{\sigma_{i, (2)}^{\pm} \pm a \hbar}}{\sigma_{i,(1)}^{\pm} - \sigma_{i,(2)}^{\pm} \mp a \hbar} \]
	\[ (\sigma_{i,(1)}^{\pm} - \sigma_{i,(2)}^{\pm} \mp a \hbar) x_{i0}^{\pm} \otimes x_{i0}^{\pm} ) =  \]
	\[ \mu ( ( e^{k \sigma_{i,(2)}^{\pm} } e^{l \sigma_{i,(1)}^{\pm}} + e^{l \sigma_{i,(2)}^{\pm} } e^{k \sigma_{i,(1)}^{\pm}} ) g_{i}^{\pm} ( \sigma_{i,(2)}^{\pm} ) \lambda_{i}^{\pm} ( \sigma_{i,(2)}^{\pm} ) ( g_{i}^{\pm} ( \sigma_{i, (1)}^{\pm} ) ) \frac{e^{\sigma_{i, (2)}^{\pm} } -  e^{\sigma_{i, (1)}^{\pm} \pm a \hbar}}{\sigma_{i,(2)}^{\pm} - \sigma_{i,(1)}^{\pm} \mp a \hbar} \]
	\[ (\sigma_{i,(1)}^{\pm} - \sigma_{i,(2)}^{\pm} \mp a \hbar) x_{i0}^{\pm} \otimes x_{i0}^{\pm} ) =  \]
	\[ \mu ( ( e^{k \sigma_{i,(1)}^{\pm}} e^{l \sigma_{i,(2)}^{\pm} } + e^{l \sigma_{i,(1)}^{\pm}} e^{k \sigma_{i,(2)}^{\pm} } ) g_{i}^{\pm} ( \sigma_{i,(2)}^{\pm} ) \lambda_{i}^{\pm} ( \sigma_{i,(2)}^{\pm} ) ( g_{i}^{\pm} ( \sigma_{i, (1)}^{\pm} ) ) \frac{ e^{\sigma_{i, (2)}^{\pm} } - e^{\sigma_{i, (1)}^{\pm} \pm a \hbar}}{\sigma_{i,(2)}^{\pm} - \sigma_{i,(1)}^{\pm} \mp a \hbar} \]
	\[ (\sigma_{i,(1)}^{\pm} - \sigma_{i,(2)}^{\pm} \mp a \hbar) x_{i0}^{\pm} \otimes x_{i0}^{\pm} ) = 0. \]
	
	Note that we don't change indices in $(\sigma_{i,(1)}^{\pm} - \sigma_{i,(2)}^{\pm} \mp a \hbar) x_{i0}^{\pm} \otimes x_{i0}^{\pm}$. Thus we get condition \eqref{eq:ccomp} for $i=j$.
\end{proof}

\begin{lemma}
	\label{lm:ql6comp}
	$\Phi$ is compatible with the relation \eqref{QL6} if and only if \eqref{eq:qcomp} holds.
\end{lemma}
\begin{proof}
	In order to give simultaneously one proof for both cases we set $X^{+}_{ir} = E_{ir}$ and $X^{-}_{ir} = F_{ir}$ for any $i \in I$ and $r \in \mathbb{Z}$. Compatibility with \eqref{QL6} reads
	\[ \Phi( X^{ \pm}_{i,k} ) \Phi( X^{ \pm}_{j,l} ) - (-1)^{|\alpha_{i}||\alpha_{j}|} \Phi ( X^{ \pm}_{j,l} ) \Phi ( X^{ \pm}_{i,k} ) = 0. \]
	
	Since
	\[ \Phi( X^{\pm}_{ir} ) \Phi( X^{ \pm}_{js} ) = e^{r \sigma_{i}^{ \pm}} e^{s \sigma_{j}^{ \pm}} g_{i}^{ \pm}( \sigma_{i}^{ \pm}) \lambda_{i}^{ \pm}( \sigma_{i}^{ \pm}) ( g_{j}^{ \pm}( \sigma_{j}^{ \pm} ) ) x_{i0}^{ \pm} x_{j0}^{ \pm}  \]
	the above reduces to for $i \ne j$
	\[ e^{k \sigma_{i}^{ \pm} } e^{ l \sigma_{j}^{ \pm} } g_{i}^{ \pm} ( \sigma_{i}^{ \pm} ) \lambda_{i}^{ \pm} ( \sigma_{i}^{ \pm} ) ( g_{j}^{ \pm} ( \sigma_{j}^{ \pm} ) ) x_{i0}^{\pm} x_{j0}^{\pm} = \]
	\[ (-1)^{|\alpha_{i}||\alpha_{j}|} e^{k \sigma_{i}^{ \pm} } e^{ l \sigma_{j}^{ \pm} } g_{j}^{ \pm} ( \sigma_{j}^{ \pm} ) \lambda_{j}^{ \pm} ( \sigma_{j}^{ \pm} ) ( g_{i}^{ \pm} ( \sigma_{i}^{ \pm} ) ) x_{j0}^{\pm} x_{i0}^{\pm} \iff \]
	\[ g_{i}^{ \pm} ( \sigma_{i}^{ \pm} ) \lambda_{i}^{ \pm} ( \sigma_{i}^{ \pm} ) ( g_{j}^{ \pm} ( \sigma_{j}^{ \pm} ) ) x_{i0}^{\pm} x_{j0}^{\pm} =  \]
	\[ (-1)^{|\alpha_{i}||\alpha_{j}|} g_{j}^{ \pm} ( \sigma_{j}^{ \pm} ) \lambda_{j}^{ \pm} ( \sigma_{j}^{ \pm} ) ( g_{i}^{ \pm} ( \sigma_{i}^{ \pm} ) ) x_{j0}^{\pm} x_{i0}^{\pm}. \]
	By \eqref{eq:rel4} of Proposition \ref{pr:y46eqc} as $c_{ij} = 0$ we have
	\[ g_{i}^{ \pm} ( \sigma_{i}^{ \pm} ) \lambda_{i}^{ \pm} ( \sigma_{i}^{ \pm} ) ( g_{j}^{ \pm} ( \sigma_{j}^{ \pm} ) ) x_{i0}^{\pm} x_{j0}^{\pm} =  \]
	\[ (-1)^{|\alpha_{i}||\alpha_{j}|} g_{i}^{ \pm} ( \sigma_{i}^{ \pm} ) \lambda_{i}^{ \pm} ( \sigma_{i}^{ \pm} ) ( g_{j}^{ \pm} ( \sigma_{j}^{ \pm} ) ) x_{j0}^{\pm} x_{i0}^{\pm} = \]
	\[ (-1)^{|\alpha_{i}||\alpha_{j}|} g_{j}^{ \pm} ( \sigma_{j}^{ \pm} ) \lambda_{j}^{ \pm} ( \sigma_{j}^{ \pm} ) ( g_{i}^{ \pm} ( \sigma_{i}^{ \pm} ) ) x_{j0}^{\pm} x_{i0}^{\pm}. \]
	The PBW Theorem \ref{th:pbwyang} shows that the above is equivalent to \eqref{eq:qcomp}.
	
	Suppose now that $i = j$. Recall that in this case $c_{ii} = 0$ and $|\alpha_{i}| = \bar{1}$. Then 
	\[ \Phi( X^{\pm}_{ir} ) \Phi( X^{ \pm}_{is} ) = \]
	\[ \mu( e^{r \sigma_{i,(1)}^{\pm} } e^{s \sigma_{i,(2)}^{\pm}} g_{i}^{\pm} ( \sigma_{i,(1)}^{\pm} ) \lambda_{i}^{\pm} ( \sigma_{i,(1)}^{\pm} ) ( g_{i}^{\pm} ( \sigma_{i, (2)}^{\pm} ) ) x_{i0}^{\pm} \otimes x_{i0}^{\pm} ). \]
	The compatibility with \eqref{QL6} therefore reduces to
	\[ \mu( e^{k \sigma_{i,(1)}^{\pm} } e^{l \sigma_{i,(2)}^{\pm}} g_{i}^{\pm} ( \sigma_{i,(1)}^{\pm} ) \lambda_{i}^{\pm} ( \sigma_{i,(1)}^{\pm} ) ( g_{i}^{\pm} ( \sigma_{i, (2)}^{\pm} ) )  x_{i0}^{\pm} \otimes x_{i0}^{\pm} ) = \]
	\[ - \mu( e^{l \sigma_{i,(1)}^{\pm} } e^{k \sigma_{i,(2)}^{\pm}} g_{i}^{\pm} ( \sigma_{i,(1)}^{\pm} ) \lambda_{i}^{\pm} ( \sigma_{i,(1)}^{\pm} ) ( g_{i}^{\pm} ( \sigma_{i, (2)}^{\pm} ) ) x_{i0}^{\pm} \otimes x_{i0}^{\pm} ), \]
	that is to
	\[ \mu ( ( e^{k \sigma_{i,(1)}^{\pm} } e^{l \sigma_{i,(2)}^{\pm}} + e^{l \sigma_{i,(1)}^{\pm} } e^{k \sigma_{i,(2)}^{\pm}} ) g_{i}^{\pm} ( \sigma_{i,(1)}^{\pm} ) \lambda_{i}^{\pm} ( \sigma_{i,(1)}^{\pm} ) ( g_{i}^{\pm} ( \sigma_{i, (2)}^{\pm} ) ) x_{i0}^{\pm} \otimes x_{i0}^{\pm} ) = 0. \]
	By \eqref{eq:rel40} of Proposition \ref{pr:y46eqc}, this equation is equivalent to the requirement that
	\[ \mu ( ( e^{k \sigma_{i,(1)}^{\pm} } e^{l \sigma_{i,(2)}^{\pm}} + e^{l \sigma_{i,(1)}^{\pm} } e^{k \sigma_{i,(2)}^{\pm}} ) g_{i}^{\pm} ( \sigma_{i,(1)}^{\pm} ) \lambda_{i}^{\pm} ( \sigma_{i,(1)}^{\pm} ) ( g_{i}^{\pm} ( \sigma_{i, (2)}^{\pm} ) ) x_{i0}^{\pm} \otimes x_{i0}^{\pm} ) = \]
	\[ \mu ( ( e^{k \sigma_{i,(2)}^{\pm} } e^{l \sigma_{i,(1)}^{\pm}} + e^{l \sigma_{i,(2)}^{\pm} } e^{k \sigma_{i,(1)}^{\pm}} ) g_{i}^{\pm} ( \sigma_{i,(2)}^{\pm} ) \lambda_{i}^{\pm} ( \sigma_{i,(2)}^{\pm} ) ( g_{i}^{\pm} ( \sigma_{i, (1)}^{\pm} ) ) x_{i0}^{\pm} \otimes x_{i0}^{\pm} ) = 0. \]
	Thus
	\[ g_{i}^{\pm} ( \sigma_{i,(1)}^{\pm} ) \lambda_{i}^{\pm} ( \sigma_{i,(1)}^{\pm} ) ( g_{i}^{\pm} ( \sigma_{i, (2)}^{\pm} ) ) = g_{i}^{\pm} ( \sigma_{i,(2)}^{\pm} ) \lambda_{i}^{\pm} ( \sigma_{i,(2)}^{\pm} ) ( g_{i}^{\pm} ( \sigma_{i, (1)}^{\pm} ) ) \]
	and we get condition \eqref{eq:qcomp} for $i=j$.
\end{proof}

\begin{lemma}
	\label{lm:ql7comp}
	$\Phi$ is compatible with the relation \eqref{QL7} by construction.
\end{lemma}
\begin{proof}
	In order to give simultaneously one proof for both cases we set $X^{+}_{ir} = E_{ir}$ and $X^{-}_{ir} = F_{ir}$ for any $i \in I$ and $r \in \mathbb{Z}$. Compatibility with \eqref{QL7} reads
	\begin{equation}
		\label{eq:trisers}
		[ \Phi( X^{\pm}_{i,r} ), [ \Phi( X^{\pm}_{i,k} ), \Phi( X^{\pm}_{j,l} ) ]_{q^{\pm c_{ij}}} ]_{q^{\mp c_{ij}}} + [ \Phi( X^{\pm}_{i,k} ), [ \Phi( X^{\pm}_{i,r} ), \Phi( X^{\pm}_{j,l} ) ]_{q^{\pm c_{ij}}} ]_{q^{\mp c_{ij}}} = 0.
	\end{equation}
	
	Set $a = c_{ij} / 2$, so that $q^{\pm c_{ij}} = e^{\pm a \hbar}$ for all $i,j \in I$. Denote by $u = \sigma_{i}^{ \pm}$ and $v = \sigma_{j}^{ \pm}$.
	
	Since
	\[ \Phi( X^{\pm}_{it} ) \Phi( X^{ \pm}_{js} ) = e^{t u} e^{s v} g_{i}^{ \pm}( u) \lambda_{i}^{ \pm}( u) ( g_{j}^{ \pm}( v ) ) x_{i0}^{ \pm} x_{j0}^{ \pm}  \]
	we have
	\[ [ \Phi( X^{\pm}_{i,t} ), \Phi( X^{\pm}_{j,s} ) ]_{e^{\pm a \hbar}} = e^{t u} e^{s v} [ \Phi( X^{\pm}_{i,0} ), \Phi( X^{\pm}_{j,0} ) ]_{e^{\pm a \hbar}}. \]
	Thus
	\[ [ \Phi( X^{\pm}_{i,z} ) , [ \Phi( X^{\pm}_{i,t} ), \Phi( X^{\pm}_{j,s} ) ]_{e^{\pm a \hbar}} ]_{e^{\mp a \hbar}} = \]
	\[ \mu ( e^{z u_{(1)}} e^{t u_{(2)}} e^{s v_{(2)}}  \Phi( X^{\pm}_{i,0} ) \otimes [ \Phi( X^{\pm}_{i,0} ), \Phi( X^{\pm}_{j,0} ) ]_{e^{\pm a \hbar}} + \]
	\[ (-1) e^{\mp a \hbar} e^{z u_{(2)}} e^{t u_{(1)}} e^{s v_{(1)}} [ \Phi( X^{\pm}_{i,0} ), \Phi( X^{\pm}_{j,0} ) ]_{e^{\pm a \hbar}} ]_{e^{\mp a \hbar}} \otimes \Phi( X^{\pm}_{i,0} ) ).  \]
	
	In this way the equation \eqref{eq:trisers} gets the following form
	\[ \mu ( ( e^{r u_{(1)}} e^{k u_{(2)}} + e^{k u_{(1)}} e^{r u_{(2)}} ) e^{l v_{(2)}} \Phi( X^{\pm}_{i,0} ) \otimes [ \Phi( X^{\pm}_{i,0} ), \Phi( X^{\pm}_{j,0} ) ]_{e^{\pm a \hbar}} + \]
	\[ (-1) e^{\mp a \hbar} ( e^{k u_{(1)}} e^{r u_{(2)}} + e^{r u_{(1)}} e^{k u_{(2)}} ) e^{l v_{(1)}}  [ \Phi( X^{\pm}_{i,0} ), \Phi( X^{\pm}_{j,0} ) ]_{e^{\pm a \hbar}} ]_{e^{\mp a \hbar}} \otimes \Phi( X^{\pm}_{i,0} ) ) = 0. \]
	It is easy to see that this equation is in fact
	\[ [ \Phi( X^{\pm}_{i,0} ) , [ \Phi( X^{\pm}_{i,0} ), \Phi( X^{\pm}_{j,l} ) ]_{e^{\pm a \hbar}} ]_{e^{\mp a \hbar}} = 0 \iff \]
	\[ [ \Phi( X^{\pm}_{i,0} ) , [ \Phi( X^{\pm}_{i,0} ), \Phi( X^{\pm}_{j,0} ) ]_{e^{\pm a \hbar}} ]_{e^{\mp a \hbar}} = 0. \]
	The equivalence follows by the definition of the function $\Phi$. Thus we need to prove that
	\[ 	0 = g_{i}^{ \pm}( u ) x_{i0}^{ \pm} g_{i}^{ \pm}(u) x_{i0}^{ \pm} g_{j}^{ \pm}( v ) x_{j0}^{ \pm} + \]
	\[ (-1) ( e^{\pm \hbar a} + e^{\mp \hbar a} ) g_{i}^{ \pm}(u) x_{i0}^{ \pm} g_{j}^{ \pm}( v ) x_{j0}^{ \pm} g_{i}^{ \pm}( u) x_{i0}^{ \pm} + \]
	\[ g_{j}^{ \pm}( v ) x_{j0}^{ \pm} g_{i}^{ \pm}( u) x_{i0}^{ \pm} g_{i}^{ \pm}( u) x_{i0}^{ \pm}. \]
	
	By \eqref{eq:rel1} of Proposition \ref{pr:y46eqc} and \eqref{eq:ccomp} of Theorem \ref{th:thcondhom} we have
	\[ g_{i}^{ \pm}( u ) \lambda_{i}^{ \pm} (u) ( g_{j}^{ \pm} ( v ) ) x_{i0}^{\pm} x_{j0}^{\pm} =  \]
	\[ g_{j}^{ \pm}( v ) \lambda_{j}^{ \pm} (v) ( g_{i}^{ \pm} ( u ) ) ( \frac{e^{u \pm a \hbar } - e^{v} }{ e^{u } - e^{v \pm a \hbar } } ) x_{j0}^{\pm} x_{i0}^{\pm}. \]
	Thus
	\[ 0 = g_{i}^{ \pm}( u ) x_{i0}^{ \pm} ( \frac{e^{u \pm a \hbar } - e^{v} }{ e^{u } - e^{v \pm a \hbar } } - e^{\pm \hbar a} ) ( g_{j}^{ \pm}( v ) x_{j0}^{ \pm} g_{i}^{ \pm}( u) x_{i0}^{ \pm} )+ \]
	\[ ( \frac{ e^{u } - e^{v \pm a \hbar } }{e^{u \pm a \hbar } - e^{v} } - e^{\mp \hbar a} ) ( g_{i}^{ \pm}( u) x_{i0}^{ \pm} g_{j}^{ \pm}( v ) x_{j0}^{ \pm} ) g_{i}^{ \pm}( u) x_{i0}^{ \pm} \iff \]
	\[ g_{i}^{\pm}( u ) x_{i0}^{ \pm} \frac{1}{e^{u-v \mp a \hbar} - 1} ( g_{j}^{\pm}( v ) x_{j0}^{ \pm} g_{i}^{\pm}( u) x_{i0}^{ \pm} ) = \]
	\[ \frac{1}{e^{u-v \pm a \hbar} - 1} ( g_{i}^{\pm}( u) x_{i0}^{ \pm} g_{j}^{\pm}( v ) x_{j0}^{ \pm} ) g_{i}^{\pm}( u) x_{i0}^{ \pm} \iff \]
	\[ \mu (\mu \otimes \id) ( g_{i}^{\pm} ( u_{(1)} ) \lambda_{i}^{\pm} ( u_{(1)} ) ( g_{j}^{\pm} ( v_{(2)} )  ) \lambda_{i}^{\pm} (u_{(1)}) ( \lambda_{j}^{\pm} ( v_{(2)} ) ( g_{i}^{\pm} ( u_{(3)} )  ) ) \circ \]
	\[  ( \frac{1}{e^{u_{(3)}-v_{(2)} \mp a \hbar} - 1} - \frac{1}{e^{u_{(1)}-v_{(2)} \pm a \hbar} - 1} ) ( x_{i0}^{\pm} \otimes x_{j0}^{\pm} \otimes x_{i0}^{\pm} ) ) = 0 \iff \]
	\[ \mu (\mu \otimes \id) (e^{u_{(1)}-v_{(2)} \pm a \hbar} - e^{u_{(3)}-v_{(2)} \mp a \hbar} ) ( x_{i0}^{\pm} \otimes x_{j0}^{\pm} \otimes x_{i0}^{\pm} ) = 0 \iff \]
	\[ \mu (\mu \otimes \id) \sum_{n \ge 0} \frac{(u_{(1)}-v_{(2)} \pm a \hbar)^{n} - (u_{(3)}-v_{(2)} \mp a \hbar)^{n}}{n!} ( x_{i0}^{\pm} \otimes x_{j0}^{\pm} \otimes x_{i0}^{\pm} ) = 0. \]
	By \eqref{eq:rel5} in Proposition \ref{pr:y46eqc}, the easy to verify by direct computations fact that
	\[  (u_{(1)}-v_{(2)} \pm a \hbar) (u_{(3)}-v_{(2)} \mp a \hbar) ( x_{i0}^{\pm} \otimes x_{j0}^{\pm} \otimes x_{i0}^{\pm} ) = \]
	\[ (u_{(3)}-v_{(2)} \mp a \hbar) (u_{(1)}-v_{(2)} \pm a \hbar) ( x_{i0}^{\pm} \otimes x_{j0}^{\pm} \otimes x_{i0}^{\pm} ) \]
	and the mathematical induction
	\[ \mu (\mu \otimes \id) ( (u_{(1)}-v_{(2)} \pm a \hbar)^{n} - (u_{(3)}-v_{(2)} \mp a \hbar)^{n} ) ( x_{i0}^{\pm} \otimes x_{j0}^{\pm} \otimes x_{i0}^{\pm} ) = 0 \]
	for all $n \in \mathbb{N}_{0}$.	The result follows.	
\end{proof}

\begin{lemma}
	\label{lm:ql8comp}
	$\Phi$ is compatible with the relation \eqref{QL8} by construction.
\end{lemma}
\begin{proof}
	In order to give simultaneously one proof for both cases we set $X^{+}_{ir} = E_{ir}$ and $X^{-}_{ir} = F_{ir}$ for any $i \in I$ and $r \in \mathbb{Z}$. Set $a = c_{ij} / 2$, so that $q^{\pm c_{ij}} = e^{\pm a \hbar}$ for all $i,j \in I$. Denote by $v = \sigma_{i-1}^{ \pm}$, $u = \sigma_{i}^{ \pm}$ and $z =  \sigma_{i+1}^{ \pm}$.
	
	Taking into account the proof of Lemma \ref{lm:ql7comp} we have
	\[ [[[ \Phi( X^{\pm}_{i-1,r} ), \Phi ( X^{\pm}_{i,s} )]_{q^{\pm c^d_{ij}}} , \Phi( X^{\pm}_{i+1,k} )]_{q^{\mp c^d_{ij}}}, \Phi( X^{\pm}_{i,l} )] = \]
	\[ \mu ( e^{r v_{(1)}} e^{s u_{(1)}} e^{k z_{(1)}} e^{l u_{(2)}} [[ \Phi( X^{\pm}_{i-1,0} ), \Phi ( X^{\pm}_{i,0} )]_{e^{\pm a \hbar }} , \Phi( X^{\pm}_{i+1,0} )]_{e^{\mp a \hbar}} \otimes \Phi( X^{\pm}_{i,0} ) + \]
	\[ (-1)^{|\alpha_{i-1}| + |\alpha_{i+1}|} e^{r v_{(2)}} e^{s u_{(2)}} e^{k z_{(2)}} e^{l u_{(1)}} \Phi( X^{\pm}_{i,0} ) \otimes [[ \Phi( X^{\pm}_{i-1,0} ), \Phi ( X^{\pm}_{i,0} )]_{e^{\pm a \hbar }} , \Phi( X^{\pm}_{i+1,0} )]_{e^{\mp a \hbar}} ). \]
	
	Thus the compatibility with \eqref{QL8} reads
	\[ \mu ( e^{r v_{(1)}} e^{k z_{(1)}} ( e^{s u_{(1)}} e^{l u_{(2)}} + e^{l u_{(1)}} e^{s u_{(2)}} ) [[ \Phi( X^{\pm}_{i-1,0} ), \Phi ( X^{\pm}_{i,0} )]_{e^{\pm a \hbar }} , \Phi( X^{\pm}_{i+1,0} )]_{e^{\mp a \hbar}} \otimes \Phi( X^{\pm}_{i,0} ) + \]
	\[ (-1)^{|\alpha_{i-1}| + |\alpha_{i+1}|} e^{r v_{(2)}} e^{k z_{(2)}} ( e^{s u_{(2)}} e^{l u_{(1)}} + e^{l u_{(2)}} e^{s u_{(1)}} ) \cdot \]
	\[ \Phi( X^{\pm}_{i,0} ) \otimes [[ \Phi( X^{\pm}_{i-1,0} ), \Phi ( X^{\pm}_{i,0} )]_{e^{\pm a \hbar }} , \Phi( X^{\pm}_{i+1,0} )]_{e^{\mp a \hbar}} ) = 0. \]
	It is easy to see that this equation is in fact
	\[ [[[ \Phi( X^{\pm}_{i-1,r} ), \Phi ( X^{\pm}_{i,0} )]_{e^{\pm a \hbar}} , \Phi( X^{\pm}_{i+1,k} )]_{e^{\mp a \hbar}}, \Phi( X^{\pm}_{i,0} )] = 0 \iff \]
	\[ [[[ \Phi( X^{\pm}_{i-1,0} ), \Phi ( X^{\pm}_{i,0} )]_{e^{\pm a \hbar}} , \Phi( X^{\pm}_{i+1,0} )]_{e^{\mp a \hbar}}, \Phi( X^{\pm}_{i,0} )] = 0. \]
	The equivalence follows by the definition of the function $\Phi$. By \eqref{eq:rel1} in Proposition \ref{pr:y46eqc} and \eqref{eq:ccomp} in Theorem \ref{th:thcondhom} and direct computations we get the equivalent form of the last equation
	\[ ( \frac{ e^{\pm a \hbar } - e^{\mp a \hbar } }{ e^{v - u \mp a \hbar } - 1 } + e^{ \pm a \hbar } \frac{ e^{u - z \mp a \hbar} -1 }{  e^{u - z \pm a \hbar } - 1 }  ) ( g_{i}^{\pm}(u) x_{i0}^{\pm} g_{i-1}^{\pm}(v) x_{i-1,0}^{\pm} g_{i+1}^{\pm}(z) x_{i+1,0}^{\pm}) g_{i}^{\pm}(u) x_{i0}^{\pm} + \]
	\[ g_{i}^{\pm}(u) x_{i0}^{\pm} ( \frac{ e^{\mp a \hbar } - e^{\pm a \hbar } }{e^{v - u \pm a \hbar } - 1 } + e^{ \mp a \hbar } \frac{e^{u - z \pm a \hbar } - 1 }{ e^{u - z \mp a \hbar } - 1 }  )  ( g_{i-1}^{\pm}(v) x_{i-1,0}^{\pm} g_{i+1}^{\pm}(z) x_{i+1,0}^{\pm} g_{i}^{\pm}(u) x_{i0}^{\pm}) = 0 \iff \]
	
	\[ \mu (\mu \otimes \id) (\mu \otimes \id \otimes \id) g_{i}^{\pm} ( u_{(1)} ) \lambda_{i}^{\pm} ( u_{(1)} ) ( g_{i-1}^{\pm} ( v_{(2)} )  ) \circ \]
	\[ \lambda_{i}^{\pm} (u_{(1)}) ( \lambda_{i-1}^{\pm} ( v_{(2)} ) ( g_{i+1}^{\pm} ( z_{(3)} )  ) ) \lambda_{i}^{\pm} (u_{(1)}) ( \lambda_{i-1}^{\pm} ( v_{(2)} ) (  \lambda_{i+1}^{\pm} ( z_{(3)} ) ( g_{i+1}^{\pm} (u_{(4)}) )  ) ) \circ \]
	\[ ( ( ( e^{\pm a \hbar } - e^{\mp a \hbar } ) ( e^{u_{(1)} - z_{(3)} \pm a \hbar } - 1 ) + e^{ \pm a \hbar } ( e^{u_{(1)} - z_{(3)} \mp a \hbar} -1 ) ( e^{v_{(2)} - u_{(1)} \mp a \hbar } - 1 ) ) \cdot \]
	\[ ( e^{v_{(2)} - u_{(4)} \pm a \hbar } - 1  ) ( e^{u_{(4)} - z_{(3)} \mp a \hbar } - 1 ) + \]
	\[ ( ( e^{\mp a \hbar } - e^{\pm a \hbar } ) ( e^{u_{(4)} - z_{(3)} \mp a \hbar } - 1 ) + e^{ \mp a \hbar } ( e^{u_{(4)} - z_{(3)} \pm a \hbar} -1 ) ( e^{v_{(2)} - u_{(4)} \pm a \hbar } - 1 ) ) \cdot \]
	\[ ( e^{v_{(2)} - u_{(1)} \mp a \hbar } - 1  ) ( e^{u_{(1)} - z_{(3)} \pm a \hbar } - 1 ) ) \]
	\[ x_{i0}^{\pm} \otimes x_{i-1,0}^{\pm} \otimes x_{i+1,0}^{\pm} \otimes x_{i0}^{\pm} = 0. \]
	By \eqref{eq:rel5} in Proposition \ref{pr:y46eqc} and the mathematical induction we get
	\[  \sum_{ n \ge 0, n_1, n_2, n_3, n_4 \ge 1} \frac{ (a \hbar)^{n} (( \pm )^{n} + ( \mp )^{n})}{n! n_{1}! n_3! n_4!} \cdot \]
	\[ ( ( 1 - ( -1 )^{n})  (u_{(1) } - z_{(3)} \pm a \hbar)^{n_1} + ( u_{(1)} - z_{(3) } \mp a \hbar  )^{n_1} \frac{ ( v_{(2)}- u_{(1)} \mp a \hbar  )^{n_2}  }{ n_{2}! }  ) \circ  \]
	\[ ( v_{(2)} -u_{(4)} \pm a \hbar )^{n_3}  ( u_{(4)} -z_{(3)} \mp a \hbar )^{n_4} \]
	\[ x_{i0}^{\pm} \otimes x_{i-1,0}^{\pm} \otimes x_{i+1,0}^{\pm} \otimes x_{i0}^{\pm} = 0 \iff \]
	\[  \sum_{ n \in 2 \mathbb{N}_{0}, n_1, n_2, n_3, n_4 \ge 1} \frac{ 2 (a \hbar)^{n} }{n! n_{1}! n_{2}! n_3! n_4!} \cdot \]
	\[  ( u_{(1)} - z_{(3) } \mp a \hbar  )^{n_1} ( v_{(2)}- u_{(1)} \mp a \hbar  )^{n_2} ( v_{(2)} -u_{(4)} \pm a \hbar )^{n_3}  ( u_{(4)} -z_{(3)} \mp a \hbar )^{n_4} \]
	\[ x_{i0}^{\pm} \otimes x_{i-1,0}^{\pm} \otimes x_{i+1,0}^{\pm} \otimes x_{i0}^{\pm} = 0. \]
	By \eqref{eq:rel6} in Proposition \ref{pr:y46eqc} and the mathematical induction we get for all $n_1, n_2, n_3, n_4 \ge 1$ that
	\[  ( u_{(1)} - z_{(3) } \mp a \hbar  )^{n_1} ( v_{(2)}- u_{(1)} \mp a \hbar  )^{n_2} ( v_{(2)} -u_{(4)} \pm a \hbar )^{n_3}  ( u_{(4)} -z_{(3)} \mp a \hbar )^{n_4} \]
	\[ x_{i0}^{\pm} \otimes x_{i-1,0}^{\pm} \otimes x_{i+1,0}^{\pm} \otimes x_{i0}^{\pm} = 0. \]
\end{proof}

Now we are ready to prove
\begin{theorem}
	\label{th:gsercindtrue}
	The series $g_{i}^{\pm}(v) = g_{i}(v)$ (see \eqref{eq:gipm} and \eqref{eq:giunvsol}) satisfy the conditions \eqref{eq:acomp} - \eqref{eq:qcomp} of Theorem \ref{th:thcondhom}, and therefore give rise to a homomorphism of superalgebras $\Phi: {U}_{\hbar}(L\mathfrak{g}) \to \widehat{Y_{\hbar}(\mathfrak{g})} $.
\end{theorem}
\begin{proof}
	The proof of conditions \eqref{eq:acomp} and \eqref{eq:bcomp} is literally the same as in \cite{GT13}. We prove the condition \eqref{eq:ccomp}.
	
	Let $i, j \in I$ and set $a = c_{ij} / 2$. We need to prove that
	\[ g_{i}(u) \lambda_{i}^{ \pm} (u) ( g_{j} (v) ) ( \frac{ (-1)^{|\alpha_i||\alpha_j|} e^{u} - e^{v \pm a \hbar}}{u-v \mp a \hbar} ) = g_{j}(v) \lambda_{j}^{ \pm} (v) ( g_{i} (u) ) ( \frac{  e^{v} - (-1)^{|\alpha_i||\alpha_j|} e^{u \pm a \hbar} }{v-u \mp a \hbar} ). \]
	By Lemmas \ref{lm:lgG}, \ref{lm:invgpm} and the fact that $G$ is even (see \eqref{eq:Gdef}), we get the following equivalent assertion
	\begin{equation}
		\label{eq:GEfrc}
		\exp( G( v - u \pm a \hbar) - G( v - u \mp a \hbar ) ) \frac{ (-1)^{|\alpha_i||\alpha_j|} e^{u} - e^{v \pm a \hbar}}{u-v \mp a \hbar} = \frac{ e^{v} - (-1)^{|\alpha_i||\alpha_j|} e^{u \pm a \hbar} }{v-u \mp a \hbar}.
	\end{equation}
	Using \eqref{eq:Gdef} we get
	\[ ( \frac{v-u \pm a \hbar}{e^{v \pm a \hbar} - e^u } ) ( \frac{e^v - e^{u \pm a \hbar}}{v-u \mp a \hbar} ) ( \frac{ (-1)^{|\alpha_i||\alpha_j|} e^{u} - e^{v \pm a \hbar}}{u-v \mp a \hbar} ) = \]
	\[ (\frac{e^v - e^{u \pm a \hbar}}{v-u \mp a \hbar}) (\frac{ (-1)^{|\alpha_i||\alpha_j|} e^{u} - e^{v \pm a \hbar}}{e^u - e^{v \pm a \hbar}} ) = \frac{ e^{v} - (-1)^{|\alpha_i||\alpha_j|} e^{u \pm a \hbar} }{v-u \mp a \hbar}. \]
	The equation holds as $|\alpha_{i}| = \bar{0}$ or $|\alpha_{j}| = \bar{0}$, if $i \ne  j$, according to our requirements for Dynkin diagrams (see reasoning before Definition \ref{phidef}); $|\alpha_{i}| = \bar{0}$, if $i = j$. Thus the left side of the equation \eqref{eq:GEfrc} coincides with the right side. The result follows.
	
	We prove the condition \eqref{eq:qcomp}. Let $i, j \in I$. We need to prove that
	\[ g_{i} ( u ) \lambda_{i}^{ \pm} ( u ) ( g_{j} ( v ) ) = g_{j} ( v ) \lambda_{j}^{ \pm} ( v ) ( g_{i} ( u ) ). \]
	By Lemmas \ref{lm:lgG}, \ref{lm:invgpm} and the fact that $G$ is even (see \eqref{eq:Gdef}), we get the following equivalent assertion
	\[ \exp( G( v - u ) - G( v - u ) ) = 1, \]
	which is trivially true.
	
\end{proof}

\vspace{1cm}

\subsection{Explicit isomorphism of completions}

We prove in this subsection that homomorphism $\Phi: {U}_{\hbar}(L\mathfrak{g}) \to \widehat{Y_{\hbar}(\mathfrak{g})}$ extends to an isomorphism of completed superalgebras. Recall previous considerations about completions in Subsections \ref{sub:compl}, \ref{sect:QLS} and \ref{lb:HSSCY}. The results that have the same proof in non-super and super cases are given without proof.

\begin{definition}
	\emph{\cite{GT13}}
	\label{df:isomzerogr}
	Let $\{ e_{i}, f_{i}, h_{i}\}_{i \in I}$ be the generators of $\mathfrak{g}$. Define by $\eta: U( \mathfrak{g}[z, z^{-1}] ) \to U_{\hbar}(L \mathfrak{g}) / \hbar U_{\hbar}(L \mathfrak{g}) $ the isomorphism of superalgebras
	\[ \eta( e_{i} \otimes z^{k} ) = E_{i,k}, \; \eta( f_{i} \otimes z^{k} ) = F_{i,k}, \; \eta( h_{i} \otimes z^{k} ) = H_{i,k} \]
	for all $i \in I$ and $k \in \mathbb{Z}$.
	
	Define by $\tau:  U( \mathfrak{g}[s] ) \to Y_{\hbar}( \mathfrak{g}) / \hbar Y_{\hbar}( \mathfrak{g}) $ the isomorphism of superalgebras
	\[ \tau( e_{i} \otimes s^{r} ) = x^{+}_{i,r}, \; \tau( f_{i} \otimes s^{r} ) = x^{-}_{i,r}, \; \tau( h_{i} \otimes s^{r} ) = h_{i,r} \]
	for all $i \in I$ and $r \in \mathbb{N}_{0}$.
\end{definition}

The following result is well known.
\begin{proposition}
	\emph{\cite{GT13}}
	Let $\Phi: {U}_{\hbar}(L\mathfrak{g}) \to \widehat{Y_{\hbar}(\mathfrak{g})}$ be the homomorphism given by the Theorem \ref{th:gsercindtrue}. Then the specialization of $\Phi$ at $\hbar=0$ is the homomorphism
	\[ \exp^{*} : U( \mathfrak{g}[z, z^{-1}] ) \to U( \mathfrak{g}[[s]] ) \subset \widehat{U(\mathfrak{g}[s])} \]
	given on $\mathfrak{g}[z, z^{-1}]$ by $\exp^{*}(X \otimes z^{k}) = X \otimes e^{ks}$, where $X \in \mathfrak{g}$ and $k \in \mathbb{Z}$.
\end{proposition}
Thus we have
\begin{corollary}
	\label{cl:isspec}
	Homomorphism $\exp^{*}$ can be extended to the isomorphism
	\[ \widehat{\exp^{*}} : \widehat{U( \mathfrak{g}[z, z^{-1}] )} \to \widehat{U(\mathfrak{g}[s])}. \]
\end{corollary}
\begin{proof}
	It is sufficient to prove that the homomorphism $\exp^{*}$ has the inverse on generators of $U( \mathfrak{g}[z, z^{-1}] )$. Also note that as in $U( \mathfrak{g}[z, z^{-1}] )$ exists $z^{-1}$ it is reasonable to consider the element $\log(z) \in U( \mathfrak{g}[[z, z^{-1}]] ) $ ($U( \mathfrak{g}[[z, z^{-1}]] ) \subset \widehat{U( \mathfrak{g}[z, z^{-1}] )}$) defined in Subsection \ref{sb:ques}. Now it is easy to see that the homomorphism 
	\[\widehat{\log^{*}} : \widehat{U(\mathfrak{g}[s])} \to \widehat{U( \mathfrak{g}[z, z^{-1}] )} \]
	given on $\mathfrak{g}[s]$ by $\widehat{\log^{*}}(X \otimes s^r) = X \otimes (\log(z))^{r}$, where $X \in \mathfrak{g}$ and $r \in \mathbb{N}_{0}$. Thus
	\[ \widehat{\log^{*}} \circ \widehat{\exp^{*}} ( X \otimes z^{k} ) = \widehat{\log^{*}} ( X \otimes e^{ks} ) = X \otimes e^{k \log(z)} = X \otimes z^{k} \]
	for $X \in \mathfrak{g}$, $k \in \mathbb{Z}$ and
	\[ \widehat{\exp^{*}} \circ \widehat{\log^{*}} ( X \otimes s^r ) = \widehat{\exp^{*}} ( X \otimes (\log(z))^{r} ) = X \otimes (\log(e^{s}))^{r} = X \otimes s^r \]
	for $X \in \mathfrak{g}$, $r \in \mathbb{N}_{0}$. The result follows.
\end{proof}

Finally we need
\begin{proposition}
	\emph{\cite{GT13}}
	\label{pr:fduu}
	\begin{enumerate}
		\item $\widehat{U_{\hbar}( L \mathfrak{g} )}$ is a flat deformation of $\widehat{U( \mathfrak{g}[z, z^{-1}] )}$ over $\Bbbk[[\hbar]]$;
		\item $\widehat{Y_{\hbar}(\mathfrak{g})}$ is a flat deformation of $\widehat{U(\mathfrak{g}[s])}$ over $\Bbbk[[\hbar]]$.
	\end{enumerate}
\end{proposition}

Recall definition of the $\mathbb{Z}_{2}$-graded ideal $\mathcal{J}$ given in Subsection \ref{sect:QLS}. Now we are ready to prove
\begin{theorem}
	Let $\Phi: {U}_{\hbar}(L\mathfrak{g}) \to \widehat{Y_{\hbar}(\mathfrak{g})}$ be a homomorphism defined by the Theorem \ref{th:gsercindtrue}. Then
	\begin{enumerate}
		\item $\Phi$ maps $\mathcal{J}$ to the $\mathbb{Z}_{2}$-graded ideal $\widehat{Y_{\hbar}(\mathfrak{g})}_{+} = \prod_{n \ge 1} Y_{\hbar}(\mathfrak{g})_{n}$;
		\item the corresponding homomorphism
		\[ \widehat{\Phi} : \widehat{U_{\hbar}( L \mathfrak{g} )} \to \widehat{Y_{\hbar}(\mathfrak{g})} \]
		is an isomorphism.
	\end{enumerate}
	\label{th:cisomsyls}
\end{theorem}
\begin{proof}
	The first part is proved as in Theorem 6.2 in \cite{GT13}. The second follows from Proposition \ref{pr:fduu} and the fact that, by Corollary \ref{cl:isspec}, the specialization of $\widehat{\Phi}$ at $\hbar=0$ coincides with the isomorphism $\widehat{\exp^{*}} : \widehat{U( \mathfrak{g}[z, z^{-1}] )} \to \widehat{U(\mathfrak{g}[s])}$.
\end{proof}

\vspace{1cm}

\section{Classification of Hopf superalgebra strunctures on Drinfeld super Yangian}

Recall Subsections \ref{subs:LS}, \ref{sb:WGD} and \ref{sect:DSY} for notations. The main task is to classify Hopf superalgebra structures on $Y_{\hbar}(\mathfrak{g}_{d})$ for $d \in D$. The way we act is very close to that in \cite{MS21}, i. e. the plan is to describe all isomorphic superalgebras of the form $Y_{\hbar}(\mathfrak{g}_{d})$ ($d \in D$) and to determine which of them are isomorphic as Hopf superalgebras. The latter are classified by Dynkin diagrams. There is only one limitation that we have to take into account, namely the restictions imposed on Dynkin diagrams in Subsection \ref{sb:mpsy} (see Theorem \ref{th:mpy}) in order to consider the minimalistic presentation and to define in Subsection \ref{sub:hssy} the Hopf superalgebra structure on a Drinfeld super Yangian. The results that we use below nonetheless remain true even for those Drinfeld super Yangians which don't satisfy our restrictions on associated Dynkin diagrams if we suppose that they have the Hopf superalgebra structure given by equations \eqref{eq:comg} - \eqref{eq:comh1}.

\subsection{Isomorphisms of Drinfeld super Yangians}

We introduce a family of superalgebra isomorphisms needed to classify Hopf superalgebra structures on Drinfeld super Yangians.

Consider a Lie superalgebra $\mathfrak{g}_{d} = \mathfrak{sl}(V^{d})$, where $n_{+} = \text{dim}(V_{\bar{0}}^d)$, $n_{-} = \text{dim}(V_{\bar{1}}^d)$ for $d \in D$. Recall the result proved in \cite{T19b} (see also Subsection \ref{subs:LS} for notations), namely
\begin{theorem}
	\emph{\cite{T19b}}
	The superalgebra $Y_{\hbar}(\mathfrak{g}_{d})$ for all $d \in D$ depends only on $(n_{+}, n_{-})$ up to an isomorphism of superalgebras.
	\label{th:SYD}
\end{theorem}

Recall Subsection \ref{sb:WGD}. Denote by $S_z$ the symmetric group of degree $z$ ($z \in \mathbb{N}$). Fix the characteristic $(m, n)$. Let $Y_{\hbar}(m|n) := Y_{\hbar}(\mathfrak{g}_{a})$, where $a \in D$ corresponds to the distinguished Dynkin diagram $D_{a}$ and $\mathfrak{g}_{a} = \mathfrak{sl}(m|n)$. Describe the distinguished Dynkin diagram associated with $P_{st}(m, n)$ (which gives rise to the Drinfeld super Yangian $Y_{\hbar}(m|n)$ in the unique way through the Lie superalgebra $\mathfrak{sl}(m|n)$) via
\begin{equation}
	\label{eq:stdddcons}
	D_{a} = ( \epsilon_{1,a} - \epsilon_{2,a}, ... , \epsilon_{i-1,a} - \epsilon_{i,a} , \epsilon_{i,a} - \epsilon_{i+1,a}, \epsilon_{i+1,a} - \epsilon_{i+2,a} , ... , \epsilon_{m+n-1,a} - \epsilon_{m+n,a} ).
\end{equation}
By Theorem \ref{th:SYD} Dynkin diagrams of isomorphic Drinfeld super Yangians have the form for $b \in D$
\begin{equation}
	D_{b} = ( \epsilon_{ \sigma(1),b} - \epsilon_{ \sigma( 2 ),b}, ... , \epsilon_{ \sigma(i),b} - \epsilon_{ \sigma(i+1),b}, ... , \epsilon_{ \sigma(m + n -1),b} - \epsilon_{ \sigma(m + n),b} ),
	\label{eq:isomdsupalg}
\end{equation}
where $\sigma$ is an arbitrary element of the symmetric group $S_{m+n}$. In order to move from one Drinfeld super Yangian to another it is natural to use the complete Weyl group $W_{c}( m, n )$, i. e. using simple reflections we are able to go through superalgebras associated with Dynkin diagrams described by the formula \eqref{eq:isomdsupalg}. Recall that $W_{c}( m, n )$ is in fact the pair $(S_{m+n}, \rho: S_{m+n} \to Aut_{\Bbbk}(P_{st}(m, n)) $. Fix $i \in I$ and apply the image of $\sigma_{i} \in W_{c}( m, n ) $ to $P_{st}(m, n)$ associated with the Dynkin diagram $D_{a}$ \eqref{eq:stdddcons}. Then we get the Dynkin diagram for $z \in D$
\begin{equation}
	D_{z} = ( \epsilon_{1,z} - \epsilon_{2,z}, ... , \epsilon_{i-1,z} - \epsilon_{i+1,z} , \epsilon_{i+1,z} - \epsilon_{i,z}, \epsilon_{i,z} - \epsilon_{i+2,z} , ... , \epsilon_{m+n-1,z} - \epsilon_{m+n,z} ).
\label{eq:fdch}
\end{equation}
Consider the case $|\sigma_{i}| = \bar{0}$. Let $L_{a} = \{ i \in I \; | \; |\sigma_{i}| = \bar{0} \}$. Then 
\begin{equation}
	( \epsilon_{j,a} - \epsilon_{j+1,a}, \epsilon_{j+1,a} - \epsilon_{j+2,a} ) = ( \epsilon_{\sigma_{i}(j),a} - \epsilon_{\sigma_{i}(j+1),a}, \epsilon_{\sigma_{i}(j+1),a} - \epsilon_{\sigma_{i}(j+2),a} ) =
	\label{eq:fdch1}
\end{equation}
\[ ( \epsilon_{j,z} - \epsilon_{j+1,z}, \epsilon_{j+1,z} - \epsilon_{j+2,z} ) \]
for all $j \in I$;
\begin{equation}
	|\epsilon_{j,a} - \epsilon_{j+1,a}| = |\epsilon_{\sigma_{i}(j),a} - \epsilon_{\sigma_{i}(j+1),a}| = 	|\epsilon_{j,z} - \epsilon_{j+1,z}|	
	\label{eq:fdch2}
\end{equation}	
for all $j \in I$. Then associate with a Dynkin diagram $D_{z}$ \eqref{eq:fdch} which satisfies equations \eqref{eq:fdch1} - \eqref{eq:fdch2} a Drinfeld super Yangian $Y_{\hbar}(\mathfrak{g}_{z})$ for an appropriate $z \in D$. Recall that $Y_{\hbar}(m|n) = Y_{\hbar}(\mathfrak{g}_{a})$ for an appropriate $a \in D$. Using this information we give
\begin{definition}
	\label{df:lbhsisomeven}
	 Define a $\mathbb{Z}_{2}$-graded even mapping $\Xi_{i}: Y_{\hbar}(m|n) \to Y_{\hbar}(\mathfrak{g}_{z})$ by
	\[ \Xi_{i} ( h_{j,r,a} ) = h_{j,r,z}, \; \Xi_{i} ( x^{\pm}_{j,r,a} ) = x^{\pm}_{j,r,z} \]
	for all $j \in I$ and for all $r \in \mathbb{N}_{0}$.
\end{definition}

\begin{lemma}
	The function $\Xi_{i}$ given by Definition \ref{df:lbhsisomeven} is the isomorphism of superalgebras.
	\label{lm:srwgsaeref}
\end{lemma}
\begin{proof}
	Recall defining relations \eqref{eq:Cer} - \eqref{eq:qssr}. It follows from equations \eqref{eq:fdch1} - \eqref{eq:fdch2} that $\mathfrak{g}_{a} \cong \Xi_{i}( \mathfrak{g}_{a} ) = \mathfrak{g}_{z}$ as Lie superalgebras. In particular their associated Cartan matrices are equal. Thus it is easy to see that $\Xi_{i}$ is a homomorphism of superalgebras and is a bijective function.
\end{proof}

As all Drinfeld super Yangians described by \eqref{eq:isomdsupalg} are isomorphic as superalgebras to $Y_{\hbar}(m|n)$ it is clear that the family of superalgebra isomorphisms given by Definition \eqref{df:lbhsisomeven} can be considered as a family of automorphisms if we set $z = a$. We are also interested if there is a well-known algebraic structure behind a subring $ \langle \{ \Xi_{i} : Y_{\hbar}(m|n) \to Y_{\hbar}(m|n)  \}_{i \in L_{a}} \rangle $ of the ring of endomorphisms $End_{\Bbbk[\hbar]}(Y_{\hbar}(m|n))$. Recall that $I_{\bar{0}} = \{1,2,..., m\}$ and $I_{\bar{1}} = \{m+1,m+2,..., m+n\}$. We prove
\begin{lemma}
	\[ ( \Xi_{i} )^2 = \id \]
	for all $i \in I_{\bar{0}} \setminus \{m\}$ and $i \in  I_{\bar{1}} \setminus \{m+n\}$ respectively;
	\[ \Xi_{i} \Xi_{j} = \Xi_{j} \Xi_{i} \text{ for } |i-j| > 1 \]
	for $i, j \in I_{\bar{0}} \setminus \{m\}$ and $i, j \in  I_{\bar{1}} \setminus \{m+n\}$ respectively;
	\[ ( \Xi_{i} \Xi_{i+1} )^{3} = \id \]
	for $i \in I_{\bar{0}} \setminus \{m-1, m\}$ and $i \in  I_{\bar{1}} \setminus \{m+n-1,m+n\}$ respectively;
	\[ \Xi_{i} \Xi_{j} = \Xi_{j} \Xi_{i} \]
	for all $i \in I_{\bar{0}} \setminus \{ m\}$ and $j \in  I_{\bar{1}} \setminus \{m+n\}$.
	\label{lm:defrellbr}
\end{lemma}
\begin{proof}
	The proof follows immediately from the definition.
\end{proof}
Set $|\Xi_{i}| = \bar{0}$ for all $i \in L_{a}$. Recall Subsection \ref{sub:WGdef}, namely the definition of the Weyl group $W(m, n)$ of type $(m, n)$. We conclude that there is the faithful representation of groups (which is a $\mathbb{Z}_{2}$-graded even function) $\varrho : S_{I_{\bar{0}}} \times S_{ I_{\bar{1}} } \to \langle \{ \Xi_{i}  \}_{i \in L_{a}} \rangle$ given by $\varrho(\sigma_{i}) = \Xi_{i}$ for all $i \in L_{a}$. Recall that there is a faithful representation (which is a $\mathbb{Z}_{2}$-graded even function) $\rho: S_{I_{\bar{0}}} \times S_{ I_{\bar{1}} } \to Aut_{\Bbbk}(P_{st}(m, n))$. Thus we have
\begin{theorem}
	There is an isomorphim of $\mathbb{Z}_{2}$-graded groups (a $\mathbb{Z}_{2}$-graded even function) $\upsilon: \langle \{ \Xi_{i}  \}_{i \in L_{a}} \rangle \to  \langle \{ \rho( \sigma_{i} ) \}_{i \in L_{a}} \rangle $ given by $\upsilon( \Xi_{i} ) = \rho( \sigma_{i} )$ for all $i \in L_{a}$. Moreover, the following diagram commutes
	\begin{center}
		\begin{tikzpicture}
			\matrix (m) [matrix of math nodes,row sep=3em,column sep=4em,minimum width=2em]
			{
				S_{I_{\bar{0}}} \times S_{ I_{\bar{1}} } \\
				\langle \{ \Xi_{i}  \}_{i \in L_{a}} \rangle & \langle \{ \rho( \sigma_{i} ) \}_{i \in L_{a}} \rangle \\};
			\path[-stealth]
			(m-1-1) edge node [left] {$\varrho$} (m-2-1)
			edge node [right] {$\rho$} (m-2-2)
			(m-2-1) edge node [below] {$\upsilon$} (m-2-2);
		\end{tikzpicture}
	\end{center}
\end{theorem}
\begin{proof}
	The result follows from the definition of $W( m, n )$, the arguments in the preceding paragraph and Lemma \ref{lm:defrellbr}.
\end{proof}

Now we want to find out which elements of $W_{c}(m,n)$ move the odd vertix (recall there is exact one such vertix) labeled by $\epsilon_{m,a} - \epsilon_{m+1,a}$ in $D_{a}$ \eqref{eq:stdddcons} in such a way that in the Dynkin diagram $D_{b}$ \eqref{eq:isomdsupalg} there is exactly one odd vertix. Let $\sigma \in W_{c}(m,n)$ be such a permutation and suppose that $\sigma(\epsilon_{m,a} - \epsilon_{m+1,a} ) = \epsilon_{\sigma(m),a} - \epsilon_{\sigma(m+1),a} = \epsilon_{j,b} - \epsilon_{j+1,b}$ for $j \in I$, where $|\epsilon_{m,a} - \epsilon_{m+1,a}| = |\epsilon_{\sigma(m),a} - \epsilon_{\sigma(m+1),a} | = |\epsilon_{j,b} - \epsilon_{j+1,b} | = \bar{1}$. There are the following possibilities:
\begin{enumerate}
	\item $|\epsilon_{j-1,b}| = |\epsilon_{m,a}|$ and $ |\epsilon_{j+2,b}| = |\epsilon_{m+1,a}|$;
	\item $|\epsilon_{j-1,b}| = |\epsilon_{m+1,a}|$ and $|\epsilon_{j+2,b}| = |\epsilon_{m,a}|$.
\end{enumerate}
In our particular case it is easy to see that we have
\begin{enumerate}
	\item $\sigma = \sigma^{'}$ for $\sigma^{'} \in W(a)$ (1-st case);
	\item $ \sigma(\epsilon_{i,a}) = \sigma^{'} \circ \sigma^{''}(\epsilon_{i,a}) = \sigma^{'} ( \epsilon_{m+n+1-i,a} ) $ for all $i \in I$ and for $\sigma^{'} \in W(b)$ (2-nd case).
\end{enumerate}
The latter second case motivates
\begin{definition}
	\label{df:lbhsisomchorder}
	Define a $\mathbb{Z}_{2}$-graded even mapping $\Xi: Y_{\hbar}(m|n) \to Y_{\hbar}(\mathfrak{g}_{b})$ by
	\[ \Xi ( h_{j,r,a} ) = h_{m+n+1-j,r,b}, \; \Xi ( x^{\pm}_{j,r,a} ) = x^{\pm}_{m+n+1-j,r,b} \]
	for all $j \in I$ and for all $r \in \mathbb{N}_{0}$.
\end{definition}
It is obvious that definition is well-defined.

\begin{lemma}
	The function $\Xi$ given by Definition \ref{df:lbhsisomchorder} is the isomorphism of superalgebras.
	\label{lm:lbhsisomchorder}
\end{lemma}
\begin{proof}
	It follows from the definition of the $\sigma \in W_{c}(m,n)$ that for elements of Cartan matrices $c_{ij}^{a} = c_{m+n+1-i,m+n+1-j}^{b}$ for all $i,j \in I$. The result follows immediately from defining relations \eqref{eq:Cer} - \eqref{eq:qssr}.
\end{proof}

We move to the general case. For a fixed $d_1 \in D$ consider the Drinfeld super Yangian $Y_{\hbar}( \mathfrak{g}_{d_1})$ with Dynkin diagram $D_{d_1}$ and its complete Weyl group $W_{c}(d_1)$ (see Remark \ref{rm:cwgfadd}). We introduce a family of isomorphisms of superagebras which are used in the next Subsection \ref{sub:ihssdsy}. We associate with an arbitrary $\sigma \in W_{c}(d_1)$ a corresponding Dynkin diagram $D_{d_2}$ ($d_2 \in D$) induced by a set of simple roots $\Pi_{d_2} = \{ \sigma(\alpha) \; | \; \alpha \in \Pi_{d_1} \}$.

Let $\sigma_{i} \in W_{c}(d_1)$  be such that $|\sigma_{i}| = \bar{0}$ for $i \in I$. Then by analogy with Definition \ref{df:lbhsisomeven} we give
\begin{definition}
	\label{df:gencaseven}
	Define a $\mathbb{Z}_{2}$-graded even mapping $\Xi_{i}: Y_{\hbar}(\mathfrak{g}_{d_1}) \to Y_{\hbar}(\mathfrak{g}_{d_2})$ by
	\[ \Xi_{i} ( h_{j,r,d_1} ) = h_{j,r,d_2}, \; \Xi_{i} ( x^{\pm}_{j,r,d_1} ) = x^{\pm}_{j,r,d_2} \]
	for all $j \in I$ and for all $r \in \mathbb{N}_{0}$.
\end{definition}

\begin{lemma}
	The function $\Xi_{i}$ given by Definition \ref{df:gencaseven} is the isomorphism of superalgebras.
	\label{lm:evenisomsay}
\end{lemma}
\begin{proof}
	The proof is similar to that given in Lemma \ref{lm:srwgsaeref}.
\end{proof}

Now we want to find out which elements of $W_{c}(d_1)$ move an arbitrary odd vertix associated with $\epsilon_{i,d_1} - \epsilon_{i+1,d_1}$ ($i \in I$) in $D_{d_1}$ in such a way that in the Dynkin diagram $D_{d_2}$ there is the same number of odd vertices as in $D_{d_1}$. Let $\sigma \in W_{c}(d_1)$ be such a permutation and suppose that $\sigma(\epsilon_{i,d_1} - \epsilon_{i+1,d_1} ) = \epsilon_{j,d_2} - \epsilon_{j+1,d_2}$ for $j \in I$, where $|\epsilon_{i,d_1} - \epsilon_{i+1,d_1}| = |\epsilon_{j,d_2} - \epsilon_{j+1,d_2}| = \bar{1}$. Let $I_{\bar{0}} = \{ i \; | \; |\epsilon_{i,d_1}| = \bar{0} \}$ and $I_{\bar{1}} = \{ i \; | \; |\epsilon_{i,d_1}| = \bar{1} \}$. There are the following possibilities:
\begin{enumerate}
	\item $|\epsilon_{j-1,d_2}| = |\epsilon_{i,d_1}|$ and $ |\epsilon_{j+2,d_2}| = |\epsilon_{i+1,d_1}|$;
	\item $|\epsilon_{j-1,d_2}| = |\epsilon_{i+1,d_1}|$ and $|\epsilon_{j+2,d_2}| = |\epsilon_{i,d_1}|$.
\end{enumerate}
It is easy to see that we have
\begin{enumerate}
	\item $\sigma = \sigma^{'}$ for $\sigma^{'} \in W(d_1)$ (1-st case);
	\item $ \sigma(\epsilon_{k,d_1}) =  \epsilon_{m+n+1-\lambda(k),d_1} $ for all $k \in I$ and a bijective function $\lambda: I \to I $ such that $\lambda(I_{\bar{0}}) = I_{\bar{1}}$ (it is a necessary condition that $m=n$) (1-st case);
	\item $ \sigma(\epsilon_{k,d_1}) = \sigma^{'} \circ \sigma^{''}(\epsilon_{k,d_1}) = \sigma^{'} ( \epsilon_{m+n+1-k,d_1} ) $ for all $k \in I$ and for $\sigma^{'} \in W(d_2)$ (2-nd case);
	\item $ \sigma(\epsilon_{k,d_1}) =  \epsilon_{\lambda(k),d_1} $ for all $k \in I$ and a bijective function $\lambda: I \to I $ such that $\lambda(I_{\bar{0}}) = I_{\bar{1}}$ (it is a necessary condition that $m=n$) (2-nd case).
\end{enumerate}

The latter second case motivates
\begin{definition}
	\label{df:fcod}
	Define a $\mathbb{Z}_{2}$-graded even mapping $\Xi^{''}: Y_{\hbar}(\mathfrak{g}_{d_1}) \to Y_{\hbar}(\mathfrak{g}_{d_2})$ by
	\[ \Xi^{''} ( h_{j,r,d_1} ) = h_{m+n+1-j,r,d_2}, \; \Xi^{''} ( x^{+}_{j,r,d_1} ) = (-1)^{1+|\alpha_{j,d_1}|} x^{-}_{m+n+1-j,r,d_2}, \; \Xi^{''} ( x^{-}_{j,r,d_1} ) =  x^{+}_{m+n+1-j,r,d_2} \]
	for all $j \in I$ and for all $r \in \mathbb{N}_{0}$.
\end{definition}

\begin{lemma}
	The function $\Xi^{''}$ given by Definition \ref{df:fcod} is the isomorphism of superalgebras.
	\label{lm:fcod}
\end{lemma}
\begin{proof}
	It follows from the definition of the $\sigma \in W_{c}(m,n)$ that for elements of Cartan matrices $c_{ij}^{d_1} = - c_{m+n+1-i,m+n+1-j}^{d_2}$ for all $i,j \in I$. The result follows immediately from defining relations \eqref{eq:Cer} - \eqref{eq:qssr}.
\end{proof}

The latter third case motivates
\begin{definition}
	\label{df:lbhsisomchorderg}
    Define a $\mathbb{Z}_{2}$-graded even mapping $\Xi: Y_{\hbar}(\mathfrak{g}_{d_1}) \to Y_{\hbar}(\mathfrak{g}_{d_2})$ by
	\[ \Xi ( h_{j,r,d_1} ) = h_{m+n+1-j,r,d_2}, \; \Xi ( x^{\pm}_{j,r,d_1} ) = x^{\pm}_{m+n+1-j,r,d_2} \]
	for all $j \in I$ and for all $r \in \mathbb{N}_{0}$.
\end{definition}

\begin{lemma}
	The function $\Xi$ given by Definition \ref{df:lbhsisomchorderg} is the isomorphism of superalgebras.
	\label{lm:lbhsisomchorderg}
\end{lemma}
\begin{proof}
	The proof is similar to that given in Lemma \ref{lm:lbhsisomchorder}.
\end{proof}

The latter fourth case motivates
\begin{definition}
	\label{df:tcod}
	Define a $\mathbb{Z}_{2}$-graded even mapping $\Xi^{'}: Y_{\hbar}(\mathfrak{g}_{d_1}) \to Y_{\hbar}(\mathfrak{g}_{d_2})$ by
	\[ \Xi^{'} ( h_{j,r,d_1} ) = h_{j,r,d_2}, \; \Xi^{'} ( x^{+}_{j,r,d_1} ) = (-1)^{1+|\alpha_{j,d_1}|} x^{-}_{j,r,d_2}, \; \Xi^{'} ( x^{-}_{j,r,d_1} ) = x^{+}_{j,r,d_2} \]
	for all $j \in I$ and for all $r \in \mathbb{N}_{0}$.
\end{definition}

\begin{lemma}
	The function $\Xi^{'}$ given by Definition \ref{df:tcod} is the isomorphism of superalgebras.
	\label{lm:tcod}
\end{lemma}
\begin{proof}
	It follows from the definition of the $\sigma \in W_{c}(m,n)$ that for elements of Cartan matrices $c_{ij}^{d_1} = - c_{ij}^{d_2}$ for all $i,j \in I$. The result follows immediately from defining relations \eqref{eq:Cer} - \eqref{eq:qssr}.
\end{proof}

\vspace{1cm}

\subsection{Isomorphisms of Hopf superalgebra structures on Drinfeld super Yangians}

\label{sub:ihssdsy}

We suppose that all Drinfeld super Yangians that we consider in this section have the Hopf superalgebra structure described in Subsection \ref{sub:hssy}. Let $\underline{\text{sHAlg}}$ be the supercategory of Hopf superalgebras over the field $\Bbbk[\hbar]$ (see \cite{MS21} for more information and references). Denote by $\underline{\text{sAlg}}$ the supercategory of superalgebras over the field $\Bbbk[\hbar]$. By $\underline{\text{sLieAlg}}$ we denote the supercategory of Lie superalgebras over the field $\Bbbk[\hbar]$. Recall that $Y_{\hbar}(\mathfrak{g}_{d})$ is $\mathbb{N}_{0}$-graded, see Subsection \ref{lb:HSSCY}. We classify Hopf superalgebra structures on Drinfeld super Yangians using Dynkin diagrams. In order to do this we prove some auxiliary results about properties of Hopf superalebgra isomorphisms.

\begin{lemma}
	Any isomorphism of Drinfeld super Yangians considered as Hopf superalgebras preserves $\mathbb{N}_{0}$-grading. 
	\label{lm:zetangr}
\end{lemma}
\begin{proof}	
	Let $\zeta \in  \text{Hom}_{\underline{\text{sHAlg}}}(Y_{\hbar}(\mathfrak{g}_{d_1}), Y_{\hbar}(\mathfrak{g}_{d_2}))$ be an isomorphism. Consider the equation \eqref{eq:SYrc1} and suppose that there are $i \ne j$ such that $c_{ij}^{d} \ne 0$ (this is always possible as $|I| > 1$). If we apply $\zeta$ to the latter formula we will get that $\deg(\zeta(h_{i0})) + \deg( \zeta(x_{j,s,d}^{\pm}) ) = \deg( \zeta(x_{j,s,d}^{\pm}) ) $. Thus $\deg(\zeta(h_{i0})) = 0$ for all $i \in I$. Consider the equation \eqref{eq:SYrc3}. It follows that for all $i \in I$ we have
	\[ \deg( \zeta(x_{i0}^{+}) ) + \deg( \zeta(x_{i0}^{-}) ) = \deg( \zeta( h_{i0} ) ) = 0 \iff \deg( \zeta(x_{i0}^{+}) ) = - \deg( \zeta(x_{i0}^{-}) ) \iff \]
	\[ \deg( \zeta(x_{i0}^{+}) ) = \deg( \zeta(x_{i0}^{-}) ) = 0. \]
	Therefore to conclude that $U(\mathfrak{g}_{d_1}) \overset{\zeta}{\cong} U(\mathfrak{g}_{d_2})$. Consider the formula \eqref{eq:SYrc3} and set $s=0$. Then we get that $\deg(\zeta(x_{i,r}^{+})) = \deg(\zeta(h_{i,r}))$ for all $i \in I$ and $r \in \mathbb{N}_{0}$. In the same way $\deg(\zeta(x_{i,r}^{-})) = \deg(\zeta(h_{i,r}))$ for all $i \in I$ and $r \in \mathbb{N}_{0}$. Moreover, if we set $s = 1$ we have 
	\[ \deg(\zeta(x_{i,r}^{+})) + \deg(\zeta(x_{i,1}^{-})) = \deg(\zeta(h_{i,r+1})). \]
	Thus by the mathematical induction on $r$ we have
	\[ \deg(\zeta(h_{i,r})) = \deg(\zeta(x_{i,r}^{\pm})) = r \deg(\zeta(h_{i,1})) \]
	for all $i \in I$ and $r \in \mathbb{N}_{0}$.
	
	Suppose $i \in I$ is such that $|\alpha_{i}| = \bar{0}$. Consider the equation \eqref{eq:SYrc4} and set $r=s=0$. Then it follows that 
	\[ \deg( \zeta( x_{i,1}^{\pm} ) ) = \deg(\hbar) = 1. \]
	Consequently $\deg( \zeta( x_{i,r}^{\pm} ) ) = \deg( \zeta( h_{i,r} ) ) = r$ for all $r \in \mathbb{N}_{0}$.
	
	Suppose $i \in I$ is such that $|\alpha_{i}| = \bar{1}$. Consider the equation \eqref{eq:mpy4}. Suppose that there is $j \ne i$, $c_{ij}^{d} \ne 0$ and $|\alpha_{j}| = \bar{0}$ (this is always possible as of our restrictions on Dynkin diagrams). Then $\deg( \zeta( \widetilde{h}_{i1} ) ) = \deg( \zeta( h_{i1} ) ) = \deg( \zeta( x_{j1}^{\pm} ) ) = 1$. Consequently $\deg( \zeta( x_{i,r}^{\pm} ) ) = \deg( \zeta( h_{i,r} ) ) = r$ for all $r \in \mathbb{N}_{0}$. The result follows.
\end{proof}

\begin{remark}
	\item It follows from equations \eqref{eq:comg} - \eqref{eq:comh1}, \eqref{eq:comrec1} - \eqref{eq:comrec2} that $\Delta$ preserves $\mathbb{N}_{0}$-grading, i. e. $\deg(\Delta (x)) = \deg(x)$. 
	\item To prove Lemma \ref{lm:zetangr} we use our assumptions on Dynkin diagrams (see the paragraph before Theorem \ref{th:mpy}).
\end{remark}

Consider the Hopf superalgebra $U(\mathfrak{g}_d) \xhookrightarrow{} Y_{\hbar}(\mathfrak{g}_{d})$. Let $P_{1,1}(U(\mathfrak{g}_d))$ be the $\Bbbk$ super vector space of primitive elements in $U(\mathfrak{g}_d)$, i. e. for any $x \in P_{1,1}(U(\mathfrak{g}_d))$ we have
\[ \Delta_{U(\mathfrak{g}_d)}(x) = x \otimes 1 + 1 \otimes x. \]
By PBW Theorem \ref{th:pbwyang} and according to the Friedrich's theorem \cite{J62} (which also holds in the super case) it follows that
\begin{equation}
	P_{1,1}(U(\mathfrak{g}_d)) = \sum_{\substack{\alpha, \beta \in \Delta_{d}, \\ i \in I}} ( \Bbbk e_{\alpha,0} + \Bbbk h_{i,0} + \Bbbk f_{\beta,0}) = \mathfrak{g}_d.
	\label{eq:prues}
\end{equation}
As $U(\mathfrak{g}_{d_1}) \overset{\zeta\vert_{U(\mathfrak{g}_{d_1})}}{\cong} U(\mathfrak{g}_{d_2})$ and 
\[ \Delta_{d_2} \circ \zeta (x) =  \zeta( x ) \otimes 1 + 1 \otimes \zeta( x ), \]
we have $P_{1,1}(U(\mathfrak{g}_{d_1})) \overset{\zeta\vert_{P_{1,1}(U(\mathfrak{g}_{d_1}))}}{\cong} P_{1,1}(U(\mathfrak{g}_{d_2}))$, i. e. $\mathfrak{g}_{d_1} \overset{\zeta\vert_{\mathfrak{g}_{d_1}}}{\cong} \mathfrak{g}_{d_2}$. Thus if we take any isomorphism of superalgebras $\gamma \in \text{Hom}_{\underline{\text{sAlg}}}(Y_{\hbar}(\mathfrak{g}_{d_1}), Y_{\hbar}(\mathfrak{g}_{d_2}))$ and consider the restriction 
\[\gamma\vert_{U(\mathfrak{g}_{d_1})} \in \text{Hom}_{\underline{\text{sAlg}}}(U(\mathfrak{g}_{d_1}), U(\mathfrak{g}_{d_2})), \]
we automatically get the isomorphism of Hopf superalgebras, i. e.
\[\gamma\vert_{U(\mathfrak{g}_{d_1})} \in \text{Hom}_{\underline{\text{sHAlg}}}(U(\mathfrak{g}_{d_1}), U(\mathfrak{g}_{d_2})) \]
and the isomorphism of Lie superalgebras
\[ \gamma\vert_{\mathfrak{g}_{d_1}} \in \text{Hom}_{\underline{\text{sLieAlg}}}(\mathfrak{g}_{d_1}, \mathfrak{g}_{d_2}). \]

Now we are ready to prove
\begin{lemma}
	Drinfeld super Yangians are not isomorphic as Hopf superalgebras if their associated Dynkin diagrams have different number of even (odd) vertices.
	\label{lm:crihsbp}
\end{lemma}
\begin{proof}
	Let $\zeta \in  \text{Hom}_{\underline{\text{sHAlg}}}(Y_{\hbar}(\mathfrak{g}_{d_1}), Y_{\hbar}(\mathfrak{g}_{d_2}))$ be an isomorphism. Recall the equations \eqref{eq:comh1}, \eqref{eq:comxbp} and \eqref{eq:comxbm}. It follows from Lemma \ref{lm:zetangr} that
	\[ \zeta( x_{z1,d_1}^{+} ) = \sum_{j \in I} a_{zj} h_{j1,d_2} + \sum_{j \in I} a^{+}_{zj} x^{+}_{j1,d_2} + \sum_{j \in I} a^{-}_{zj} x^{-}_{j1,d_2} + \hbar a_{z}, \]
	where	
	\[ a_{zj}, a_{zj}^{\pm}, a_{z} \in  U( \mathfrak{g}_{d_2} ), \]
	\[ a_{zj} = \sum_{t \in \mathbb{N}} (a_{zj})_{t}, \; a_{zj}^{\pm} = \sum_{t \in \mathbb{N}} (a_{zj}^{\pm})_{t}, \]
	\[ (a_{zj})_{0}, (a_{zj}^{\pm})_{0} \in \Bbbk, \; (a_{zj})_{t}, (a_{zj}^{\pm})_{t} \notin \Bbbk \text{ for all } t > 0. \]
	
	As $\zeta$ is an isomorphism of Hopf superalgebras it must satisfy the equation
	\[ \zeta \otimes \zeta \Delta_{d_1} (x_{z1,d_1}^{+}) = \Delta_{d_2} \zeta (x_{z1,d_1}^{+}). \]
	Explore the left part
	\[ \zeta \otimes \zeta \Delta_{d_1} (x_{z1,d_1}^{+}) =  \]
	\[ \square( \sum_{j \in I} a_{zj} h_{j1,d_2} + \sum_{j \in I} a^{+}_{zj} x^{+}_{j1,d_2} + \sum_{j \in I} a^{-}_{zj} x^{-}_{j1,d_2} + \hbar a_{z} ) - \hbar \zeta \otimes \zeta ( [ 1 \otimes x_{z0,d_1}^{+} , \Omega_{d_1}^{+} ] ) =  \]
	\[ \square( \sum_{j \in I} a_{zj} h_{j1,d_2} + \sum_{j \in I} a^{+}_{zj} x^{+}_{j1,d_2} + \sum_{j \in I} a^{-}_{zj} x^{-}_{j1,d_2} ) + \hbar ( \square( a_{z} ) - \zeta \otimes \zeta ( [ 1 \otimes x_{z0,d_1}^{+} , \Omega_{d_1}^{+} ] ) );  \]
	Explore the right part
	\[ \Delta_{d_2} \zeta (x_{z1,d_1}^{+}) = \]
	\[ \Delta_{d_2} ( \sum_{j \in I} a_{zj} h_{j1,d_2} + \sum_{j \in I} a^{+}_{zj} x^{+}_{j1,d_2} + \sum_{j \in I} a^{-}_{zj} x^{-}_{j1,d_2} + \hbar a_{z} ) = \]
	\[ \sum_{j \in I} \Delta_{d_2}( a_{zj} ) ( \square(h_{j1,d_2}) + \hbar( h_{j0,d_2} \otimes h_{j0,d_2} - \sum_{ \alpha \in \Delta_{d_2}^{+}} (\alpha_{j}, \alpha) x_{\alpha}^{-} \otimes x_{\alpha}^{+} ) ) + \]
	\[ \sum_{j \in I} \Delta_{d_2}( a^{+}_{zj} ) ( \square( x_{j1,d_2}^{+} ) - \hbar [ 1 \otimes x_{j0,d_2}^{+} , \Omega_{d_2}^{+} ] ) + \]
	\[ \sum_{j \in I} \Delta_{d_2}( a^{-}_{zj} ) ( \square( x_{j1,d_2}^{-} ) + \hbar [ x_{j0,d_2}^{-} \otimes 1 , \Omega_{d_2}^{+} ] ) + \hbar \Delta_{d_2} ( a_{z} ) = \]
	\[ \square( \sum_{j \in I} (a_{zj})_{0} h_{j1,d_2} + \sum_{j \in I} (a^{+}_{zj})_{0} x^{+}_{j1,d_2} + \sum_{j \in I} (a^{-}_{zj})_{0} x^{-}_{j1,d_2} ) + \]
	\[ \square( \sum_{j \in I, t > 0} (a_{zj})_{t} h_{j1,d_2} + \sum_{j \in I, t > 0} (a^{+}_{zj})_{t} x^{+}_{j1,d_2} + \sum_{j \in I, t > 0} (a^{-}_{zj})_{t} x^{-}_{j1,d_2} ) + \]
	\[ \sum_{j \in I, t>0} ( (a_{zj})_{t} \otimes h_{j1,d_2} + h_{j1,d_2} \otimes (a_{zj})_{t} )  + \sum_{j \in I, t>0} ( (a^{+}_{ij})_{t} \otimes x^{+}_{j1,d_2} + (-1)^{|\alpha_{j}| |(a^{+}_{ij})_{t}|} x^{+}_{j1,d_2} \otimes (a^{+}_{ij})_{t} ) + \]
	\[ \sum_{j \in I, t>0} ( (a^{-}_{ij})_{t} \otimes x^{-}_{j1,d_2} + (-1)^{|\alpha_{j}| |(a^{-}_{ij})_{t}|} x^{-}_{j1,d_2} \otimes (a^{-}_{ij})_{t} ) + \]
	\[ \hbar ( \sum_{j \in I} \Delta_{d_2}( a_{zj} ) ( h_{j0,d_2} \otimes h_{j0,d_2} - \sum_{ \alpha \in \Delta_{d_2}^{+}} (\alpha_{j}, \alpha) x_{\alpha}^{-} \otimes x_{\alpha}^{+} ) ) + \]
	\[ \sum_{j \in I} (-1) \Delta_{d_2}( a^{+}_{zj} )  [ 1 \otimes x_{j0,d_2}^{+} , \Omega_{d_2}^{+} ] + \sum_{j \in I} \Delta_{d_2}( a^{-}_{zj} ) [ x_{j0,d_2}^{-} \otimes 1 , \Omega_{d_2}^{+} ] + \Delta_{d_2} ( a_{z} ) ). \]
	
	It is easy to see that
	\[ \square( a_{z} ) = \Delta_{d_2} ( a_{z} ) \Rightarrow a_{z} \in \mathfrak{g}_{d_2}, \; a_{zj}, a_{zj}^{\pm} \in \{ 0, 1 \} \]
	for all $j \in I$. Therefore we get
	\[ \zeta \otimes \zeta ( [ 1 \otimes x_{z0,d_1}^{+} , \Omega_{d_1}^{+} ] ) = \sum_{j \in I} \Delta_{d_2}( a_{zj} ) ( - h_{j0,d_2} \otimes h_{j0,d_2} + \sum_{ \alpha \in \Delta_{d_2}^{+}} (\alpha_{j}, \alpha) x_{\alpha}^{-} \otimes x_{\alpha}^{+} ) + \]
	\[ \sum_{j \in I} [ a^{+}_{zj} 1 \otimes x_{j0,d_2}^{+} - a^{-}_{zj} x_{j0,d_2}^{-} \otimes 1 , \Omega_{d_2}^{+} ]. \]
	
	Now we suppose that $|x_{z1,d_1}^{+}| = \bar{1}$. Thus we have $a_{zj} = 0$ as $|h_{j1,d_2}| = \bar{0}$ and by PBW Theorem \ref{th:pbwyang} for all $j \in I$. The last equation reduces to
	\[ \zeta \otimes \zeta ( [ 1 \otimes x_{z0,d_1}^{+} , \Omega_{d_1}^{+} ] ) = \sum_{j \in I} [ a^{+}_{zj} 1 \otimes x_{j0,d_2}^{+} - a^{-}_{zj} x_{j0,d_2}^{-} \otimes 1 , \Omega_{d_2}^{+} ]. \]
	Recall that $|\Omega_{d_1}^{+}| = |\Omega_{d_2}^{+}| = \bar{0}$. Note that $|\alpha_{j,d_1}| = |x_{z0,d_1}^{+}| = |x_{j0,d_2}^{\pm}| = \bar{1}$ when coefficients $a_{zj}^{\pm}$ are non-zero for $j \in I$. Thus as $\zeta$ is a $\mathbb{Z}_{2}$-graded even mapping we must have
	\[ \zeta \otimes \zeta ( \Omega_{d_1}^{+} ) = \Omega_{d_2}^{+}. \]
	In this way
	\[ 1 \otimes \zeta( x_{z0,d_1}^{+} ) = \sum_{j \in I} ( a^{+}_{zj} 1 \otimes x_{j0,d_2}^{+} - a^{-}_{zj} x_{j0,d_2}^{-} \otimes 1 ). \]
	Thus we have $a^{-}_{zj} = 0$ for all $j \in I$. It follows that
	\[ \zeta( x_{z0,d_1}^{+} ) = \sum_{j \in I} a^{+}_{zj} x_{j0,d_2}^{+}. \]
	
	Suppose that Dynkin diagrams associated with $Y_{\hbar}(\mathfrak{g}_{d_1})$ and $Y_{\hbar}(\mathfrak{g}_{d_2})$ respectively have different number of odd vertices. We assume without loss of generality that the number of odd vertices in the Dynkin diagram $D_{d_1}$ is strictly greater than the number of odd vertices in $D_{d_2}$. Then the elements $\zeta( x_{z0,d_1}^{+} )$ such that $|\alpha_{z,d_1}| = \bar{1}$ for $z \in I$ by PBW Theorem \ref{th:pbwyang} are linear dependent. As we suppose that $\zeta$ is an isomorphism of Hopf superalgebras we get a contradiction. The result follows.
\end{proof}

\begin{remark}
	Notice that Lemma \ref{lm:crihsbp} remains true for those Drinfeld super Yangians that have the Hopf superalgebra structure described in Subsection \ref{sub:hssy}. We don't use the minimalistic presentation in the proof.
\end{remark}

Recall Subsections \ref{sub:WGdef}. Consider a Drinfeld super Yangian $Y_{\hbar}(\mathfrak{g}_{d_1})$. Let $I_{\bar{0}} = \{ i \; | \; |\epsilon_{i,d_1}| = \bar{0} \}$ and $I_{\bar{1}} = \{ i \; | \; |\epsilon_{i,d_1}| = \bar{1} \}$. Finally we prove
\begin{theorem}
	Consider the Dynkin diagram of the Hopf superalgebra $Y_{\hbar}(\mathfrak{g}_{d_1})$ (for some $d_1 \in D$)
	\[ D_{d_1} = ( \epsilon_{1,d_1} - \epsilon_{2,d_1}, \epsilon_{2,d_1} - \epsilon_{3,d_1}, ..., \epsilon_{m+n-1,d_1} - \epsilon_{m+n,d_1} ). \]	
	Then Dynkin diagrams of isomorphic Hopf superlagebras of the type $Y_{\hbar}(\mathfrak{g}_{d_2})$ ($d_2 \in D$) have the form
	\begin{enumerate}
		\item
		\[ ( s ( \epsilon_{1,d_1} - \epsilon_{2,d_1} ), s (\epsilon_{2,d_1} - \epsilon_{3,d_1}), ..., s  (\epsilon_{m+n-1,d_1} - \epsilon_{m+n,d_1} ) ); \]
		\label{en:dd1}
		\item
		\[ ( s ( \epsilon_{m+n,d_1} - \epsilon_{m+n-1,d_1} ), ..., s (\epsilon_{3,d_1} - \epsilon_{2,d_1}), s (\epsilon_{2,d_1} - \epsilon_{1,d_1} ) ); \]
		\item
		\[ ( \lambda ( \epsilon_{1,d_1}  -  \epsilon_{2,d_1} ), \lambda ( \epsilon_{2,d_1} - \epsilon_{3,d_1} ), ..., \lambda ( \epsilon_{m+n-1,d_1} - \epsilon_{m+n,d_1} ) ), \]
		if $m=n$ and all vertices in $D_{d_1}$ are odd;
		\item
		\[ ( \lambda ( \epsilon_{m+n,d_1}  - \epsilon_{m+n-1,d_1}  ), ..., \lambda ( \epsilon_{3,d_1}  - \epsilon_{2,d_1}  ), \lambda ( \epsilon_{2,d_1}  - \epsilon_{1,d_1}  ) ), \]
		if $m=n$ and all vertices in $D_{d_1}$ are odd;
		\label{en:dd4}
	\end{enumerate}
	for all $s \in W(d_1)$ and all bijective functions $\lambda: I \to I$ such that $\lambda(I_{\bar{0}}) = I_{\bar{1}}$.
	\label{th:chsdy}
\end{theorem}
\begin{proof}
  	Consider the set $B$ of all Drinfeld super Yangians of the type $Y_{\hbar}(\mathfrak{g}_{d_2})$ ($d_2 \in D$) with the same number of even (odd) verices in associated Dynkin diagrams as in $D_{d_1}$. It follows from Lemma \ref{lm:crihsbp} that all we have is to find out which elements of $B$ are isomorhic to $Y_{\hbar}(\mathfrak{g}_{d_1})$ as Hopf superalgebras. We have already described which elements of $B$ are isomorphic as superalgebras to $Y_{\hbar}(\mathfrak{g}_{d_1})$ by Definitions \ref{df:gencaseven}, \ref{df:fcod}, \ref{df:lbhsisomchorderg}, \ref{df:tcod} and Lemmas \ref{lm:evenisomsay}, \ref{lm:fcod}, \ref{lm:lbhsisomchorderg}, \ref{lm:tcod}. It is easy to see that isomorphisms given by Definitions \ref{df:gencaseven} and \ref{df:lbhsisomchorderg} are indeed the Hopf superalgebra isomorphisms by direct computations (verify the compatibility with the equations \eqref{eq:comg} - \eqref{eq:comh1}, \eqref{eq:comrec1} - \eqref{eq:comrec2}). Isomorhisms given by Definitions \ref{df:fcod} and \ref{df:tcod} are also isomorphisms of Hopf superalgebra if we assume that all vertices in $D_{d_1}$ are odd. The latter requirement is motivated by the equation \eqref{eq:SYrc3}, i. e. isomorphisms are compatible with this equation if all vertices in $D_{d_1}$ are odd.
\end{proof}

\begin{remark}
	We describe how we can apply Theorem \ref{th:chsdy} about classification of Drinfeld super Yangians considered as Hopf superalgebras to quantum loop superalgebras. It follows from Proposition \ref{pr:homcomcY} that Theorem \ref{th:chsdy} remains true for completions of Drinfeld super Yangians $\widehat{Y_{\hbar}(\mathfrak{g}_{d})}$ for $d \in D$. From Theorem \ref{th:cisomsyls} it follows that we have the explicit isomorphisms of associative superalgebras 
	\[ \widehat{\Phi}_{d} : \widehat{U_{\hbar}( L \mathfrak{g}_{d} )} \to \widehat{Y_{\hbar}(\mathfrak{g}_{d})} \]
	for $d \in D$. Then $\widehat{\Phi}_{d}$ ($d \in D$) becomes an isomorphism of Hopf superalgebras if we induce the Hopf superalgebra structure on the completion of quantum loop superalgebra using the completion of super Yangian. In this case, we can consider the isomorphism $\widehat{\Phi}_{d}$ as a twist of type 2 (see for more information \cite[Definition 4.5, Proposition 4.4]{MS21}). Thus Theorem \ref{th:chsdy} remains true for completions of quantum loop superalgebras $\widehat{U_{\hbar}( L \mathfrak{g}_{d} )}$ for $d \in D$.
\end{remark}

\end{document}